\documentclass[12pt]{amsart}
\usepackage{xspace,amssymb,amsfonts,epsfig,marvosym,euscript,eufrak,xypic,enumerate,amsmath,nicefrac,mathrsfs}
\usepackage{tikz,etex,bbm,xspace,hyperref}
 \usepackage{epstopdf,ifpdf}
\usepackage{multicol,microtype}
\usepackage[margin=3.1cm,footskip=25pt,headheight=20pt]{geometry}
\usepackage{pxfonts,fancyhdr}
\ifpdf
  \usepackage{pdfsync}
\fi

\begin{document}
\newtheoremstyle{all}
  {11pt}
  {11pt}
  {\slshape}
  {}
  {\bfseries}
  {}
  {.5em}
  {}

\theoremstyle{all}
\newtheorem{thm}{Theorem}[section]
\newtheorem{prop}[thm]{Proposition}
\newtheorem{cor}[thm]{Corollary}
\newtheorem{lemma}[thm]{Lemma}
\newtheorem{defn}[thm]{Definition}
\newtheorem{ques}[thm]{Question}
\newtheorem{conj}[thm]{Conjecture}
\newtheorem{hypothesis}[thm]{Hypothesis}
\newtheorem{rem}[thm]{Remark}

\newcommand{\nc}{\newcommand}
\newcommand{\renc}{\renewcommand}
  \nc{\kac}{\kappa^C}
\nc{\Lco}{L_{\la}}
\nc{\qD}{q^{\nicefrac 1D}}
\nc{\ocL}{M_{\la}}
\nc{\excise}[1]{}
\nc{\mpmod}{\operatorname{-pmod}}
\newcommand{\arxiv}[1]{\href{http://arxiv.org/abs/#1}{\tt arXiv:\nolinkurl{#1}}}

\nc{\Dbe}{D_{\mathsf{gf}}}
\nc{\tr}{\operatorname{tr}}
\nc{\tla}{\mathsf{t}_\la}
\nc{\llrr}{\langle\la,\rho\rangle}
\nc{\lllr}{\langle\la,\la\rangle}
\nc{\K}{\mathbbm{k}}
\nc{\Stosic}{Sto{\v{s}}i{\'c}\xspace}
\nc{\cd}{\mathcal{D}}
\nc{\vd}{\mathbb{D}}
\nc{\R}{\mathbb{R}}
  \nc{\Lam}[3]{\La^{#1}_{#2,#3}}
  \nc{\Lab}[2]{\La^{#1}_{#2}}
  \nc{\Lamvwy}{\Lam\Bv\Bw\By}
  \nc{\Labwv}{\Lab\Bw\Bv}
  \nc{\nak}[3]{\mathcal{N}(#1,#2,#3)}
  \nc{\hw}{highest weight\xspace}
  \nc{\al}{\alpha}
  \nc{\be}{\beta}
  \nc{\bM}{\mathbf{m}}
  \nc{\bx}{\mathbf{x}}
\nc{\bp}{\mathbf{p}}
\nc{\by}{\mathbf{y}}
 \nc{\bkh}{\backslash}
  \nc{\Bi}{\mathbf{i}}
  \nc{\Bm}{\mathbf{m}}
\nc{\bd}{\mathbf{d}}
  \nc{\bpi}{\boldsymbol{\pi}}

  \nc{\Bj}{\mathbf{j}}
\nc{\RAA}{R^\A_A}
  \nc{\Bv}{\mathbf{v}}
  \nc{\Bw}{\mathbf{w}}
\nc{\Id}{\operatorname{Id}}
  \nc{\By}{\mathbf{y}}
\nc{\eE}{\EuScript{E}}
\nc{\eI}{\EuScript{I}}
  \nc{\Bz}{\mathbf{z}}
  \nc{\coker}{\mathrm{coker}\,}
  \nc{\C}{\mathbb{C}}
  \nc{\ch}{\mathrm{ch}}
  \nc{\de}{\delta}
  \nc{\ep}{\epsilon}
  \nc{\Rep}[2]{\mathsf{Rep}_{#1}^{#2}}
  \nc{\Ev}[2]{E_{#1}^{#2}}
  \nc{\fr}[1]{\mathfrak{#1}}
  \nc{\fp}{\fr p}
  \nc{\fq}{\fr q}
  \nc{\fl}{\fr l}
  \nc{\fgl}{\fr{gl}}
\nc{\rad}{\operatorname{rad}}
\nc{\ind}{\operatorname{ind}}
  \nc{\GL}{\mathrm{GL}}
  \nc{\Hom}{\mathrm{Hom}}
  \nc{\im}{\mathrm{im}\,}
 \nc{\SHom}{\EuScript{H}\!\mathit{om}}
  
 \nc{\La}{\Lambda}
  \nc{\la}{\lambda}
  \nc{\mult}{b^{\mu}_{\la_0}\!}
  \nc{\mc}[1]{\mathcal{#1}}
  \nc{\om}{\omega}
\nc{\gl}{\mathfrak{gl}}
  \nc{\cF}{\mathcal{F}}
 \nc{\cC}{\mathcal{C}}
  \nc{\Vect}{\mathsf{Vect}}
 \nc{\modu}{\operatorname{-mod}}
  \nc{\qvw}[1]{\La(#1 \Bv,\Bw)}
  \nc{\van}[1]{\nu_{#1}}
  \nc{\Rperp}{R^\vee(X_0)^{\perp}}
  \nc{\si}{\sigma}
  \nc{\croot}[1]{\al^\vee_{#1}}
\nc{\di}{\mathbf{d}}
  \nc{\SL}[1]{\mathrm{SL}_{#1}}
  \nc{\Th}{\theta}
  \nc{\vp}{\varphi}
  \nc{\wt}{\mathrm{wt}}
  \nc{\Z}{\mathbb{Z}}
  \nc{\Q}{\mathbb{Q}}
  \nc{\Znn}{\Z_{\geq 0}}
  \nc{\ver}{\EuScript{V}}
  \nc{\Res}[2]{\operatorname{Res}^{#1}_{#2}}
  \nc{\edge}{\EuScript{E}}
  \nc{\Spec}{\mathrm{Spec}}
  \nc{\tie}{\EuScript{T}}
  \nc{\ml}[1]{\mathbb{D}^{#1}}
  \nc{\fQ}{\mathfrak{Q}}
        \nc{\fg}{\mathfrak{g}}
  \nc{\Uq}{U_q(\fg)}
        \nc{\bom}{\boldsymbol{\omega}}
\nc{\bla}{{\underline{\boldsymbol{\la}}}}
\nc{\bmu}{{\boldsymbol{\mu}}}
\nc{\bal}{{\boldsymbol{\al}}}
\nc{\bet}{{\boldsymbol{\eta}}}
\nc{\rola}{X}
\nc{\wela}{Y}
\nc{\fM}{\mathfrak{M}}
\nc{\fX}{\mathfrak{X}}
\nc{\fH}{\mathfrak{H}}
\nc{\fE}{\mathfrak{E}}
\nc{\fF}{\mathfrak{F}}
\nc{\fI}{\mathfrak{I}}
\nc{\qui}[2]{\fM_{#1}^{#2}}
\renc{\cL}{\mathcal{L}}
\nc{\ca}[2]{\fQ_{#1}^{#2}}
\nc{\cat}{\mathcal{V}}
\nc{\cata}{\mathfrak{V}}
\nc{\pil}{{\boldsymbol{\pi}}^L}
\nc{\pir}{{\boldsymbol{\pi}}^R}
\nc{\cO}{\mathcal{O}}
\nc{\tO}{\tilde{\cO}}
\nc{\Ko}{\text{\Denarius}}
\nc{\Ei}{\fE_i}
\nc{\Fi}{\fF_i}
\nc{\fil}{\mathcal{H}}
\nc{\brr}[2]{\beta^R_{#1,#2}}
\nc{\brl}[2]{\beta^L_{#1,#2}}
\nc{\so}[2]{\EuScript{Q}^{#1}_{#2}}
\nc{\EW}{\mathbf{W}}
\nc{\rma}[2]{\mathbf{R}_{#1,#2}}
\nc{\Dif}{\EuScript{D}}
\nc{\MDif}{\EuScript{E}}
\renc{\mod}{\mathsf{mod}}
\nc{\modg}{\mathsf{mod}^g}
\nc{\fmod}{\mathsf{mod}^{fd}}
\nc{\id}{\operatorname{id}}
\nc{\DR}{\mathbf{DR}}
\nc{\End}{\operatorname{End}}
\nc{\Fun}{\operatorname{Fun}}
\nc{\Ext}{\operatorname{Ext}}
\nc{\tw}{\tau}
\nc{\lcm}{\operatorname{lcm}}
\nc{\A}{\EuScript{A}}
\nc{\Loc}{\mathsf{Loc}}
\nc{\eF}{\EuScript{F}}
\nc{\LAA}{\Loc^{\A}_{A}}
\nc{\perv}{\mathsf{Perv}}
\nc{\teF}{\tilde{\EuScript{F}}}
\nc{\teE}{\tilde{\EuScript{E}}}
\nc{\gfq}[2]{B_{#1}^{#2}}
\nc{\qgf}[1]{A_{#1}}
\nc{\qgr}{\qgf\rho}
\nc{\tqgf}{\tilde A}
\nc{\IC}{\mathbf{IC}}
\nc{\Tr}{\operatorname{Tr}}
\nc{\Tor}{\operatorname{Tor}}
\nc{\cQ}{\mathcal{Q}}
\nc{\st}[1]{\Delta(#1)}
\nc{\cst}[1]{\nabla(#1)}
\nc{\ei}{\mathbf{e}_i}
\nc{\Be}{\mathbf{e}}
\nc{\Hck}{\mathfrak{H}}
\renc{\P}{\mathbb{P}}
\nc{\cI}{\mathcal{I}}
\nc{\tU}{\mathcal{U}}
\nc{\hU}{\widehat{\mathcal U}}
\nc{\coe}{\mathfrak{K}}
\nc{\pr}{\operatorname{pr}}
\nc{\ttU}{\tilde{\mathcal{U}}}
\nc{\bra}{\mathfrak{B}}
\nc{\rcl}{\rho^\vee(\la)}
\nc{\bz}{\mathbf{z}}
\nc{\bLa}{\boldsymbol{\Lambda}}
\nc{\hwo}{\mathbb{V}}
\nc{\cosoc}{\operatorname{cosoc}}
\nc{\socle}{\operatorname{soc}}
\nc{\alg}{T}

\excise{
\newenvironment{block}
\newenvironment{frame}
\newenvironment{tikzpicture}
\newenvironment{equation*}
}

\setcounter{tocdepth}{1}

\baselineskip=1.1\baselineskip
\renc{\theequation}{\arabic{section}.\arabic{equation}}

 \usetikzlibrary{decorations.pathreplacing,backgrounds,decorations.markings}
\tikzset{wei/.style={draw=red,double=red!40!white,double distance=1.5pt,thin}}
\tikzset{bdot/.style={fill,circle,color=blue,inner sep=3pt,outer sep=0}}
\tikzset{dir/.style={postaction={decorate,decoration={markings,
    mark=at position .8 with {\arrow[scale=1.3]{<}}}}}}
\begin{center}
\noindent {\large  \bf Knot invariants and higher representation theory I:\\ diagrammatic and geometric categorification of tensor products}
\medskip

\noindent {\sc Ben Webster}\footnote{Supported by an NSF Postdoctoral Research Fellowship and  by the NSA under Grant H98230-10-1-0199.}\\
Department of Mathematics\\ Northeastern University\\
Boston, MA\\
Email: {\tt b.webster@neu.edu}
\end{center}
\bigskip
{\small
\begin{quote}
\noindent {\em Abstract.}
In this paper, we study 2-representations of 2-quantum groups (in the sense of Rouquier and Khovanov-Lauda) categorifying tensor products of irreducible representations. Our aim is to construct knot homologies categorifying Reshetikhin-Turaev invariants of knots for arbitrary representations, which will be done in a follow-up paper.

We consider an algebraic construction of these categories, via an
explicit diagrammatic presentation, generalizing the cyclotomic
quotient of the quiver Hecke algebra.  When the Lie algebra under
consideration is $\mathfrak{sl}_n$, we show that these categories
agree with certain subcategories of parabolic category $\cO$ for $\mathfrak{gl}_k$.

We also investigate finer structure of these categories.  Like many similar representation-theoretic categories, they are standardly stratified and satisfy a double centralizer property with respect to their self-dual modules.  The standard modules of the stratification play an important role, as Vermas do in more classical representation theory, as test objects for functors. 

The existence of these representations has consequences for the
structure of previously studied categorifications; it allows us to
prove the non-degeneracy of Khovanov and Lauda's 2-category (that its
Hom spaces have the expected dimension) in all symmetrizable types,
and that the cyclotomic quiver Hecke algebras are symmetric Frobenius.
\end{quote}
}

\vspace{5mm}

\renc{\thethm}{\Alph{thm}}


The program of ``higher representation theory,'' begun (at least as an
explicit program) by Chuang and Rouquier in \cite{CR04} and continued
by Rouquier \cite{Rou2KM} and Khovanov-Lauda \cite{KLIII} is aimed at
studying ``2-analogues'' of the universal enveloping algebras of
simple Lie algebras $U(\fg)$, and their quantizations $U_q(\fg)$.  In
this paper, we study certain representations of these analogues. Our
objects of study are certain explicitly given categories which are
categorifications of tensor products of simple integrable modules for
$U_q(\fg)$ (in the sense that their Grothendieck groups are integral
forms of these representations).  Our interest in these categories has
arisen because of their applicability to the construction of knot
invariants, which we address in a sequel to this paper \cite{KI-HRT};
however, we believe they are also of independent interest.

These algebras also have connections in the type A case to classical
representation theory, as has been explored by Brundan and Kleshchev
\cite{BKKL}. We will build on their work in Section \ref{sec:type-A}
by showing that our categories appear in the context of category $\cO$
in type A.  For a general $\fg$, our categories should be viewed as a generalization of the
type A category $\cO$ orthogonal to that of category $\cO$ for other
groups, just as quiver varieties are a generalization of the type A
flag variety orthogonal to the flag varieties of other types.

Our primary construction of these categories is algebraic; the
underlying category $\cata^\bla$ is the representations of an algebra
$\alg^\bla$ defined in this paper.  The algebra $T^\bla$ is a
generalization of the cyclotomic quiver Hecke algebra introduced by
Khovanov and Lauda.  This categorification is well defined for any
symmetrizable Kac-Moody algebra, and it depends on a choice of base
field $\K$ and polynomial $Q_{ij}\in \K[u,v]$ for all $i,j$ in the
Dynkin diagram.
Our
main theorem is as follows:
\begin{thm}\label{main}
  The category $\cata^\bla$ is a
  categorification in the sense of Khovanov-Lauda, that is, it carries
  an action of the 2-category $\tU$ defined in \cite[\S 2]{CaLa} and its Grothendieck group is canonically isomorphic to
  the tensor product $$V_\bla\cong V_{\la_1}\otimes \cdots\otimes
  V_{\la_\ell}.$$
\end{thm}
We should note that here, and throughout the paper, ``2-category'' is
meant in the strict sense, not that of bicategories, so associativity
laws hold ``on the nose.''

In the case where $\fg$ is finite-type and simply-laced, we can
strengthen this theorem to show in \cite[6.11]{WebCB} that the
indecomposable projectives in this category give Lusztig's canonical
basis of a tensor product.

We should note that even in the case of $\bla=(\la)$, where the
algebra $T^\la$ is a cyclotomic quotient in the sense of \cite[\S
3.4]{KLIII}, this is a new theorem, which in particular implies that
the induction and restriction functors on these categories are
biadjoint.  This was proved independently by Kang and Kashiwara
\cite{KK} by completely different methods.

We show that these categories have many properties that would be
expected by analogy with similar representation-theoretic categories:
\begin{thm}
  The projectives-injective objects of $\cata^\bla$
  form a categorification of the subrepresentation $V_{\la_1+\dots
    +\la_n}\subset V_\bla$.  In particular, if $\bla=(\la)$, then 
  all projectives are injective; in fact, the algebra $\alg^{(\la)}$
  is Frobenius.

  The sum of all indecomposable projective-injectives has the double
  centralizer property; this realizes $\alg^\bla$ as the endomorphisms
  of a natural collection of modules over the algebra for the
  corresponding simple module $\alg^{(\la_1+\dots +\la_n)}$.

  The algebra $\alg^\bla$ is standardly stratified; the
  semi-orthogonal decomposition for this stratification categorifies
  the decomposition of $V_\bla$ as the sum of tensor products of
  weight spaces.
\end{thm}
This double centralizer result allows us to generalize a theorem of
Brundan and Kleshchev \cite[Main Theorem]{BKKL}, and show that in type
$A$, the algebras $\alg^\bla$ are endomorphism algebras of certain
projectives in parabolic category $\cO$, while in type $\widehat A$, they are related to the representations of the cyclotomic $q$-Schur algebra.  This relationship will be explored more fully in work of the author and Stroppel \cite{SWschur}.

We see no reason to think that our category has a similar description
in terms of classical representation theory when
$\fg\not\cong\mathfrak{sl}_n, \widehat{\mathfrak{sl}}_n$, though we would be quite pleased to be
proven wrong in this speculation.\medskip

The action on these categories plays a similar role to the actions of
equivariant cohomology studied by Lauda in \cite{LauSL21,LauSL2} and
Khovanov-Lauda in \cite{KLIII}; it shows by direct construction that
the set of diagrams conjectured by Khovanov and Lauda to give a basis
of 2-morphisms indeed does (because there is no linear combination of
them that acts trivially on all categories $\cata^\bla$).
\begin{thm}\label{nond}
  The 2-category $\tU$ is nondegenerate (in the
  sense of \cite[Definition 3.15]{KLIII}) over any field.
\end{thm}

Let us now summarize the structure of the paper.
\begin{itemize}
\item In Section \ref{sec:categorification}, we discuss the basics of
  the 2-category $\tU$, and prove it acts on $\cata^\la$.
  This is accomplished by the construction of categorifications
  $\tU^-_i$ for the minimal non-solvable parabolics $U(\mathfrak{p}_i)$.
  These categories carry a mixture of the characteristics of
  $U(\mathfrak{b})$ and $U(\mathfrak{sl}_2)$; an
  appropriate non-degeneracy result is already known for both of
  these algebras separately. By modifying the
  proofs of these previous results, we can show that $\tU_i^-$ acts on $\cata^\la$.  It
  is an easy consequence of this that the full $\tU$ acts, which proves Theorem \ref{nond}.
\item In Section \ref{sec:KL}, we define the algebras $\alg^\bla$.  
  As far as we know, these algebras are new to the literature, but are
  constructed using the familiar tool of Khovanov-Lauda's graphical
  calculus.  This graphical calculus gives an easy description of the
  action of the category $\tU$.  We also study the relationship of this category to $\alg^{(\la_1+\cdots +\la_\ell)}$.
\item In Section \ref{sec:standard},  we develop a
  special class of modules which we term {\bf standard modules}, which categorify pure tensors.
  These are typically not the standard modules of a quasi-hereditary
  structure, but rather of a weaker standardly stratified structure.  Amongst other things, these modules will prove crucial as ``test'' objects for understanding how functors decategorify.
\item In Section \ref{sec:type-A}, we consider the case
  $\fg=\mathfrak{sl}_n$ or $\widehat{\mathfrak{sl}}_n$.  In this case, we employ results of Brundan
  and Kleshchev to show that $\alg^\bla$ is in fact the endomorphism algebra of
  a projective in a parabolic category $\cO$ in finite type and in the representation category of the cyclotomic $q$-Schur algebra in affine type.  This result will be important for comparing our construction of knot homology in the sequel to versions previously defined using category $\cO$.
\end{itemize}

We should note that an earlier version of this paper contained a
section on the connection between the algebraic material in this paper
to the geometry of quiver varieties.  In the interest of length and
heaviness of machinery, that material has been moved to other papers \cite{Webcatq,WebwKLR}.

\subsection*{Notation}

We let $\fg$ be a symmetrizable Kac-Moody algebra, which we will assume is
fixed for the remainder of the paper.  Let $\Gamma$ denote the Dynkin
diagram of this algebra, considered as an unoriented graph.
We denote the weight
lattice $\wela(\fg)$ and root lattice $\rola(\fg)$, the
simple roots $\al_i$ and coroots $\al_i^\vee$.  Let
$c_{ij}=\al_j^{\vee}(\al_i)$ be the entries of the Cartan matrix.

We let $\langle
-,-\rangle$ denote the symmetrized inner product on $\wela(\fg)$,
fixed by the fact that the shortest root has length $\sqrt{2}$
and $$2\frac{\langle \al_i,\la\rangle}{\langle
  \al_i,\al_i\rangle}=\al_i^\vee(\la).$$ As usual, we let $2d_i
=\langle\al_i,\al_i\rangle$, and for $\la\in\wela(\fg)$, we
let $$\la^i=\al_i^\vee(\la)=\langle\al_i,\la\rangle/d_i.$$
We note that we have $d_ic_{ji}=d_jc_{ij}$ for all $i,j$.

\excise{
We let $\rho$ be the unique weight such that $\al^\vee_i(\rho)=1$ and
$\rho^\vee$ the unique coweight such that $\rho^\vee(\al_i)=1$.  We
note that since $\rho\in \nicefrac 12\rola$ and $\rho^\vee\in
\nicefrac 12\wela^*$, for any weight $\la$, the numbers $\llrr$ and
$\rcl$ are not necessarily integers, but $2\llrr$ and $2\rcl$
are (not necessarily even) integers.

Throughout the paper, we will use $\bla=(\la_1,\dots, \la_\ell)$ to
denote an ordered $\ell$-tuple of dominant weights, and always use the
notation $\la=\sum_{i}\la_i$.  }

We let $U_q(\fg)$ denote the deformed universal enveloping algebra of
$\fg$; that is, the 
associative $\C(q)$-algebra given by generators $E_i$, $F_i$, $K_{\mu}$ for $i\in \Gamma$ and $\mu \in \wela(\fg)$, subject to the relations:
\begin{center}
\begin{enumerate}[i)]
 \item $K_0=1$, $K_{\mu}K_{\mu'}=K_{\mu+\mu'}$ for all $\mu,\mu' \in \wela(\fg)$,
 \item $K_{\mu}E_i = q^{\al_i^{\vee}(\mu)}E_iK_{\mu}$ for all $\mu \in
 \wela(\fg)$,
 \item $K_{\mu}F_i = q^{- \al_i^{\vee}(\mu)}F_iK_{\mu}$ for all $\mu \in
 \wela(\fg)$,
 \item $E_iF_j - F_jE_i = \delta_{ij}
 \frac{\tilde{K}_i-\tilde{K}_{-i}}{q^{d_i}-q^{-d_i}}$, where
 $\tilde{K}_{\pm i}=K_{\pm d_i \al_i}$,
 \item For all $i\neq j$ $$\sum_{a+b=-c_{ij}+1}(-1)^{a} E_i^{(a)}E_jE_i^{(b)} = 0
 \qquad {\rm and} \qquad
 \sum_{a+b=-c_{ij} +1}(-1)^{a} F_i^{(a)}F_jF_i^{(b)} = 0 .$$
\end{enumerate} \end{center}

This is a Hopf algebra with coproduct on Chevalley generators given
by $$\Delta(E_i)=E_i\otimes 1
+\tilde K_i\otimes E_i\hspace{1cm}\Delta(F_i)=F_i\otimes \tilde K_{-i}
+ 1 \otimes F_i$$

We let $U_q^\Z(\fg)$ denote the Lusztig (divided powers) integral form
generated over $\Z[q,q^{-1}]$ by
$\frac{E_i^n}{[n]_q!},\frac{F_i^n}{[n]_q!}$ for all integers $n$ of
this quantum group.  The integral form of the representation of
highest weight $\la$ over this quantum group will be denoted by
$V_\la^{\Z}$; for a sequence $\bla$, we will be interested in the
tensor product \[V_\bla^\Z=V_{\la_1}^\Z\otimes_{\Z[q,q^{-1}]}\cdots
\otimes_{\Z[q,q^{-1}]}V_{\la_\ell}^\Z;\] we will also consider the
completion of these modules in the $q$-adic topology
$V_\bla=V_\bla^\Z\otimes_{\Z[q,q^{-1}]}\Z((q))$.

We will always use $K_0(R)$ for a graded ring $R$ to denote the
Grothendieck group of finitely generated graded projective
$R$-modules. This group carries an action of $\Z[q,q^{-1}]$ by grading
shift $[A(i)]=q^i[A]$, where $A(i)$ is the graded module whose $j$th
grade is the $i+j$th of $A$.  The careful reader should note that this
is opposite to the grading convention of Khovanov and Lauda.

\subsection*{Acknowledgments}
I would like to thank Catharina Stroppel for both suggestions on this
project, and teaching me a great deal of quasi-hereditary and
standardly stratified representation theory; Jon Brundan, Ben Elias,
Alex Ellis, Jun Hu, Joel Kamnitzer, Masaki Kashiwara, Mikhail
Khovanov, Marco Mackaay, Josh Sussan and Oded Yacobi for valuable
comments on earlier versions of this paper; Hao Zheng, Rapha\"el
Rouquier, Nick Proudfoot, Tony Licata, Aaron Lauda and Tom Braden for
very useful discussions.

I'd also like to thank the organizers of the conference ``Categorification and Geometrization from Representation Theory'' in Glasgow which made a huge difference in the development of these ideas.

\tableofcontents

\section{Categorification of quantum groups}\label{sec:categorification}
\subsection{2-Categories}
\renc{\thethm}{\arabic{section}.\arabic{thm}}

Let me re-emphasize that in this paper, unless specified otherwise,
``2-category'' will mean a strict 2-category, not a weak one or bicategory.

In this paper, our notation builds on that of Khovanov and Lauda, who
give a graphical version of the 2-quantum group, which we denote $\tU$
(leaving $\fg$ understood).  These constructions could also be
rephrased in terms of Rouquier's description and we have striven to
make the paper readable following either \cite{KLIII} or
\cite{Rou2KM}; however, it is most sensible for us to follow the
2-category defined by Cautis and Lauda \cite{CaLa}, which is a
variation on both of these.  The difference between this category and
the categories defined by Rouquier in \cite{Rou2KM} is quite
subtle--it concerns precisely whether the inverse to a particular map
is formally added, or imposed to be a particular composition of other
generators in the category. Most important for our purposes, the
2-category $\tU$ receives a canonical map from each of Rouquier's
categories $\mathfrak{A}$ and $\mathfrak{A}'$, so a representation of
it is a representation in Rouquier's sense as well.  Furthermore,
Cautis and Lauda have shown that any representation in Rouquier's
sense satisfying very mild technical conditions will be a
representation of $\tU$.

Since the construction of these categories is rather complex, we give a somewhat abbreviated description.  The most important points are these:
\begin{itemize}
\item an object of this category is a weight $\la\in \wela$.
\item a 1-morphism $\la\to \mu$ is a formal sum of words in the
  symbols $\eE_i$ and $\eF_i$ where $i$ ranges over $\Gamma$ of weight
  $\la-\mu$, $\eE_i$ and $\eF_i$ having weights $\pm\al_i$.  In \cite{Rou2KM},
the corresponding 1-morphisms are denoted $E_i,F_i$, but we use these for elements of
$U_q(\fg)$.  Composition is simply concatenation of words.  In fact, we will take idempotent completion, and thus add a new 1-morphism for every projection from a 1-morphism to itself (once we have added 2-morphisms).

By convention, $\eF_{\Bi}=\eF_{i_n}\cdots \eF_{i_1}$ if
$\Bi=(i_1,\dots,i_n)$ (this somewhat dyslexic convention is designed
to match previous work on cyclotomic quotients by Khovanov-Lauda and
others).  In Khovanov and Lauda's graphical calculus, this 1-morphism
is represented by a sequence of dots on a horizontal line labeled with
the sequence $\Bi$.  

We should warn the reader, this convention
requires us to read our diagrams differently from the conventions of
\cite{LauSL21,KLIII,CaLa}; in our diagrammatic calculus, 1-morphisms point
from the left to the right, not from the right to the left as
indicated in \cite[\S 4]{LauSL21}.  Technically, the 2-category $\tU$ we
define is the 1-morphism dual of Cautis and Lauda's 2-category: the
objects are the same, but the 1-morphisms are all reversed.  The
practical implication will be that our relations are the reflection
through a vertical line of Cautis and Lauda's (without changing the
labeling of regions).

\item 2-morphisms are a certain quotient of the $\K$-span of certain
  immersed oriented 1-manifolds carrying an arbitrary number of dots
  whose boundary is given by the domain sequence on the line $y=1$ and
  the target sequence on $y=0$.  We require that any component begin
  and end at like-colored elements of the 2 sequences, and that they
  be oriented upward at an $\eE_i$ and downward at an $\eF_i$.  We
  will describe their relations momentarily. We require that these
  1-manifolds satisfy the same genericity assumptions as projections
  of tangles (no triple points or tangencies), but intersections are
  not over- or under-crossings; our diagrams are genuinely planar.  We
  consider these up to isotopy which preserves this genericity.

We draw these 2-morphisms in the style of Khovanov-Lauda, by labeling the regions of the plane by the weights (objects) that the 1-morphisms are acting on.

By Morse theory, we can see that these are generated by 
\renewcommand{\labelitemii}{$*$} 
\begin{itemize}
\item a cup $\ep:\eE_i\eF_i\to \emptyset$ or  $\ep':\eF_i\eE_i\to
  \emptyset$
\[
\ep=\tikz[baseline,very thick,scale=2.5]{\draw[<-] (.2,.1)
  to[out=-120,in=-60]  node[at end,above left,scale=.8]{$i$}
  node[at start,above right,scale=.8]{$i$} (-.2,.1);\node[scale=.8] at
  (0,.3){$\la$}; \node[scale=.8] at (0,-.2){$\la+\al_i$};}\qquad \qquad
\ep'=\tikz[baseline,very thick,scale=2.5]{\draw[->] (.2,.1)
  to[out=-120,in=-60] node[at end,above left,scale=.8]{$i$}
  node[at start,above right,scale=.8]{$i$} (-.2,.1) ;\node[scale=.8] at
  (0,.3){$\la$}; \node[scale=.8] at (0,-.2){$\la-\al_i$};}
\]
\item a cap $\iota':\emptyset\to \eE_i\eF_i$ or  $\iota:\emptyset\to
  \eF_i\eE_i$
\[
\iota'=\tikz[baseline,very thick,scale=2.5]{\draw[->] (.2,.1)
  to[out=120,in=60] node[at end,below left,scale=.8]{$i$}
  node[at start,below right,scale=.8]{$i$}  (-.2,.1) ;\node[scale=.8] at
  (0,.4){$\la$}; \node[scale=.8] at (0,-.1){$\la-\al_i$};}\qquad \qquad
\iota=\tikz[baseline,very thick,scale=2.5]{\draw[<-] (.2,.1)
  to[out=120,in=60] node[at end,below left,scale=.8]{$i$}
  node[at start,below right,scale=.8]{$i$} (-.2,.1) ;\node[scale=.8] at
  (0,.4){$\la$}; \node[scale=.8] at (0,-.1){$\la+\al_i$};}
\]
\item a crossing $\psi:\eF_i\eF_j\to \eF_j\eF_i$
\[\psi=
\tikz[baseline=-2pt,very thick,scale=2.5]{\draw[->] (.2,.2) --
  (-.2,-.2) node[at end,below,scale=.8]{$i$} node[at start,above,scale=.8]{$i$} ; \draw[<-]
  (.2,-.2) -- (-.2,.2) node[at start,below,scale=.8]{$j$} node[at
  end,above,scale=.8]{$j$}; \node[scale=.8] at (0,.4) {$\la$}; \node[scale=.8] at (0,-.4) {$\la$};
} \]
\item a dot $y:\eF_i\to \eF_i$
\[
 y=\tikz[baseline=-2pt,very thick,->,scale=2.5]{\draw
  (0,.2) -- (0,-.2) node[at end,below,scale=.8]{$i$} node[at start,above,scale=.8]{$i$}
  node[midway,circle,fill=black,inner
  sep=2pt]{}; \node[scale=.8] at (0,.45) {$\la$}; \node[scale=.8] at (0,-.45) {$\la$};}\]
\end{itemize}
\end{itemize}

Once and for all, fix a matrix of
polynomials $Q_{ij}(u,v)$ for $i\neq j\in \Gamma$ (by convention
$Q_{ii}=0$) valued in $\K$. We assume each polynomial is homogeneous
of degree $\langle\al_i,\al_j\rangle= -2d_jc_{ij}=-2d_ic_{ji}$ when
$u$ is given degree $2d_i$ and $v$ degree $2d_j$.   We will always
assume that the leading order of $Q_{ij}$ in $u$ is $-c_{ji}$, and
that $Q_{ij}(u,v)=Q_{ji}(v,u)$.  We let $t_{ij}=Q_{ij}(1,0)$; by
convention $t_{ii}=1$. In order to match with \cite{CaLa}, we take
\[Q_{ij}(u,v)=t_{ij} u^{-c_{ji}}+t_{ij} v^{-c_{ij}}+\sum_{q c_{ji}+ pc_{ij}=c_{ji}c_{ij}} s^{pq}_{ij}u^pv^q.\]
Khovanov and Lauda's original category is the
choice $Q_{ij}=u^{-c_{ji}}+v^{-c_{ij}}$.

Before writing the relations, let us remind the reader that these
2-morphism spaces are actually graded; the degrees are given by \[
\deg\tikz[baseline,very thick,scale=1.5]{\draw[->] (.2,.3) --
  (-.2,-.1) node[at end,below, scale=.8]{$i$}; \draw[<-]
  (.2,-.1) -- (-.2,.3) node[at start,below,scale=.8]{$j$};} =-\langle\al_i,\al_j\rangle \qquad  \deg\tikz[baseline,very thick,->,scale=1.5]{\draw
  (0,.3) -- (0,-.1) node[at end,below,scale=.8]{$i$}
  node[midway,circle,fill=black,inner
  sep=2pt]{};}=\langle\al_i,\al_i\rangle \qquad \deg\tikz[baseline,very thick,scale=1.5]{\draw[<-] (.2,.3) --
  (-.2,-.1) node[at end,below,scale=.8]{$i$}; \draw[->]
  (.2,-.1) -- (-.2,.3) node[at start,below,scale=.8]{$j$};} =-\langle\al_i,\al_j\rangle \qquad  \deg\tikz[baseline,very thick,<-,scale=1.5]{\draw
  (0,.3) -- (0,-.1) node[at end,below,scale=.8]{$i$}
  node[midway,circle,fill=black,inner
  sep=2pt]{};}=\langle\al_i,\al_i\rangle\]
\[
\deg\tikz[baseline,very thick,scale=1.5]{\draw[->] (.2,.1)
  to[out=-120,in=-60] (-.2,.1)  node[at end,above left,scale=.8]{$i$};\node[scale=.8] at (0,.3){$\la$};} =-\langle\la,\al_i\rangle-d_i \qquad 
\deg\tikz[baseline,very thick,scale=1.5]{\draw[<-] (.2,.1)
  to[out=-120,in=-60] (-.2,.1) node[at end,above left,scale=.8]{$i$};\node[scale=.8] at (0,.3){$\la$};}
=\langle \la,\al_i\rangle-d_i
\] 
\[
\deg\tikz[baseline,very thick,scale=1.5]{\draw[->] (.2,.1)
  to[out=120,in=60] (-.2,.1) node[at end,below left,scale=.8]{$i$};\node[scale=.8] at (0,-.1){$\la$};} =-\langle\la,\al_i\rangle-d_i \qquad 
\deg\tikz[baseline,very thick,scale=1.5]{\draw[<-] (.2,.1)
  to[out=120,in=60] (-.2,.1) node[at end,below left,scale=.8]{$i$};\node[scale=.8] at (0,-.1){$\la$};}
=\langle \la,\al_i\rangle-d_i.
\]

\medskip

The relations satisfied by the 2-morphisms include:
\begin{itemize}
\item the cups and caps are the units and counits of a biadjunction.
  The morphism $y$ is cyclic, whereas the morphism $\psi$ is double
  right dual to $t_{ij}/t_{ji}\cdot \psi$ (see \cite{CaLa} for more details).
\item Any bubble of negative degree is zero,
  any bubble of degree 0 is equal to 1.  We must add formal symbols
  called ``fake bubbles'' which are bubbles labelled with a negative
  number of dots (these are explained in \cite[\S 3.1.1]{KLIII});
  given these, we have the inversion formula for bubbles, shown in Figure
  \ref{inv-rels}.
\begin{figure}[h!]
\begin{tikzpicture}
\node at (-1,0) {$\displaystyle \sum_{k=\la^i-1}^{j+\la^i+1}$};
\draw[postaction={decorate,decoration={markings,
    mark=at position .5 with {\arrow[scale=1.3]{<}}}},very thick] (.5,0) circle (15pt);
\node [fill,circle,inner sep=2.5pt,label=right:{$k$},right=11pt] at (.5,0) {};
\node[scale=1.5] at (1.375,-.6){$\la$};
\draw[postaction={decorate,decoration={markings,
    mark=at position .5 with {\arrow[scale=1.3]{>}}}},very thick] (2.25,0) circle (15pt);
\node [fill,circle,inner sep=2.5pt,label=right:{$j-k$},right=11pt] at (2.25,0) {};
\node at (5.7,0) {$
  =\begin{cases}
    1 & j=-2\\
    0 & j>-2
  \end{cases}
$};
\end{tikzpicture}
\caption{Bubble inversion relations; all strands are colored with $\al_i$.}
\label{inv-rels}
\end{figure}
\item 2 relations connecting the crossing with cups and caps, shown in Figure \ref{pop-rels}.

\begin{figure}[t!]
\begin{tikzpicture}[scale=.9]
\node at (-3,-3.5){
\scalebox{.95}{\begin{tikzpicture} [scale=1.3]
\node at (0,0){\begin{tikzpicture} [scale=1.3]
\node[scale=1.5] at (-.7,1){$\la$};
\draw[postaction={decorate,decoration={markings,
    mark=at position .5 with {\arrow[scale=1.3]{>}}}},very thick] (0,0) to[out=90,in=-90] (1,1) to[out=90,in=-90] (0,2);  
\draw[postaction={decorate,decoration={markings,
    mark=at position .5 with {\arrow[scale=1.3]{<}}}},very thick] (1,0) to[out=90,in=-90] (0,1) to[out=90,in=-90] (1,2);
\end{tikzpicture}};

\node at (1.7,0) {$=$};
\node at (5.4,0){\begin{tikzpicture} [scale=1.3]

\node[scale=1.5] at (.3,1){$\la$};

\node at (.7,1) {$-$};
\draw[postaction={decorate,decoration={markings,
    mark=at position .5 with {\arrow[scale=1.3]{>}}}},very thick] (1,0) to[out=90,in=-90]  (1,2);  
\draw[postaction={decorate,decoration={markings,
    mark=at position .5 with {\arrow[scale=1.3]{<}}}},very thick] (1.7,0) to[out=90,in=-90]  (1.7,2);

\node at (2.5,1.15) {$+$}; 
\node at (3,1) {$\displaystyle\sum_{a+b+c=-1}$};
\draw[postaction={decorate,decoration={markings,
    mark=at position .5 with {\arrow[scale=1.3]{>}}}},very thick] (4,0) to[out=90,in=-180] (4.5,.5) to[out=0,in=90] node [pos=.6, fill,circle,inner sep=2.5pt,label=above:{$a$}] {} (5,0);  
\draw[postaction={decorate,decoration={markings,
    mark=at position .5 with {\arrow[scale=1.3]{<}}}},very thick] (4,2) to[out=-90,in=-180] (4.5,1.5) to[out=0,in=-90] node [pos=.6, fill,circle,inner sep=2.5pt,label=below:{$c$}] {} (5,2);
\draw[postaction={decorate,decoration={markings,
    mark=at position .5 with {\arrow[scale=1.3]{>}}}},very thick] (5.5,1) circle (10pt);
\node [fill,circle,inner sep=2.5pt,label=right:{$b$},right=10.5pt] at (5.5,1) {};
\node[scale=1.5] at (6.5,1){$\la$};
\end{tikzpicture}};

\end{tikzpicture}}};
\node at (-3,0){\scalebox{.95}{\begin{tikzpicture} [scale=1.3]
\node at (0,0){\begin{tikzpicture} [scale=1.3]
\node[scale=1.5] at (-.7,1){$\la$};
\draw[postaction={decorate,decoration={markings,
    mark=at position .5 with {\arrow[scale=1.3]{<}}}},very thick] (0,0) to[out=90,in=-90] (1,1) to[out=90,in=-90] (0,2);  
\draw[postaction={decorate,decoration={markings,
    mark=at position .5 with {\arrow[scale=1.3]{>}}}},very thick] (1,0) to[out=90,in=-90] (0,1) to[out=90,in=-90] (1,2);
\end{tikzpicture}};

\node at (1.7,0) {$=$};
\node at (5.4,0){\begin{tikzpicture} [scale=1.3]

\node[scale=1.5] at (.3,1){$\la$};

\node at (.7,1) {$-$};
\draw[postaction={decorate,decoration={markings,
    mark=at position .5 with {\arrow[scale=1.3]{<}}}},very thick] (1,0) to[out=90,in=-90]  (1,2);  
\draw[postaction={decorate,decoration={markings,
    mark=at position .5 with {\arrow[scale=1.3]{>}}}},very thick] (1.7,0) to[out=90,in=-90]  (1.7,2);

\node at (2.5,1.15) {$+$}; 
\node at (3,1) {$\displaystyle\sum_{a+b+c=-1}$};
\draw[postaction={decorate,decoration={markings,
    mark=at position .5 with {\arrow[scale=1.3]{<}}}},very thick] (4,0) to[out=90,in=-180] (4.5,.5) to[out=0,in=90] node [pos=.6, fill,circle,inner sep=2.5pt,label=above:{$a$}] {} (5,0);  
\draw[postaction={decorate,decoration={markings,
    mark=at position .5 with {\arrow[scale=1.3]{>}}}},very thick] (4,2) to[out=-90,in=-180] (4.5,1.5) to[out=0,in=-90] node [pos=.6, fill,circle,inner sep=2.5pt,label=below:{$c$}] {} (5,2);
\draw[postaction={decorate,decoration={markings,
    mark=at position .5 with {\arrow[scale=1.3]{<}}}},very thick] (5.5,1) circle (10pt);
\node [fill,circle,inner sep=2.5pt,label=right:{$b$},right=10.5pt] at (5.5,1) {};
\node[scale=1.5] at (6.5,1){$\la$};
\end{tikzpicture}};

\end{tikzpicture}}};
\node at (-3,3){\scalebox{.95}{\begin{tikzpicture} [scale=1.3] 
\node[scale=1.5] at (-.7,1){$\la$};
\draw[postaction={decorate,decoration={markings,
    mark=at position .5 with {\arrow[scale=1.3]{<}}}},very thick] (0,0) to[out=90,in=-90]  (1,1) to[out=90,in=0]  (.5,1.5) to[out=180,in=90]  (0,1) to[out=-90,in=90]  (1,0);  
\node at (1.5,1.15) {$=$}; 
\node at (2.05,.95) {$\displaystyle\sum_{a+b=-1}$};
\draw[postaction={decorate,decoration={markings,
    mark=at position .5 with {\arrow[scale=1.3]{>}}}},very thick]
(3,0) to[out=90,in=180]  (3.5,.5) to[out=0,in=90]  node
[pos=.7,fill=black,circle,label={[label distance=5pt] right:{$a$}},inner sep=2.5pt]{} (4,0); 

\draw[postaction={decorate,decoration={markings,
    mark=at position .5 with {\arrow[scale=1.3]{<}}}},very thick] (3.5,1.3) circle (10pt);
\node [fill,circle,inner sep=2.5pt,label=right:{$b$},right=10.5pt] at (3.5,1.3) {};
\node[scale=1.5] at (4.7,1){$\la$};
\end{tikzpicture}}};
\node at (-3,6){\scalebox{.95}{\begin{tikzpicture} [scale=1.3] 
\node[scale=1.5] at (-.7,1){$\la$};
\draw[postaction={decorate,decoration={markings,
    mark=at position .5 with {\arrow[scale=1.3]{>}}}},very thick] (0,0) to[out=90,in=-90]  (1,1) to[out=90,in=0]  (.5,1.5) to[out=180,in=90]  (0,1) to[out=-90,in=90]  (1,0);  
\node at (1.5,1.15) {$= \,-$}; 
\node at (2.05,.95) {$\displaystyle\sum_{a+b=-1}$};
\draw[postaction={decorate,decoration={markings,
    mark=at position .5 with {\arrow[scale=1.3]{<}}}},very thick]
(3,0) to[out=90,in=180]  (3.5,.5) to[out=0,in=90]  node
[pos=.7,fill=black,circle,label={[label distance=5pt] right:{$a$}},inner sep=2.5pt, outer sep=3pt]{} (4,0); 
\draw[postaction={decorate,decoration={markings,
    mark=at position .5 with {\arrow[scale=1.3]{>}}}},very thick] (3.5,1.3) circle (10pt);
\node [fill,circle,inner sep=2.5pt,label=right:{$b$},right=10.5pt] at (3.5,1.3) {};
\node[scale=1.5] at (4.7,1){$\la$};
\end{tikzpicture}}};
\end{tikzpicture}

\caption{``Cross and cap'' relations; all strands are colored with $\al_i$.  By convention, a negative number of dots on a strand which is not closed into a bubble is 0.}
\label{pop-rels}
\end{figure}

\begin{figure}[h!] \scalebox{.9}{
\begin{tikzpicture} [scale=1.2]
\node at (0,0){\begin{tikzpicture} [scale=1.3]
\node[scale=1.5] at (-.7,1){$\la$};
\draw[postaction={decorate,decoration={markings,
    mark=at position .5 with {\arrow[scale=1.3]{<}}}},very thick] (0,0) to[out=90,in=-90] node[at start,below]{$i$} (1,1) to[out=90,in=-90] (0,2) ;  
\draw[postaction={decorate,decoration={markings,
    mark=at position .5 with {\arrow[scale=1.3]{>}}}},very thick] (1,0) to[out=90,in=-90] node[at start,below]{$j$} (0,1) to[out=90,in=-90] (1,2);
\end{tikzpicture}};

\node at (1.7,0) {$=$};
\node[scale=1.1] at (2.3,0) {$t_{ij}$};
\node at (3.7,0){\begin{tikzpicture} [scale=1.3]

\node[scale=1.5] at (2.4,1){$\la$};

\draw[postaction={decorate,decoration={markings,
    mark=at position .5 with {\arrow[scale=1.3]{<}}}},very thick] (1,0) to[out=90,in=-90]  node[at start,below]{$i$} (1,2);  
\draw[postaction={decorate,decoration={markings,
    mark=at position .5 with {\arrow[scale=1.3]{>}}}},very thick] (1.7,0) to[out=90,in=-90] node[at start,below]{$j$} (1.7,2);
\end{tikzpicture}};

\node at (0,-3){\begin{tikzpicture} [scale=1.3]
\node[scale=1.5] at (-.7,1){$\la$};
\draw[postaction={decorate,decoration={markings,
    mark=at position .5 with {\arrow[scale=1.3]{>}}}},very thick] (0,0) to[out=90,in=-90] node[at start,below]{$i$} (1,1) to[out=90,in=-90] (0,2) ;  
\draw[postaction={decorate,decoration={markings,
    mark=at position .5 with {\arrow[scale=1.3]{<}}}},very thick] (1,0) to[out=90,in=-90] node[at start,below]{$j$} (0,1) to[out=90,in=-90] (1,2);
\end{tikzpicture}};

\node at (1.7,-3) {$=$};
\node[scale=1.1] at (2.3,-3) {$t_{ji}$};
\node at (3.7,-3){\begin{tikzpicture} [scale=1.3]

\node[scale=1.5] at (2.4,1){$\la$};

\draw[postaction={decorate,decoration={markings,
    mark=at position .5 with {\arrow[scale=1.3]{>}}}},very thick] (1,0) to[out=90,in=-90]  node[at start,below]{$i$ }(1,2) ;  
\draw[postaction={decorate,decoration={markings,
    mark=at position .5 with {\arrow[scale=1.3]{<}}}},very thick] (1.7,0) to[out=90,in=-90]  node[at start,below]{$j$} (1.7,2);
\end{tikzpicture}};

\end{tikzpicture}}
\caption{The cancellation of oppositely oriented crossings with different labels.}
\label{opp-cancel}
\end{figure}

\item Oppositely oriented crossings of differently colored strands
  simply cancel with a scalar, shown in Figure \ref{opp-cancel}.

\item the endomorphisms of words only using $\eF_i$ (or by duality only $\eE_i$'s) satisfy the relations of the {\bf quiver Hecke algebra} $R$, shown in Figure \ref{quiver-hecke}.

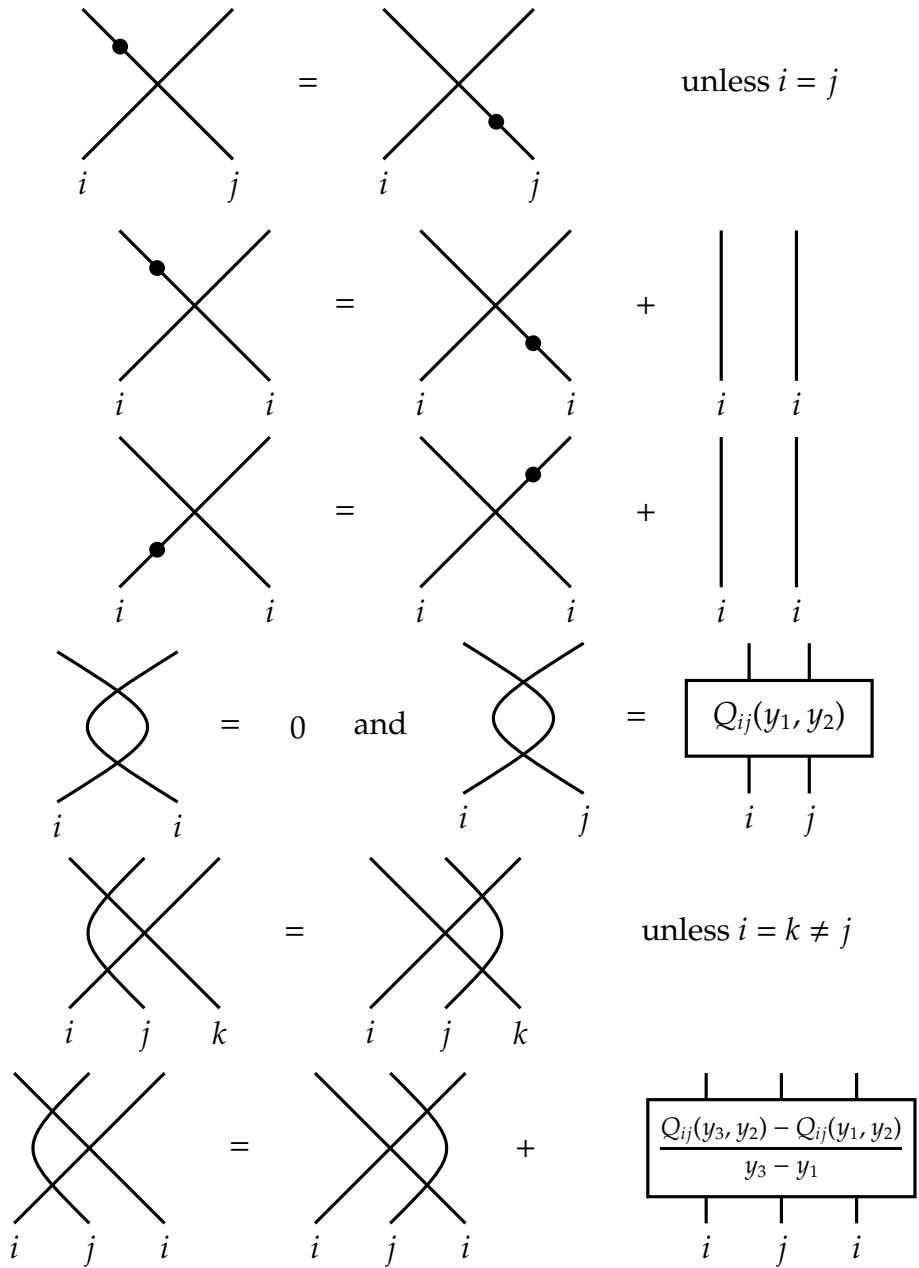
\begin{figure}[h!]
\begin{equation*}
    \begin{tikzpicture}[scale=1]
      \draw[very thick](-4,0) +(-1,-1) -- +(1,1) node[below,at start]
      {$i$}; \draw[very thick](-4,0) +(1,-1) -- +(-1,1) node[below,at
      start] {$j$}; \fill (-4.5,.5) circle (3pt);
      \node at (-2,0){=}; \draw[very thick](0,0) +(-1,-1) -- +(1,1)
      node[below,at start] {$i$}; \draw[very thick](0,0) +(1,-1) --
      +(-1,1) node[below,at start] {$j$}; \fill (.5,-.5) circle (3pt);
      \node at (4,0){unless $i=j$};
    \end{tikzpicture}
  \end{equation*}
  \begin{equation*}
    \begin{tikzpicture}[scale=1]
      \draw[very thick](-4,0) +(-1,-1) -- +(1,1) node[below,at start]
      {$i$}; \draw[very thick](-4,0) +(1,-1) -- +(-1,1) node[below,at
      start] {$i$}; \fill (-4.5,.5) circle (3pt);
      \node at (-2,0){=}; \draw[very thick](0,0) +(-1,-1) -- +(1,1)
      node[below,at start] {$i$}; \draw[very thick](0,0) +(1,-1) --
      +(-1,1) node[below,at start] {$i$}; \fill (.5,-.5) circle (3pt);
      \node at (2,0){+}; \draw[very thick](4,0) +(-1,-1) -- +(-1,1)
      node[below,at start] {$i$}; \draw[very thick](4,0) +(0,-1) --
      +(0,1) node[below,at start] {$i$};
    \end{tikzpicture}
  \end{equation*}
 \begin{equation*}
    \begin{tikzpicture}[scale=1]
      \draw[very thick](-4,0) +(-1,-1) -- +(1,1) node[below,at start]
      {$i$}; \draw[very thick](-4,0) +(1,-1) -- +(-1,1) node[below,at
      start] {$i$}; \fill (-4.5,-.5) circle (3pt);
      \node at (-2,0){=}; \draw[very thick](0,0) +(-1,-1) -- +(1,1)
      node[below,at start] {$i$}; \draw[very thick](0,0) +(1,-1) --
      +(-1,1) node[below,at start] {$i$}; \fill (.5,.5) circle (3pt);
      \node at (2,0){+}; \draw[very thick](4,0) +(-1,-1) -- +(-1,1)
      node[below,at start] {$i$}; \draw[very thick](4,0) +(0,-1) --
      +(0,1) node[below,at start] {$i$};
    \end{tikzpicture}
  \end{equation*}
  \begin{equation*}
    \begin{tikzpicture}[very thick,scale=1]
      \draw (-2.8,0) +(0,-1) .. controls +(1.6,0) ..  +(0,1)
      node[below,at start]{$i$}; \draw (-1.2,0) +(0,-1) .. controls
      +(-1.6,0) ..  +(0,1) node[below,at start]{$i$}; \node at (-.5,0)
      {=}; \node at (0.4,0) {$0$};
\node at (1.5,.05) {and};
    \end{tikzpicture}
\hspace{.4cm}
    \begin{tikzpicture}[very thick,scale=1]

      \draw (-2.8,0) +(0,-1) .. controls +(1.6,0) ..  +(0,1)
      node[below,at start]{$i$}; \draw (-1.2,0) +(0,-1) .. controls
      +(-1.6,0) ..  +(0,1) node[below,at start]{$j$}; \node at (-.5,0)
      {=};

\draw (1.8,0) +(0,-1) -- +(0,1) node[below,at start]{$j$};
      \draw (1,0) +(0,-1) -- +(0,1) node[below,at start]{$i$}; 
\node[inner xsep=10pt,fill=white,draw,inner ysep=8pt] at (1.4,0) {$Q_{ij}(y_1,y_2)$};
    \end{tikzpicture}
  \end{equation*}
  \begin{equation*}
    \begin{tikzpicture}[very thick,scale=1]
      \draw (-3,0) +(1,-1) -- +(-1,1) node[below,at start]{$k$}; \draw
      (-3,0) +(-1,-1) -- +(1,1) node[below,at start]{$i$}; \draw
      (-3,0) +(0,-1) .. controls +(-1,0) ..  +(0,1) node[below,at
      start]{$j$}; \node at (-1,0) {=}; \draw (1,0) +(1,-1) -- +(-1,1)
      node[below,at start]{$k$}; \draw (1,0) +(-1,-1) -- +(1,1)
      node[below,at start]{$i$}; \draw (1,0) +(0,-1) .. controls
      +(1,0) ..  +(0,1) node[below,at start]{$j$}; \node at (5,0)
      {unless $i=k\neq j$};
    \end{tikzpicture}
  \end{equation*}
  \begin{equation*}
    \begin{tikzpicture}[very thick,scale=1]
      \draw (-3,0) +(1,-1) -- +(-1,1) node[below,at start]{$i$}; \draw
      (-3,0) +(-1,-1) -- +(1,1) node[below,at start]{$i$}; \draw
      (-3,0) +(0,-1) .. controls +(-1,0) ..  +(0,1) node[below,at
      start]{$j$}; \node at (-1,0) {=}; \draw (1,0) +(1,-1) -- +(-1,1)
      node[below,at start]{$i$}; \draw (1,0) +(-1,-1) -- +(1,1)
      node[below,at start]{$i$}; \draw (1,0) +(0,-1) .. controls
      +(1,0) ..  +(0,1) node[below,at start]{$j$}; \node at (2.8,0)
      {$+$};        \draw (6.2,0)
      +(1,-1) -- +(1,1) node[below,at start]{$i$}; \draw (6.2,0)
      +(-1,-1) -- +(-1,1) node[below,at start]{$i$}; \draw (6.2,0)
      +(0,-1) -- +(0,1) node[below,at start]{$j$}; 
\node[inner ysep=8pt,inner xsep=5pt,fill=white,draw,scale=.8] at (6.2,0){$\displaystyle \frac{Q_{ij}(y_3,y_2)-Q_{ij}(y_1,y_2)}{y_3-y_1}$};
    \end{tikzpicture}
  \end{equation*}
\caption{The relations of the quiver Hecke algebra.  These relations are insensitive to labeling of the plane.}
\label{quiver-hecke}
\end{figure}
\end{itemize}

This categorification has analogues of the positive and negative
Borels given by the representations of {\bf quiver Hecke algebras},
the algebra given by diagrams where all strands are oriented downwards,
modulo the relations in Figure \ref{quiver-hecke}, which is discussed
in \cite[\S 4]{Rou2KM} and an earlier paper of Khovanov and Lauda
\cite{KLI}. We denote  these 2-categories $\tU^+$ and $\tU^-$.
\clearpage

\subsection{Categorifications for parabolics}

For our purposes, it will be crucial to have a nondegeneracy result for $\tU$; the most important consequence of this will be that the quiver Hecke algebra injects into $\End_{\tU}(\oplus_{\Bi}\eF_{\Bi}\mu)$ for any weight $\mu$. 
Luckily, we know such results for $\mathfrak{sl}_2$, and for the Borel
$\mathfrak{b}_-$ by work of Lauda \cite{LauSL21} and Khovanov-Lauda
\cite{KLI}, with independent proofs given by Rouquier
\cite[Proposition 5.15 \& Proposition 3.12]{Rou2KM}.  Since a
Kac-Moody algebra is essentially a bunch of $\mathfrak{sl}_2$'s with
their interactions described by a Borel, we can hope that these cases
can lead us to the more general case.

Let us first give a rough sketch of the argument:
\begin{itemize}
\item First, we construct an auxilliary 2-category which corresponds to
  a categorification of a minimal parabolic $\mathfrak{b}_-+\C\cdot E_i$.  This
  category contains a copy of Lauda's categorification of
  $\mathfrak{sl}_2$ and of $\tU_i$.  A variant of Lauda's
  non-degeneracy argument works for this category.
\item We can use this non-degeneracy argument to show that the
  projective modules over the cyclotomic quotient are a quotient of an
  obvious 2-representation of 
  this category, and thus also carry an action of it.  This establishes that $\fE_i$ and $\fF_i$ define an action of
  $\tU_{\mathfrak{sl}_2}$; in particular, these functors are
  biadjoint.
\item We then need only check one extra relation to confirm that we
  have an action of all $\tU$ on the projective modules over the
  cyclotomic quotient; this action can be used to confirm
  non-degeneracy for the whole of $\tU$. 
\end{itemize}

In order to follow though on this argument, we consider a new category categorifying the parabolic generated by $\mathfrak{b}_-$ and $E_i$, for a fixed index $i$ (which we leave fixed for the remainder of this section).
\begin{defn}
We let $\tU_i^-$ be the $2$-category whose \begin{itemize}
\item objects are weights of $\fg$,
\item 1-morphisms are compositions of 1-morphisms in $\mathcal{U}^-$ and the single 1-mor\-phism $\eE_i$ from $\tU^+$,
\item 2-morphisms are a quotient of the $\K$-span of string diagrams
  of the form used in $\tU$ in which only $i$-colored strands are
  allowed to go downwards. The relations killed are exactly those from $\tU$ that relate
  such diagrams.
\end{itemize}  
\end{defn} 
In Rouquier's language, we would construct this category by adjoining $\eE_i$ to the lower half categorification as a formal left adjoint to $\eF_i$, and impose the relations that
\begin{itemize}
\item the map $\rho_{s,\la}$ is an isomorphism whose inverse is described by the lower relation in Figure \ref{pop-rels} (in the ``style'' of Rouquier, one would not impose this equation, but simply adjoin an inverse to $\rho_{s,\la}$).
\item the right adjunction between  $\eF_i$ and $\eE_i$ is given by the upper relation of Figure \ref{pop-rels}.
\end{itemize} 

There a functor $\tU_i^-\to \tU$, which is not manifestly faithful,
since new relations could appear when the other objects are added.  We
note that the 2-morphisms in this category have a spanning set
given by the set $B_{\Bi,\Bj,\la}$ defined in \cite[\S 3.2.3]{KLIII}
for the morphisms allowed in $\tU_i^-$,
which we will denote $B_i$.  Let us briefly recall this definition:
\begin{defn}
  For any two 1-morphisms $G$ and $H$ given by words in the $\eF_j$'s
  and $\eE_i$'s, we let $B_{G,H,i}$ denote the set of 2-morphisms
  given by
  \begin{itemize}
  \item for each perfect matching on the collection of all the glyphs in the words $G$ and $H$
    which can only matches
    \begin{itemize}
    \item an $\eF_j$ or $\eE_i$ in $G$ to one in $H$, or 
\item an $\eF_i$ to an $\eE_i$ within the same word
    \end{itemize}
  we choose an arbitrary 2-morphism which connects the matched dots
  without any self-intersection or any strands intersecting twice.  We
  fix an arbitrary point on each strand in this 2-morphism
\item we let $B_{G,H,i}$ be the set obtained by multiplying these
  chosen 2-morphisms by an arbitrary monomial in the bubbles at the
  left and by an arbitrary number of dots at each of the fixed
  location on the strands.
  \end{itemize}
The set $B_i$ is union of these sets over all 1-morphisms $G$ and $H$.
\end{defn}

Just as Lauda's categorification of $\mathfrak{sl}_2$ acts on a ``flag
category,'' this parabolic categorification acts on a ``quiver flag
category,'' which can be thought of as arising from Zheng's
construction \cite{Zheng2008} if one only quotients out by the thick
subcategory for the vertex $i$.  While this geometric perspective can
be made precise for symmetric Kac-Moody algebras, we wish to give a proof for all symmetrizable types,
and thus will give a completely algebraic construction.

For ease, we let $m_\mu^j=\om_j^\vee(\la-\mu)$, where
$\om_i^\vee:\rola\to \Z$ is the linear function sending $\al_i$ to 1,
and all other simple roots to 0.  As usual, we let $\Lambda(\bp)$ be
the algebra of symmetric polynomials on an alphabet $\bp$, and let $e_i(\bp),h_i(\bp)$ denote the elementary and complete symmetric polynomials of degree $i$.  Let $$\tilde{\bLa}_\mu \cong\bigotimes_{j\in \Gamma}\Lambda(p_{j,1},\dots,p_{j,m_\mu^j}).$$
Now consider the polynomial in $\tilde{\bLa}_\mu$ given by 
\[\Xi_\mu(\bp,t)= \left(\sum_{k=0}^{\infty}
  h_{k}(\bp_i)(-t)^{k}\right)\prod_{j\neq i}\prod_{k=0}^{m_\mu^j}
t_{ij}^{-1}\cdot t^{-c_{ji}}Q_{ji}(p_{j,k},-t),\] 
where $\bp_i$ denotes the alphabet of variables $p_{i,*}$.

We let $\bLa_\mu$ be the quotient of $\tilde{\bLa}_\mu $ by the relations:
\begin{equation*}\label{grass-rels}
\Xi_\mu\{t^g\}=0\qquad  \text{ for all } g> \mu^i+m_{\mu}^j  
\end{equation*} 
 Here $f(t)\{t^g\}$ denotes the $t^g$ coefficient of a polynomial.  We note that these are quite reminiscent of the relations in a Grassmannian $\mathrm{Gr}(n,m)$, which are simply that $h_k(\bp)=0$ for all $k>n-m$.  In the symmetric case, for a specific choice of $Q_{ij}$, the ring $\bLa_\mu$ is the cohomology ring of a Grassmannian bundle over a module space of quiver representations, and these constructions can be interpreted geometrically.
\begin{defn}
The ``quiver flag category'' $\mathcal{G}_\la$ is a 2-category that sends each weight $\mu$ to the category of modules over $\bLa_\mu$
with 1-morphisms given by the categories of functors between
$\bLa_\mu \modu$ isomorphic to tensoring with a bimodule between the
corresponding algebras
algebras, and 2-morphisms given by natural transformations isomorphic
to bimodule morphisms.
\end{defn}
It may seem rather strange to use ``functors isomorphic to tensoring
with a bimodule'' here; the point is that this is a strict 2-category,
whereas considering the bimodules themselves would be a weak one.  
\begin{thm}\label{non-degenerate}
There is a strict 2-functor $\tU_i^-\to \mathcal{G}_\la$, and every non-trivial linear combination of elements of $B_i$ in $\tU_i^-$ acts non-trivially in one of these categories.  That is, $\tU_i^-$ is non-degenerate in the sense of Khovanov-Lauda.
\end{thm}
\begin{proof}
First, we describe the action on the level of 1-morphisms.
\begin{itemize}

\item The functors $\eF_j$ for $j\neq i$ act by tensoring with the $\bLa_\mu\operatorname{-}\bLa_{\mu-\al_i}$ bimodule $\bLa_{\mu}[p_{j,m_{\mu}^j+1}]$.
The left-module structure over $\bLa_\mu$ is the obvious one, and right-module over $\bLa_{\mu-\al_j}$ is a slight tweak of this: $e_k(\bp_j')$ acts by $e_k(\bp_j,p_{j,m_{\mu}^j+1})$, $e_k(\bp_m')$ by $e_k(\bp_m)$ for $m\neq j$.
\item The functor $\eF_i$ acts by an analogue of the action in Lauda's paper \cite{LauSL21}; tensor product with a natural $\bLa_\mu\operatorname{-}\bLa_{\mu-\al_i}$-bimodule $\bLa_{\mu;i}$ which is a quotient of $\bLa_\mu[p_{i,m^i_\mu+1}]$ by the relation
 \begin{equation}\label{bim-rel}
 \left(\sum_{c=0}^\infty (-p_{i,m^i_\mu+1}t)^c\right) \Xi_\mu\{t^g\}=0\qquad  \text{ for all } g> \mu^i+m_{\mu}^j -1 \end{equation}
with the same left and right actions as above.
\item Similarly, the functor $\eE_i$ acts by tensor product with $\dot{\bLa}_{\mu+\al_i;i}$, the bimodule defined above with the actions above reversed.  This can also be presented as a quotient of $\bLa_{\mu}[p_{i,m^i_\mu}]$ by the relation 
\begin{equation*}
 \left(1+ p_{i,m^i_\mu+1}t\right) \Xi_\mu\{t^g\}=0\qquad  \text{ for all } g> \mu^i+m_{\mu}^j.\end{equation*}
\end{itemize}
If we only consider $\eE_i$'s and $\eF_i$'s, then we obtain a sum of specializations of Lauda's construction of a representation of $\tU_{\mathfrak{sl}_2}$ on the equivariant cohomology of Grassmannians.  That is, for each fixed choice of $m^j_\mu$ for $i\neq j$, we realize the functors along the $\mathfrak{sl}_2$ weight-string of $\eta=\la-\sum m^j_\mu\al_j$ by extending scalars from Lauda's construction by the map $H^*_{GL_\infty}(\mathrm{Gr}(m^i_\mu,\infty);\K)\to \bLa$ given by sending $$x_k\mapsto e_k(\bp_i)\qquad y_k\mapsto\Xi_\mu(t)\{t^k\}.$$

Clearly, we have $$\bLa_\mu\cong H^*_{GL_\infty}(\mathrm{Gr}(m^i_\mu,\infty);\K)\otimes_{H^*_{GL_{\eta^i}}(\mathrm{Gr}(m^i_\mu,\eta^i);\K)}\bLa.$$
This allows us to define all necessary 2-morphisms between $\eF_i$'s and $\eE_i$'s, which automatically satisfy all the appropriate relations by \cite[Theorem 4.13]{LauSL2}.  

On the other hand, 2-morphisms between $\fF_j$'s other than $i$ act as in Khovanov and Lauda \cite{KLII} or Rouquier \cite[Proposition 3.12]{Rou2KM}. Similarly, the proof of relations follows over immediately.  Thus, the only issue is the interaction between these 2 classes of functors.

In particular, it remains to show the maps corresponding to elements of  $R(\nu)$ are well defined (the relations between them then automatically hold, since quotienting out by relations will not cause two things to become unequal).  

Now, consider the bimodules  $\bLa_{\mu;i}\otimes_{\bLa_{\mu-\al_i}}
\bLa_{\mu-\al_i;j}$ and $
\bLa_{\mu;j}\otimes_{\bLa_{\mu-\al_j}}\bLa_{\mu-\al_j;i}$.  The
functors of tensor with these are canonically isomorphic to
$\eF_i\eF_j$ and $\eF_j\eF_i$, respectively (though the are not the
same ``on the nose''), so it suffices to define the map $\psi$ as a
map between these bimodules.  The former is just $\bLa_{\mu;i}[ p_{j,m^j_\mu+1}]$, so the relations are just (\ref{bim-rel}).  

The latter is a quotient of $\bLa_{\mu}[p_{j,m^j_\mu+1},p_{i,m^i_\mu+1}]$ by \[ t^{-c_{ji}}Q_{ji}(p_{j,m_\mu^j+1},-t^{-1})\left(\sum_{c=0}^\infty (-p_{i,m^i_\mu+1}t)^c\right) \Xi_\mu\{t^g\}=0\qquad  \text{ for all } g> \mu^i+m_{\mu}^j -1-c_{ij}.\] Modulo the relations (\ref{grass-rels}) of $\bLa_\mu$  this polynomial is congruent to \[t^{-c_{ji}}Q_{ji}(p_{j,m_\mu^j+1},p_{i,m^i_\mu+1})\left(\sum_{c=0}^\infty (-p_{i,m^i_\mu+1}t)^c\right) \Xi_\mu,\]
 so the new relations introduced are exactly $Q_{ji}(p_{j,m_\mu^j+1},p_{i,m^i_\mu+1})$ times those of $\bLa_{\mu;i}[ p_{j,m^j_\mu+1}]$.

Thus, the usual definition of $\psi$ from Khovanov and Lauda indeed induces a map of modules, as long as we are careful to use the convention that $e(j,i)\psi$ corresponds to the identity map (in \cite{KLI}, this is the switch map for two variables, since they do not index the variables for different colors separately) and $e(i,j)\psi$ corresponds to multiplication by $Q_{ji}(p_{j,m_\mu^j+1},p_{i,m^i_\mu+1})$.

Let us illustrate this point in the simplest case, when $\mu=\la$.  \begin{align*}\bLa_{\la}&=\K, &\bLa_{\la-\al_i}&=\K[p_i]/(p_i^{\al_i^\vee(\la)})\\
\bLa_{\la-\al_j}&=\K[p_j] &\bLa_{\la-\al_i-\al_j}&=\K[p_i,p_j]/(p_i^{\al_i^\vee(\la)}Q_{ji}(p_j,p_i))\end{align*}

The only one of these requiring any appreciable computation is the last.  In this case, we have the relation $p_i^{\la^i}Q_{ji}(p_j,p_i)=0$ by relating the $t^{\langle \la-\al_j-\al_i,\al_i\rangle}+2d_i$ term of $(1-p_it+\cdots)t^{-c_{ji}}Q_{ji}(p_j,-t^{-1})$. 

Finally, we must prove the relation shown in Figure \ref{opp-cancel}.  This is simply a calculation, given that we have already defined the morphisms for all the diagrams which appear.  The composition
\begin{equation}
\eF_j\eE_i\overset{\iota_1}\longrightarrow \eE_i\eF_i\eF_j\eE_i\overset{\psi_2}\longrightarrow \eE_i\eF_j\eF_i\eE_i\overset{\ep_3}\longrightarrow \eE_i\eF_j\label{eq:FE}
\end{equation}
is given by 
\begin{align*}
\ep_3\psi_2\iota_1(p_{i,m_\mu^i}^a\otimes p_{j,m_\mu^j+1}^b)&= \ep_3\psi_2\left(\sum_{k=0}^{m_\mu^i-1}p_{i,m_\mu^i}^a\otimes p_{j,m_\mu^j+1}^b \otimes p_{i,m_\mu^i}^{m_\mu^i-k+1}\otimes h_{k}(\bp_i') \right)\\
&=\ep_3\left(\sum_{k=0}^{m_\mu^i-1}p_{i,m_\mu^i}^a\otimes p_{i,m_\mu^i}^{m_\mu^i-k+1} \otimes p_{j,m_\mu^j+1}^b \otimes h_{k}(\bp_i')\right)\\
&=\sum_{k=0}^{a} (-1)^k p_{j,m_\mu^j+1}^b \otimes e_{a-k}(\bp_i) h_{k}(\bp_i')\\
&=p_{j,m_\mu^j+1}^b\otimes p_{i,m_\mu^i}^a 
\end{align*}

Now, note that by our assumptions on $Q_{ij}$, the power series
$\Xi(t)$ has a non-zero constant term, and thus has a formal inverse
in $\Lambda(\bp)[[t]]$, which we denote $\Xi^{-1}(t)$.  By the usual
Cauchy formula, we have \[\Xi^{-1}(t)=\left(\sum_{k=0}^{\infty}
  e_{k}(\bp_i)t^{k}\right)\prod_{j\neq i}\prod_{k=0}^{m_\mu^j}
\frac{t_{ij}\cdot t^{c_{ji}}}{Q_{ji}(p_{j,k},-t)}, \]
and by \cite[Definition 3.1]{LauSL21},  
$$X_k\mapsto \Xi^{-1}(t)\{t^k\}\qquad Y_k\mapsto (-1)^k h_k(\bp_i).$$  

The composition
\begin{equation}
\eE_i\eF_j\overset{\iota_3'}\longrightarrow \eE_i\eF_j\eF_i\eE_i\overset{\psi_2}\longrightarrow \eE_i\eF_i\eF_j\eE_i\overset{\ep_1'}\longrightarrow \eF_j\eE_i\label{eq:EF}
\end{equation}
is given by 
\begin{align*}
\ep_1'\psi_2\iota_3'(p_{j,m_\mu^j+1}^b\otimes p_{i,m_\mu^i}^a )&= \ep_1'\psi_2\left(\sum_{k=0}^{m_\mu^i-1} (-1)^k \Xi(\bp_i',t)\{t^k\}\otimes p_{i,m_\mu^i}^{m_\mu^i-k+1} \otimes p_{j,m_\mu^j+1}^b \otimes p_{i,m_\mu^i}^a  \right)\\
&=\ep_1'\left(\sum_{k=0}^{m_\mu^i-1} (-1)^k \Xi(\bp_i',t)\{t^k\} \otimes p_{j,m_\mu^j+1}^b Q_{ji}(p_{j,m_\mu^j+1},p_{i,m^i_\mu+1})\otimes p_{i,m_\mu^i}^{m_\mu^i-k+1}  \otimes p_{i,m_\mu^i}^a 
\right)\\
&=\sum_{k=0}^{a} (-1)^k \Xi(\bp_i',t)\{t^k\}\cdot \Xi(\bp_i,t)^{-1}Q_{ji}(p_{j,m_\mu^j+1},-t)\{t^{a-k-c_{ij}}\}   \otimes    p_{j,m_\mu^j+1}^b \\
&=\frac{t_{ij}}{1-p_{i,m_\mu^i}t}\{t^a\}\otimes p_{j,m_\mu^j+1}^b\\
&=t_{ij}\cdot p_{i,m_\mu^i}^a \otimes p_{j,m_\mu^j+1}^b 
\end{align*}
Thus, composing the maps \eqref{eq:FE} and \eqref{eq:EF} in either
order gives $t_{ij}$ times the identity, confirming the relation of
Figure \ref{opp-cancel}.

If there is any pair of 1-morphisms where the set $B_i$ is not a basis
(i.e. it has non-trivial relation), then using the biadjunction of
$\eF_i$ and $\eE_i$ and the commutation relations, we can find a pair
of such morphisms where only $\eF_i$'s are used.  In this case, the
functor $\eF_{\Bi}$ corresponds to outer tensor with a polynomial
ring, followed by modding out the appropriate ideal, where morphisms
in $\tU^-$ act on the polynomial ring by the usual polynomial
representation of the KLR algebra.  

No linear combination in $B_i$ acts trivially before modding out by
this ideal.  Furthermore, if $\la$ is sufficiently large, then we can
assure that all relations in $\bLa_\mu$ are of arbitrarily large
degree, so any linear combination of diagrams in $B_i$ can be killed
for all $\la$ by degree reasons.
\end{proof}

\subsection{Cyclotomic quotients}
\label{sec:cyc}
Now that we understand how to add the adjoint of one of the $\eF_i$'s
to $\tU^-$, we move towards considering all of them.  Just as with
$\tU^-$ and $\tU^-_i$, we prove non-degeneracy by constructing a
family of actions which are jointly faithful.  As in the previous
section, $i$ will denote a fixed element of $\Gamma$, and we will use
$j$ for an arbitrary index.

\begin{defn}
  The  {\bf cyclotomic quiver Hecke algebra} $R^\la$ for a weight
  $\lambda$ is the quotient of $R$ by the {\bf cyclotomic ideal}, the
  2-sided ideal generated by the elements $y_1^{\la^{i_1}}e(\Bi)$ for
  all sequences $\Bi$.

We let $\cata^\la$ denote the category of finite dimensional graded $R^\la$-modules. 
\end{defn}
This algebra has attracted great interest recently in the work of
Brundan-Kleshchev \cite{BKKL}, Kleshchev-Ram \cite{KlRa},
Hoffnung-Lauda and Lauda-Vazirani \cite{LV,HL}, 
Hill-Melvin-Mondragon \cite{HMM} and Tingley and the author \cite{TW}.  It has a very rich structure and
representation theory, and some surprising connections to classical
representation theory.  More importantly for our purposes, $\cata^\la$
is a module category over $\tU$, as we will show below.

Consider the map $\nu_j\colon R^\la_{\mu}\to R^\la_{\mu-\al_i}$ that adds a strand
labeled with $j$ at the right.
\begin{defn}
  We let $\fF_j=-\otimes_{R^\la_{\mu}}R^\la_{\mu-\al_i}$ denote the functor of extension
  of scalars by this map; we will refer to this as an {\bf induction}
  functor.

We let $\fE_j=\Hom_{R^\la_{\mu-\al_i}}(R^\la_{\mu},-)(\langle \mu,\al_j\rangle-d_i)$ denote restriction of scalars by
this map (with a grading shift), 
the functors left adjoint to the $\fF_i$'s; we call these {\bf
  restriction functors}.
\end{defn}
It's worth noting that these are graded
differently from the most obvious restriction functors; the presence
of a cup (see Figure~\ref{funcs}) shifts the grading.

The first step to understanding this relation is to realize the
cyclotomic quotient in terms of the category $\tU^-_i$.  Given any
object $\la$ in the 2-category $\tU^-_i$, we have a representation
$\tU^-_i(\la)$ of
this 2-category (i.e. a strict 2-functor to $\mathsf{Cat}$), given by
$\mu\mapsto \Hom(\la,\mu)$ with 1-morphisms giving functors between
these categories and 2-morphisms natural transformations by
composition.   Given any collection $J$ of 2-morphisms closed under
both vertical composition and horizontal composition on the right with
arbitrary morphisms (a ``ideal'' which is 2-sided for the vertical
composition, and 1-sided for horizontal composition), we can consider
the quotient representation $\tU^-_i(\la)/J$ by these 2-morphisms; this is again a
2-functor from $\tU^-_i$ to $\mathsf{Cat}$.  It sends $\mu$ to the
quotient of $\Hom(\la,\mu)/J$, the category whose objects coincide
with those of $\Hom(\la,\mu)$, but where all morphisms in $J$ are
identified with 0 (of course, if the identity morphism of an object is
in $J$, that object is isomorphic to 0 in the quotient category).

\begin{prop}\label{quotient-is-cyc}
  Let $J$ be the smallest set of morphisms
  containing \[\id\colon {\eE_i\la}\to {\eE_i\la}\quad \text{ and }\quad
 y^{\la^j}\colon \eF_j\la\to \eF_{j}\la\quad\text{ for all $j$}\] which is closed under
  both vertical composition and horizontal composition on the right
  with arbitrary morphisms.  The idempotent completion of
  $\tU^-_i(\la)/J$ is equivalent to the category of projective
  $R^\la$-modules, and this equivalence is intertwines the functor
  $\eF_i$ with $\fF_i$ and the functor
  $\eE_i$ with $\fE_i$.

That is, this equivalence induces a strict 2-functor $\tU^-_i\to\mathsf{Cat}$ given by
\begin{align*}
  \mu& \mapsto R^\la_\mu\mathsf{-pmod}\\
  \eF_j&\mapsto \fF_j\\
  \eE_i&\mapsto \fE_i.
\end{align*}
In particular, the functors $\fE_i$ and $\fF_i$ are biadjoint (up to
grading shift) since $\eF_i$ and $\eE_i$ are biadjoint in $\tU_i^-$.
\end{prop}
\begin{proof}
  First, we show that $R^\la$ can also be written as a quotient of the
  larger algebra
  $\tilde{R}^\la=\End_{\tU_i^-}(\bigoplus_{\Bi}\eF_{\Bi})$, again by
  the 2-sided ideal generated by $\id_A \cdot y^{\la^j}:A\eF_i\la\to
  A\eF_j\la$ for all 1-morphisms $A$; we call this ideal ``the
  cyclotomic ideal of $\tilde{R}^\la$.''  This ideal contains all
  positive degree clockwise bubbles at the left of the diagram (since
  all of these carry at least $\la^i$ dots), so the quiver Hecke
  algebra surjects onto the quotient.  On the other hand, if a diagram
  in $\End_{\tU_i^-}(\bigoplus_{\Bi}\eF_{\Bi})$ contains a positive
  degree bubble, it cannot be rewritten by the relations to be an
  element of the quiver Hecke algebra.  Thus, the intersection of the
  cyclotomic ideal in $\tilde R^\la$ with the included copy of $R$ is
  the cyclotomic ideal of that smaller algebra.

We also note that in $\tU^-_i(\la)/J$, every object is a summand of one
of the form $\oplus_{\Bi}\eF_{\Bi}\la$ for some set of $\Bi$'s.  Since
such objects generate under the action of $\tU^-_i$ (after all, $\la$
alone generates), it suffices to show such objects are closed under
the action of $\eE_i$.  We induct on the length of $\Bi$.  If
$\Bi=\emptyset$, then $\eE_i\la=0$ and we are done.  In general, we have
that $\eE_i\eF_{i_n}\eF_{\Bi'}\la$ is a summand of $\eF_{i_n}\eE_i
\eF_{\Bi'}\la$ plus some number of copies of $\eF_{\Bi'}\la$, by the
relations in Figures \ref{pop-rels} and \ref{opp-cancel} (this is
discussed in more detail in \cite[\S 5.7]{LauSL21}).  Thus,
by induction, we are done.

Combining these results, we see that the statement of the theorem is
equivalent to the statement that in $\tU^-_i(\la)/J$, the morphism
space $\Hom_{\tU_i^-/J}(\eF_{\Bi},\eF_{\Bj})$ is isomorphic to $e_\Bi R^\la
e_{\Bj}$ with composition sent to multiplication.  We have a
multiplicative map $e_\Bi\tilde{R}^\la e_{\Bj}\to
\Hom_{\tU_i^-/J}(\eF_{\Bi},\eF_{\Bj})$, and this map sends the cyclotomic
ideal to the indicated subcategory, so it induces a map $e_\Bi R^\la
e_{\Bj}\to \Hom_{\tU_i^-/J}(\eF_{\Bi},\eF_{\Bj})$.  If an element of $R^\la$
is in the kernel of this map, its image in
$\Hom_{\tU_i^-}(\eF_{\Bi},\eF_{\Bj})$ is in $J$. Since $R^\la$
injects into $\Hom_{\tU_i^-}(\eF_{\Bi},\eF_{\Bj})$ by Theorem \ref{non-degenerate}, this element can be rewritten as a sum
of diagrams that factor through $A\eE_i\la$ for some 1-morphism $A$
plus elements of the cyclotomic ideal.  We can assume without loss of
generality that it is a sum of elements of the former form.

Said differently, this 2-morphism can be obtained by starting with a 2-morphism $a\colon\eF_{\Bi}\eF_i(\la+\al_i)\to \eF_{\Bj}\eF_i(\la+\al_i)$, and ``capping off'' the $\eF_i$. We rewrite $a$ in terms of Khovanov and Lauda's spanning set, where we choose reduced expressions for our permutations so that the left-most simple reflection only happens once.

``Capping off,'' we obtain an element where every diagram appearing
has either a clockwise bubble at the far left, or a
loop-de-loop turning leftward.  We can apply the relation of Figure
\ref{pop-rels} to see that it is a sum of elements in the cyclotomic
ideal plus diagrams with a clockwise bubble at the left.  By
the relation of Figure \ref{inv-rels}, every positive degree
clockwise bubble can be written as a  a polynomial in positive
degree counter-clockwise bubbles.  A positive degree counter-clockwise bubble must
carry at least $\la^i$  dots and thus lies in the cyclotomic ideal of
$\tilde{R}^\la$.  

This shows that $\tU_i^-$ acts on the category of projective modules
of $R^\la$ and clearly $\eF_i$ is sent to $\fF_i$.  Since
$\eE_i(\langle\mu,\al_i\rangle-d_i)$
(resp. $\eE_i(-\langle\mu,\al_i\rangle+d_i)$) is left (resp. right)
biadjoint to $\eF_i$ in $\tU_i^-$ (up to shift), $\eE_i$ is sent to
$\fE_i$ by the uniqueness of left adjoints.  This also shows that
$\fE_i(-\langle\mu,\al_i\rangle+d_i)$ is right adjoint to $\fF_i$.
\end{proof}

In particular, this shows that every inclusion $\tU^-\hookrightarrow\tU^-_j$ induces the same action of $\tU^-$ on $\cata^\la$.

Using these biadjunctions, we can interpret any picture of the type Khovanov and Lauda draw where all strands begin and end pointing downward as an element of the cyclotomic quotient.  We note that it is not immediately obvious that this assignment satisfies all of Cautis and Lauda's relations.

Still, this equips $R^\la$ with a map $\tau_\la:R^\la\to \K$ given by closing a diagram at the right (if top and bottom strands match) and interpreting this as an element of $R^\la(0)\cong \K$, as shown in Figure \ref{closing}.  The biadjunction implies that this functional makes $R^\la$ into a Frobenius algebra.

Recall that a {\bf Frobenius} structure on a $\K$-algebra $A$  is a
linear map
$\operatorname {tr}\colon A\to \K$ which kills no left ideal.

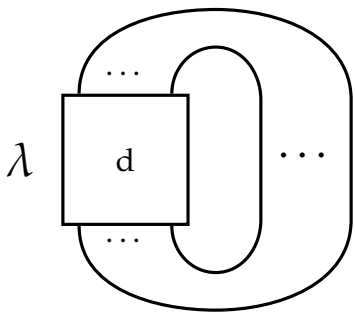
\begin{figure}[ht]
\begin{tikzpicture}[very thick]
\node (a) at (0,0)[draw,inner sep=20pt, label=above:{$\cdots$},label=below:{$\cdots$}] {d}; \node[scale=1.5] at (-1.4,0){$\lambda$};
\draw (a.55) to[out=90,in=180] (1.2,1.5) to[out=0,in=90] (1.8,.8) to[out=-90,in=90] (1.8,-.8) to[out=-90,in=0] (1.2,-1.5) to[out=180,in=-90] (a.-55);
\draw (a.125) to[out=90,in=180] (1.2,2) to[out=0,in=90] (3,.8) to[out=-90,in=90] (3,-.8) to[out=-90,in=0] (1.2,-2) to[out=180,in=-90] (a.-125);
\node[scale=1.4] at (2.4,0){$\cdots$};
  \end{tikzpicture}
\caption{Closing a diagram}
\label{closing}
\end{figure}

\begin{thm}\label{cyc-action}
  The assignment  $\eE_j\mapsto \fE_j, \eF_j\mapsto \fF_j$ gives a strict
  action of $\tU$ on $R^\la_\mu\mathsf{-pmod}$ and thus on
  $\cata^\la$. Any non-trivial linear combination of Khovanov and Lauda's spanning set acts non-trivially on some $\cata^\la$.  In particular, the functors $\fE_j$ and $\fF_j$ are biadjoint and $\tau_\la$ is a Frobenius structure on $R^\la$.
  
As a $U_q(\fg)$-representation, $K_0(R^\la)$ is
  naturally isomorphic to $V_\la^\Z$.
\end{thm}

We should note that this theorem has been independently proven by
Cautis and Lauda \cite[8.1]{CaLa} based on work of Kang and Kashiwara
\cite{KK}.

\begin{rem}
  This Frobenius trace can be easily adjusted to become symmetric. One
  fixes one reference sequence $\Bi_\mu$ for each weight $\mu$; for
  each other sequence $\Bi$, we pick a diagram connecting it to
  $\Bi_\mu$ and for each crossing with and consider the scalar $t(\Bi)$
  which is the product over all crossings in the diagram of
  $t_{ji}/t_{ij}$ where the NE/SW strand of the crossing is labeled
  with $i$ and the NW/SE strand is labeled $j$.  If we multiply the
  trace on $e(\Bi)R^\la e(\Bi)$ by $t(\Bi)$, the result will still be
  Frobenius and be cyclic.

The reader may sensibly ask why we use the trace above instead; it is
in large part so we may match the conventions of \cite{CaLa} and use
their results.  That said, their choice arises very naturally from a
coherent principle: that degree 0 bubbles should be 1.  Trying to
recover cyclicity in $\tU$ will definitely break this condition.
\end{rem}

\begin{proof}[Proof of Theorem \ref{cyc-action}]
We have already established that we have actions of the
categorification of $\mathfrak{sl}_2$ for each simple root and of
$\tU^\pm$, so any relation only involving these subcategories must be
satisfied.  In fact, we already know that any relation only involving
one $\eE_i$ is satisfied.  This leaves exactly one from Khovanov and
Lauda's relations: fixing the double duals of morphisms.  

This is actually equivalent to $\tr$ satisfying the condition that
$\tr(ab)=\tr(ba)$ if $a$ is a diagram only involving dots and
crossings in one color (which we already know from the action of
$\tU_-^i$) and $\tr(\psi\cdot a)=t_{ji}/t_{ij}\cdot
\tr(a\cdot \psi)$ if $\psi$ is crossing with the NE/SW strand labeled
with $i$ and the NW/SE strand labeled $j$, and the latter condition is somewhat simpler to prove (primarily as a matter of organizing induction).  We prove that $\tau_\la$ is symmetric by induction on the number of strands, noting that we already know that $\tau_\la(ab)=\tau_\la(ba)$ if $b$ is a diagram where all dots and crossing only occur in one color.  This establishes the base case of one strand.

We can always use relations in $a$ to assure that the strands at the far right at the top and bottom (if different) cross each other before any other strands.  Thus, if $b$ doesn't cross the rightmost strand, then we can collapse the loop formed when closing $ab$ by crushing the rightmost bubble in $a$. We thus can obtain a diagram $a'$ with fewer strands such that if $b'$ is $b$ with the rightmost strand removed, then $\tr(ab)=\tr(a'b')$ and $\tr(ba)=\tr(b'a')$. Thus, by induction, we have $\tr(ab)=\tr(ba)$.

This reduces us to the case where $b$ is a single crossing of the two
rightmost strands, which may assume are of a different color.  This
separates into 3 cases, grouped by how many the 2 rightmost terminals
at top are connected to  the the 2 rightmost terminals at the bottom;
this is either 0, 1, or 2.  Each of these individual cases is an easy
calculation, which we show in Figure \ref{cycl}. This establishes that
the correct duals hold, and thus that $\tU$ acts on $ R^\la_\mu\mpmod$.
\begin{figure}

\begin{tikzpicture}[xscale=1.1, yscale=.9]
\node at (-5,5.5){
\begin{tikzpicture}[scale=1.3]
\draw[very thick, postaction={decorate,decoration={markings,
    mark=at position .5 with {\arrow[scale=1.3]{<}}}}] (-.5,.1) to [in=-90,out=45]  (.5,1)  to[out=90, in=180] (1,1.5) to [out=0,in=90] (1.5,1) to[out=-90,in=90] node [left,midway]{$i$} (1.5,-.5) to [out=-90,in=0] (1,-1) to [out=180,in=-90] (.5,-.5) to [in=-90,out=90] (0,0)  to [in=-45,out=90] (-1,.9);

\draw[dashed,very thick] (-1,.9) --(-1.3,1.2);
\draw[dashed,very thick] (-.5,.9) --(-.8,1.2);

\draw[dashed,very thick] (-1,.1) --(-1.3,-.2);
\draw[dashed,very thick] (-.5,.1) --(-.8,-.2);

\draw[very thick, postaction={decorate,decoration={markings,
    mark=at position .5 with {\arrow[scale=1.3]{<}}}}] (-1,.1) to [in=-90,out=45]  (0,1)  to[out=90, in=180] (1,2) to [out=0,in=90] (2,1) to[out=-90,in=90] node [right,midway]{$j$} (2,-.5) to [out=-90,in=0] (1,-1.5) to [out=180,in=-90] (0,-.5)  to [in=-90,out=90] (.5,0)  to [in=-45,out=90] (-.5,.9) ;

\draw[dashed] (-.2,0) -- (.7,0);
\draw[dashed] (-.2,1) -- (.7,1);

\node (a) [circle,draw,inner sep=3pt] at  (.07,-.17) {};
\draw[thick,<-, dash pattern=on 1pt off 2pt on 4pt off 2pt] (a.-135) to[out=-140,in=140] (.8,-2) ;
\end{tikzpicture}

};
\node at (0,7){
\begin{tikzpicture}[scale=1.3]

\draw[very thick, postaction={decorate,decoration={markings,
    mark=at position .3 with {\arrow[scale=1.3]{<}}}}] (-1,-.4) to [in=-90,out=45](0,.5) to [in=-90,out=90] (.5,1)  to[out=90, in=180] (1,1.5) to [out=0,in=90] (1.5,1) to[out=-90,in=90] node [left,midway]{$j$} (1.5,-.5) to [out=-90,in=0] (1,-1) to [out=180,in=-90] (.5,-.5)  to [in=-45,out=90] (-.5,.4);

\draw[very thick, postaction={decorate,decoration={markings,
    mark=at position .8 with {\arrow[scale=1.3]{<}}}}] (-.5,-.4) to [in=-90,out=45] (.5,.5) to [in=-90,out=90] (0,1)  to[out=90, in=180] (1,2) to [out=0,in=90] (2,1) to[out=-90,in=90] node [right,midway]{$i$} (2,-.5) to [out=-90,in=0] (1,-1.5) to [out=180,in=-90] (0,-.5) to [in=-45,out=90] (-1,.4);

\draw[dashed,very thick] (-1,.4) --(-1.3,.7);
\draw[dashed,very thick] (-1,-.4) --(-1.3,-.7);

\draw[dashed,very thick] (-.5,.4) --(-.8,.7);

\draw[dashed,very thick] (-.5,-.4) --(-.8,-.7);

\draw[dashed] (-.2,-.5) -- (.7,-.5);
\draw[dashed] (-.2,.5) -- (.7,.5);
\node (a) [circle,draw,inner sep=3pt] at  (.07,.67) {};
\draw[thick,<-, dash pattern=on 1pt off 2pt on 4pt off 2pt] (a.135) to[out=140,in=-140] (.8,2.5) ;

\end{tikzpicture}
};
\node at (-5,-1.5){
\begin{tikzpicture}[scale=1.3]
\draw[very thick, postaction={decorate,decoration={markings,
    mark=at position .8 with {\arrow[scale=1.3]{<}}}}] (.5,-.5) to [in=-90,out=90] (0,0) to [in=-90,out=90] (.5,1)  to[out=90, in=180] (1,1.5) to [out=0,in=90] (1.5,1) to[out=-90,in=90] node [left,midway]{$i$} (1.5,-.5) to [out=-90,in=0] (1,-1) to [out=180,in=-90] (.5,-.5);

\draw[very thick, postaction={decorate,decoration={markings,
    mark=at position .8 with {\arrow[scale=1.3]{<}}}}] (-.5,.1) to [in=-90,out=45]  (0,1)  to[out=90, in=180] (1,2) to [out=0,in=90] (2,1) to[out=-90,in=90] node [right,midway]{$j$} (2,-.5) to [out=-90,in=0] (1,-1.5) to [out=180,in=-90] (0,-.5)  to [in=-90,out=90] (.5,0)  to [in=-45,out=90] (-.5,.9) ;

\draw[dashed] (-.2,0) -- (.7,0);
\draw[dashed] (-.2,1) -- (.7,1);

\node (a) [circle,draw,inner sep=3pt] at  (.07,-.17) {};
\draw[thick,<-, dash pattern=on 1pt off 2pt on 4pt off 2pt] (a.-135) to[out=-140,in=140] (.8,-2) ;

\draw[dashed,very thick] (-.5,.9) --(-.8,1.2);

\draw[dashed,very thick] (-.5,.1) --(-.8,-.2);

\end{tikzpicture}

};
\node at (0,0){
\begin{tikzpicture}[scale=1.3]

\draw[very thick, postaction={decorate,decoration={markings,
    mark=at position .3 with {\arrow[scale=1.3]{<}}}}] (-.5,-.4) to [in=-90,out=45](0,.5) to [in=-90,out=90] (.5,1)  to[out=90, in=180] (1,1.5) to [out=0,in=90] (1.5,1) to[out=-90,in=90] node [left,midway]{$j$} (1.5,-.5) to [out=-90,in=0] (1,-1) to [out=180,in=-90] (.5,-.5)  to [in=-45,out=90] (-.5,.4);

\draw[very thick, postaction={decorate,decoration={markings,
    mark=at position .8 with {\arrow[scale=1.3]{<}}}}] (0,-.5) to [in=-90,out=90] (.5,.5) to [in=-90,out=90] (0,1)  to[out=90, in=180] (1,2) to [out=0,in=90] (2,1) to[out=-90,in=90] node [right,midway]{$i$} (2,-.5) to [out=-90,in=0] (1,-1.5) to [out=180,in=-90] (0,-.5);

\draw[dashed,very thick] (-.5,.4) --(-.8,.7);

\draw[dashed,very thick] (-.5,-.4) --(-.8,-.7);

\draw[dashed] (-.2,-.5) -- (.7,-.5);
\draw[dashed] (-.2,.5) -- (.7,.5);
\node (a) [circle,draw,inner sep=3pt] at  (.07,.67) {};
\draw[thick,<-, dash pattern=on 1pt off 2pt on 4pt off 2pt] (a.135) to[out=140,in=-140] (.8,2.5) ;

\end{tikzpicture}
};
\node at (-5,-8){
\begin{tikzpicture}[scale=1.3]
\draw[very thick, postaction={decorate,decoration={markings,
    mark=at position .8 with {\arrow[scale=1.3]{<}}}}] (.5,-.5) to [in=-90,out=90] (0,0) to [in=-90,out=90] (.5,.5)to [in=-90,out=90] (.5,1)  to[out=90, in=180] (1,1.5) to [out=0,in=90] (1.5,1) to[out=-90,in=90] node [left,midway]{$i$} (1.5,-.5) to [out=-90,in=0] (1,-1) to [out=180,in=-90] (.5,-.5);

\draw[very thick, postaction={decorate,decoration={markings,
    mark=at position .8 with {\arrow[scale=1.3]{<}}}}] (0,-.5) to [in=-90,out=90] (.5,0) to [in=-90,out=90] (0,.5) to [in=-90,out=90] (0,1)  to[out=90, in=180] (1,2) to [out=0,in=90] (2,1) to[out=-90,in=90] node [right,midway]{$j$} (2,-.5) to [out=-90,in=0] (1,-1.5) to [out=180,in=-90] (0,-.5);

\draw[dashed] (-.2,0) -- (.7,0);
\draw[dashed] (-.2,.5) -- (.7,.5);

\node (a) [circle,draw,inner sep=3pt] at  (.07,-.17) {};
\draw[thick,<-, dash pattern=on 1pt off 2pt on 4pt off 2pt] (a.-135) to[out=-140,in=140] (.8,-2) ;
\end{tikzpicture}

};
\node at (0,-7.3){
\begin{tikzpicture}[scale=1.3]
\draw[very thick, postaction={decorate,decoration={markings,
    mark=at position .8 with {\arrow[scale=1.3]{<}}}}] (.5,-.5) to [in=-90,out=90] (.5,0) to [in=-90,out=90] (0,.5) to [in=-90,out=90] (.5,1)  to[out=90, in=180] (1,1.5) to [out=0,in=90] (1.5,1) to[out=-90,in=90] node [left,midway]{$j$} (1.5,-.5) to [out=-90,in=0] (1,-1) to [out=180,in=-90] (.5,-.5);

\draw[very thick, postaction={decorate,decoration={markings,
    mark=at position .8 with {\arrow[scale=1.3]{<}}}}] (0,-.5) to [in=-90,out=90] (0,0) to [in=-90,out=90] (.5,.5) to [in=-90,out=90] (0,1)  to[out=90, in=180] (1,2) to [out=0,in=90] (2,1) to[out=-90,in=90] node [right,midway]{$i$} (2,-.5) to [out=-90,in=0] (1,-1.5) to [out=180,in=-90] (0,-.5);

\draw[dashed] (-.2,0) -- (.7,0);
\draw[dashed] (-.2,.5) -- (.7,.5);
\node (a) [circle,draw,inner sep=3pt] at  (.07,.67) {};
\draw[thick,<-, dash pattern=on 1pt off 2pt on 4pt off 2pt] (a.135) to[out=140,in=-140] (.8,2.5) ;

\end{tikzpicture}
};
\end{tikzpicture}

\caption{Establishing the double dual of $\si_{ij}$.  In each case,
  the proof of double dual is to ``pull'' the indicated strand in the direction of the thin dashed line.}
\label{cycl}
\end{figure}
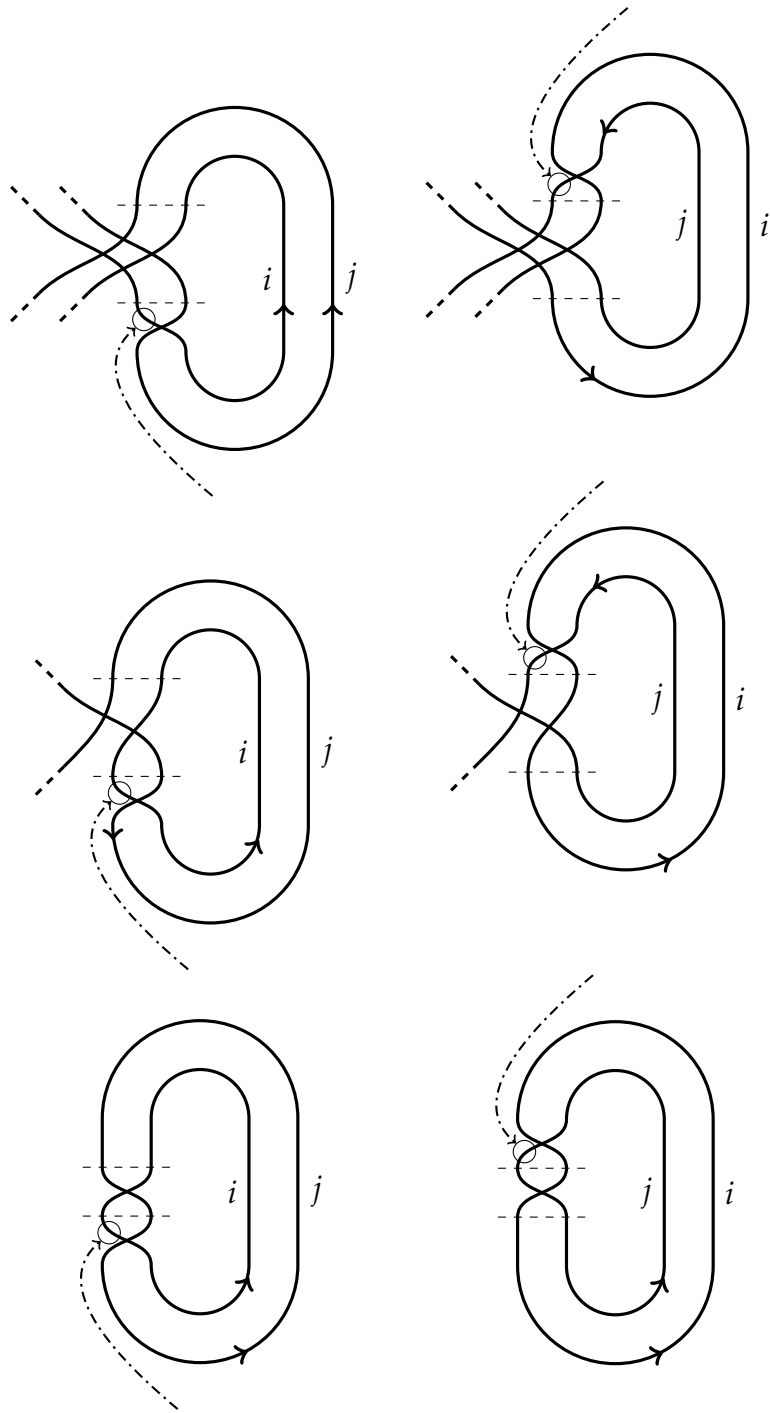

We know that the functors $\fE_j$ and $\fF_j$ extend to all
modules as do the natural transformations defined by 2-morphisms in
$\tU$.  Since every object in $\cata^\la$ has a presentation by
projectives, it is enough to check relations between natural
transformations on the subcategory of projectives.  Thus, these
functors also define an action of $\tU$ on $\cata^\la$.

To show that any non-trivial linear combination of Khovanov and Lauda's spanning set acts non-trivially, it is enough to show that any polynomial in the dots acts non-trivially for some $\la$ (since no element of $R^\la$ kills the polynomial representation).  This, in turn, reduces to the case of a polynomial in positive degree bubbles (we can simply multiply our polynomial in dots by a monomial to assure that each bubble obtain upon closing is positive degree).

Consider the highest degree monomial in the bubbles, and let $\al_i$ be a simple root such that a positive degree bubble colored with $\al_i$ appears in this term.  Let $j$ be the sum of the degrees of the $i$-colored bubbles in this term.  Let $k=\max(1,1-\mu^i)$, and surround this polynomial in bubbles with $k$ bubbles colored with $i$, with the outer one carrying $\mu^i-1$ dots.  This is a non-zero polynomial in bubbles with lower degree.  By induction, we get a non-zero polynomial of 0 degree, i.e. a scalar map $\id_{\la'}\to\id_{\la'}$ for some weight $\la'$.  Thus, we need only choose $\la$ such that the $\la'$-weight space of $\cata^\la$ is non-trivial.

Finally, we must check that $K_0(R^\la)\cong V_\la$. For this, we need only
note that 
\begin{itemize}
\item $K_0(R^\la)$ is generated by a single highest weight vector of weight $\la$.  Thus it is a quotient of the Verma module of highest weight $\la$.
\item On the other hand, $\cata^\la$ is an integrable categorification in the
  sense of Rouquier: acting by $\fF_i$ or $\fE_i$ a sufficiently large
  number of times kills any $R^\la$-module, so $K_0(R^\la)$ is integrable.
\item $V_\la^\Z$ is the only integrable quotient of the the Verma module which is free as a $\Z[q,q^{-1}]$ module.\qedhere
\end{itemize}
\end{proof}

Since no element of $\dot{U}$ kills all finite dimensional
representations, an immediate consequence of this is that
\begin{cor}
  The map $\gamma\colon \dot{U}\to K(\tU)$ defined by Khovanov and
  Lauda in \cite[\S 3.6]{KLIII} is an isomorphism.
\end{cor}

Recall that the {\bf $q$-Shapovalov form}  $\langle-,-\rangle$ is the
unique bilinear form on $V_\la^\Z$ such that \begin{itemize}
\item $\langle v_h,v_h\rangle =1$ for a fixed highest weight vector $v_h$.
\item $\langle u\cdot v,v'\rangle=\langle v,\tau(u)\cdot v'\rangle$ for any $v,v'\in V_\la$ and $u\in U_q(\fg)$, where $\tau$ is the $q$-antilinear antiautomorphism defined by $$\tau(E_i)=q_i^{-1}\tilde{K}_{-i}F_i \qquad \tau(F_i)=q_i^{-1}\tilde{K}_{i}E_i \qquad \tau(K_\mu)=K_{-\mu}$$

\item $f\langle v,v'\rangle=\langle \bar f v,v'\rangle=\langle v,f v'\rangle$ for any  $v,v'\in V_\la^\Z$ and $f\in\Z[q,q^{-1}]$. 
\end{itemize}
\begin{cor}\label{Euler-form}
The isomorphism $K_0(R^\la)\cong V^\Z_\la$ intertwines the Euler form $$\langle[P_1],[P_2]\rangle=\dim_q\Hom(P_1,P_2)$$ 
with the $q$-Shapovalov form described above. 
In particular, \[\displaystyle\dim_q e(\Bi) R^\la e(\Bj)=\langle F_\Bi v_h,F_\Bj v_h\rangle\]
\end{cor}

We let $\langle-,-\rangle_1$ denote the specialization of this form at
$q=1$, which is thus the ungraded Euler form.

\subsection{Universal categorifications}
\label{sec:univ-quant}

In \cite[\S 5.1.2]{Rou2KM}, Rouquier discusses universal categorifications
of simple integrable modules.  Of course, to speak of universality, we
must have a notion of morphisms between categorical modules.  Let
$\aleph_1,\aleph_2\colon \tU\to \mathsf{Cat}$ be two strict
2-functors.
\begin{defn}
  A {\bf strongly equivariant} functor $\be$ is a collection of
  functors $\be(\la)\colon \aleph_1(\la) \to \aleph_2(\la)$ together
  with natural isomorphisms of functors $c_u\colon \be\circ\aleph_1(u)\cong
  \aleph_2(u)\circ \be$ for every 1-morphism $u\in \tU$ such that 
 \[c_v\circ (\id_{\be}\otimes\,  \aleph_1(\al))= (\aleph_2(\al) \otimes
 \id_{\be}) \circ c_u\] for every 2-morphism $\al\colon u\to v$ in
 $\tU$. (Here we use $\otimes$ for horizontal composition, and $\circ$
 for vertical composition of 2-morphisms). 
\end{defn}

In \cite[\S 5.1.2]{Rou2KM}, it is proven that there is a unique
$\tU$-module category $\check\cata^\la$ (he uses the notation
$\mathcal{L}(\la)$) with generating highest weight object $P$ with the
universal property that
\begin{itemize}
\item [($*$)] for any additive, idempotent-complete $\tU$-module category $\mathcal{C}$ and any
  object $C\in \operatorname{Ob}\mathcal{C}_\la$ with $\fE_i(C)=0$ for
  all $i$, there is a unique (up to unique isomorphism) strongly
  equivariant
  functor $\phi_C\colon \check\cata^\la \to \mathcal{C}$ sending $P_\emptyset$
  to $C$.
\end{itemize}
On purely formal grounds, such a category must exist for any version
of the 2-category categorifying $U_q(\fg)$; thus we will study the
corresponding module for the 2-category $\tU$ we have been using,
which is different from Rouquier's.  

In any case, this is a higher categorical
analogue of the universal property of a Verma module, but somewhat
surprisingly, $\check\cata^\la$ does {\it not} categorify a Verma
module, but rather an integrable module.  We recall that
$\End_{\tU}(\oplus_{\Bi}\eF_{\Bi}\la)\cong R\otimes\bLa$ where
$\bLa\cong \left(\otimes_{j\in\Gamma}\La(\bp_j)\right)$ and $\bp_j$ is an
infinite alphabet attached to each node, with the clockwise bubble of
degree $2n$ corresponding to $(-1)^ne_{n}(\bp_j)$, and the counterclockwise
one of degree $2n$ corresponding to $h_{n}(\bp_i)$.

\begin{defn}
  Let $\check{R}^\la$ be the quotient of $\End_{\tU}(\oplus_{\Bi}\eF_{\Bi}\la)$
  by the relations
$$\begin{tikzpicture}

\node at (12.1,0){$0$};

\node at (11.4,0){$=$};

\node at (-.9,0){$-$};
\node at (0,0){
\begin{tikzpicture}[baseline=-2.75pt, yscale=1.2, xscale=-1.2]
\node [fill=white,circle,inner sep=2pt,label={[scale=.7,white]left:{$-\la^j$}},above=4.5pt] at (1.3,.4) {};
\draw[very thick,postaction={decorate,decoration={markings,
    mark=at position .5 with {\arrow{<}}}}]    
(1,-.5) to[out=90,in=180] node[at start,below]{$j$} 
(1.5,.2) to[out=0,in=90] (1.75,0) to[out=-90,in=0] (1.5, -.2)
to[out=left,in=-90] node[at end,above]{$j$}  (1,.5);
\end{tikzpicture}};
\node at (1,0){$\dots$};
\node at (1.7,0){$=$};
\node at (2,0){
\begin{tikzpicture}[baseline=-2.75pt, yscale=1.2, xscale=-1.2]
\node [fill=white,circle,inner sep=2pt,label={[scale=.7,white]left:{$-\la^j$}},above=4.5pt] at (1.3,.4) {};

\draw[very thick,postaction={decorate,decoration={markings,
    mark=at position .65 with {\arrow{<}}}}]    (1,-.5) -- (1,.5) node[pos=.3,fill,circle,inner sep=2pt,label={[scale=.7]left:{$\la^j$}}]{} node[at end,above]{$j$} node[at start,below]{$j$} ;
\end{tikzpicture}};
\node at (3.1,0){$\dots$};
\node at (3.75,0){$+$};
\node at (5,0){
\begin{tikzpicture}[baseline=-2.75pt, yscale=1.2, xscale=-1.2]
\draw[very thick,postaction={decorate,decoration={markings,
    mark=at position .65 with {\arrow{<}}}}]    (1,-.5) -- (1,.5) node[pos=.3,fill,circle,inner sep=2pt,label={[scale=.7]left:{$\la^j-1$}}]{} node[at end,above]{$j$} node[at start,below]{$j$} ;
\draw[postaction={decorate,decoration={markings,
    mark=at position .5 with {\arrow[scale=1.3]{>}}}},very thick] (1.7,.4) circle (6pt);
\node [fill,circle,inner sep=2pt,label={[scale=.7]45:{$-\la^j$}},above=4.5pt] at (1.7,.4) {};
\end{tikzpicture}};
\node at (6,0){$\dots$};
\node at (6.6,0){$+$};
\node at (7.4,0){$\dots$};
\node at (8.2,0){$+$};
\excise{\node at (4,-2){
\begin{tikzpicture}[baseline=-2.75pt, yscale=1.2, xscale=-1.2]
\draw[very thick,postaction={decorate,decoration={markings,
    mark=at position .65 with {\arrow{<}}}}]    (1,-.5) -- (1,.5) node[pos=.3,fill,circle,inner sep=2pt]{} ;
\draw[postaction={decorate,decoration={markings,
    mark=at position .5 with {\arrow[scale=1.3]{>}}}},very thick] (1.8,.4) circle (6pt);
\node [fill,circle,inner sep=2pt,label={[scale=.7]right:{$-2$}},above=4.5pt] at (1.8,.4) {};
\end{tikzpicture}};

\node at (5.5,-2){$+$};}
\node at (9.4,0){
\begin{tikzpicture}[baseline=-2.75pt, yscale=1.2, xscale=-1.2]
\draw[very thick,postaction={decorate,decoration={markings,
    mark=at position .65 with {\arrow{<}}}}]    (1,-.5) -- (1,.5)  node[at end,above]{$j$} node[at start,below]{$j$} ;
\draw[postaction={decorate,decoration={markings,
    mark=at position .5 with {\arrow[scale=1.3]{>}}}},very thick] (1.8,.4) circle (6pt);
\node [fill,circle,inner sep=2pt,label={[scale=.7]45:{$-1$}},above=4.5pt] at (1.8,.4) {};
\end{tikzpicture}};
\node at (10.7,0){$\dots$};
\end{tikzpicture}$$
$$\tikz [baseline=-2pt, yscale=2, xscale=-2]{\draw[postaction={decorate,decoration={markings,
    mark=at position .5 with {\arrow[scale=1.3]{>}}}},very thick] (0,0) circle (6pt);
\node [fill,circle,inner
sep=2pt,label={[scale=.7]45:{$n$}},above=8.2pt] at (0,0)
{};\node at (-.6,0){$\dots$};}=0\qquad (n\geq 0)$$
where in both pictures, the ellipses indicate that the portion of the
diagram shown is at the far left.  More algebraically, these relations
can be written in the form
  \begin{align*}
    e(\Bi)(y^{\la^{i_i}}_1-e_1(\bp_{i_1})y^{\la_{i_1}-1}_1+\cdots + (-1)^{\la^{i_1}}e_{\la^{i_1}}(\bp_{i_1}))&=0\\
    e_{n}(\bp_j) &=0 \qquad (n> \la^i)
  \end{align*}
\end{defn}

Note that if we specialize $e_n(\bp_j)=0$ for every $n>0$, then we
recover the usual cyclotomic quotient $R^\la$.

If we extend scalars to polynomials in the $p_{*,*}$ and form the
algebra $\check{R}^\la\otimes_{\bLa}\K[p_{1,1},\dots,]$ then we can
rewrite these equations as
  \begin{align*}
    e(\Bi)(y_1-p_{i_1,1})(y_1-p_{i_1,2})\cdots (y_1-p_{i_1,\la^{i_1}})&=0\\
    p_{j,n} &=0 \qquad (n> \la^j)
  \end{align*}
In terms of the geometry of quiver varieties, $\check{R}^\la$ arises from considering equivariant sheaves for
the action of the group $\prod \operatorname{GL}(W_i)$, and its
extension to polynomials from equivariant sheaves for a maximal torus
of this group. 

The following is an analogue of Proposition \ref{quotient-is-cyc}; its
proof is so similar that we leave it as an exercise to the reader.
\begin{prop}\label{quotient-is-cyc-2}
  Let $J'$ be the set of all 2-morphisms in $\tU$ factoring through a
  1-morphism of the form $u\eE_i\colon\la\to \nu$ for all $u\colon\la+\al_i\to \nu$.  The idempotent
  completion of the quotient $\tU(\la)/J'$ is equivalent to the
  category $\check{R}^\la\mpmod$.
\end{prop}

\begin{cor}
   For any additive,  idempotent-complete $\tU$-module category $\mathcal{C}$ and any
  object $C\in \operatorname{Ob}\mathcal{C}_\la$ with $\fE_i(C)=0$ for
  all $i$, there is a unique strongly equivariant functor (up to unique isomorphism)
  $\phi_C\colon \check R^\la\mpmod \to \mathcal{C}$ sending $P_\emptyset$
  to $C$.
\end{cor}
\begin{proof}
For any object $C$, there is a unique strongly equivariant functor $\tU(\la)\to \mathcal{C}$ sending $\id_\la\mapsto C$.  We wish to
show that this factors through the functor from $\tU(\la)\to
\check{R}^\la\mpmod$. By Proposition \ref{quotient-is-cyc-2}, it
suffices to check that this map kills any 2-morphism factoring through
$u\eE_i \id_\la$.  Indeed, this is sent to $u\fE_i(C)=0$, so we kill
the required 2-morphisms.
\end{proof}

These algebras are quite interesting; though they are infinite
dimensional (unlike $R^\la$), they seem to have
finite global dimension (unlike $R^\la$).  We will explore these
algebras and their tensor product analogues in future work.

\section{The tensor product algebras}
\label{sec:KL}
\setcounter{equation}{0}

\subsection{Definition and basic properties}
We now proceed to the algebraic construction mentioned in the
introduction.  This is structured around certain algebras which are
pictorial in definition, and similar in flavor to the algebras $R^\la$ we have already defined.

The generators of our algebra are pictures in $\R^2$ consisting of red and black oriented embedded smooth curves decorated with a number (possibly 0) of dots such that:
\begin{itemize}
\item each curve begins on the line $y=0$ and ends on the line $y=1$
\item each curve is never tangent to a horizontal line
\item locally around each point, our diagram is either a single line or one of the pictures:
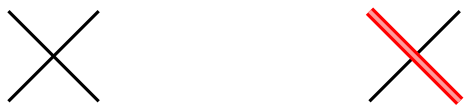
\begin{equation*}
\begin{tikzpicture}[scale=.6]
  \draw[very thick](-4,0) +(-1,-1) -- +(1,1);
  \draw[very thick](-4,0) +(1,-1) -- +(-1,1);


  \draw[very thick](4,0) +(-1,-1) -- +(1,1);
  \draw[wei, very thick](4,0) +(1,-1) -- +(-1,1);
\end{tikzpicture}
\end{equation*}
In particular, red lines are never allowed to cross, and no pair of lines are allowed to meet the lines $y=0$ or $y=1$ at the same point.
\end{itemize}
We will only ever be interested in these pictures up to isotopy preserving the conditions above.

Consider the algebra $\alg$ over $\K$ whose
generators are pictures as above, with each black line labeled by a
simple root of $\fg$, and each red line labeled with a dominant
weight.  Multiplication is given by the stacking of diagrams if the
pattern of red and black lines with their labels can be isotoped to
match up at $y=1$ in the first diagram and $y=0$ in the second and is
defined to be 0 otherwise.  Of course, this stacking must be followed
by smoothing any kinks at the joins of the lines (which is unique up
to isotopy) and vertical scaling to match the ends up with the correct
horizontal lines.
By convention the product $ab$ means stacking the diagram $b$ on top
of the diagram $a$.

The black strands satisfy the quiver Hecke relations from Figure \ref{quiver-hecke}, which again we apply as local relations (i.e.\ any time a small portion of a larger diagram matches one side of the relation, we equate it to the diagram with the small portion changed to match the other side of the relation).

We must also include new relations involving red lines which are:
\begin{itemize}
\item  All black crossings and dots can pass through red lines, with a correction term similar to Khovanov and Lauda's (for the latter two relations, we also include their mirror images):
  \begin{equation*}
    \begin{tikzpicture}[very thick]
      \draw (-3,0)  +(1,-1) -- +(-1,1) node[at start,below]{$i$};
      \draw (-3,0) +(-1,-1) -- +(1,1)node [at start,below]{$j$};
      \draw[wei] (-3,0)  +(0,-1) .. controls +(-1,0) ..  +(0,1);
      \node at (-1,0) {=};
      \draw (1,0)  +(1,-1) -- +(-1,1) node[at start,below]{$i$};
      \draw (1,0) +(-1,-1) -- +(1,1) node [at start,below]{$j$};
      \draw[wei] (1,0) +(0,-1) .. controls +(1,0) ..  +(0,1);   
\node at (2.8,0) {$+ $};
      \draw (6.5,0)  +(1,-1) -- +(1,1) node[midway,circle,fill,inner sep=2.5pt,label=right:{$a$}]{} node[at start,below]{$i$};
      \draw (6.5,0) +(-1,-1) -- +(-1,1) node[midway,circle,fill,inner sep=2.5pt,label=left:{$b$}]{} node [at start,below]{$j$};
      \draw[wei] (6.5,0) +(0,-1) -- +(0,1);
\node at (3.8,-.2){$\displaystyle \sum_{a+b+1=\la^i} \delta_{i,j} $}  ;
 \end{tikzpicture}
  \end{equation*}
\begin{equation}\label{dumb}
    \begin{tikzpicture}[very thick,baseline=2.85cm]
      \draw[wei] (-3,3)  +(1,-1) -- +(-1,1);
      \draw (-3,3)  +(0,-1) .. controls +(-1,0) ..  +(0,1);
      \draw (-3,3) +(-1,-1) -- +(1,1);
      \node at (-1,3) {=};
      \draw[wei] (1,3)  +(1,-1) -- +(-1,1);
  \draw (1,3)  +(0,-1) .. controls +(1,0) ..  +(0,1);
      \draw (1,3) +(-1,-1) -- +(1,1);    \end{tikzpicture}
  \end{equation}
\begin{equation*}
    \begin{tikzpicture}[very thick]
  \draw(-3,6) +(-1,-1) -- +(1,1);
  \draw[wei](-3,6) +(1,-1) -- +(-1,1);
\fill (-3.5,5.5) circle (3pt);
\node at (-1,6) {=};
 \draw(1,6) +(-1,-1) -- +(1,1);
  \draw[wei](1,6) +(1,-1) -- +(-1,1);
\fill (1.5,6.5) circle (3pt);
    \end{tikzpicture}
  \end{equation*}
\item The ``cost'' of a separating a red and a black line is adding $\la^i=\al_i^\vee(\la)$ dots to the black strand.
  \begin{equation}\label{cost}
  \begin{tikzpicture}[very thick,baseline=1.6cm]
    \draw (-2.8,0)  +(0,-1) .. controls +(1.6,0) ..  +(0,1) node[below,at start]{$i$};
       \draw[wei] (-1.2,0)  +(0,-1) .. controls +(-1.6,0) ..  +(0,1) node[below,at start]{$\la$};
           \node at (-.3,0) {=};
    \draw[wei] (2.8,0)  +(0,-1) -- +(0,1) node[below,at start]{$\la$};
       \draw (1.2,0)  +(0,-1) -- +(0,1) node[below,at start]{$i$};
       \fill (1.2,0) circle (3pt) node[left=3pt]{$\la^i$};
          \draw[wei] (-2.8,3)  +(0,-1) .. controls +(1.6,0) ..  +(0,1) node[below,at start]{$\la$};
  \draw (-1.2,3)  +(0,-1) .. controls +(-1.6,0) ..  +(0,1) node[below,at start]{$i$};
           \node at (-.3,3) {=};
    \draw (2.8,3)  +(0,-1) -- +(0,1) node[below,at start]{$i$};
       \draw[wei] (1.2,3)  +(0,-1) -- +(0,1) node[below,at start]{$\la$};
       \fill (2.8,3) circle (3pt) node[right=3pt]{$\la^i$};
  \end{tikzpicture}
\end{equation}
\item If at any point in the diagram any black line is to the left of all reds (i.e., there is a value $a$ such that the left-most intersection of $y=a$ with a strand is with a black strand), then the diagram is 0.  We will refer to such a strand as {\bf violating}.
\end{itemize}

We also let $\tilde \alg$ denote the algebra without the last relation
above.  While ${\alg}$ is the algebra of primary importance for us,
$\tilde \alg$ will be of great technical utility, since we can
construct a basis for it, whereas for $\alg$, this seems to be quite
out of reach.  Furthermore, the algebra $\tilde{T}$ has a more simple
geometric description, as we discuss in \cite[\S
4]{WebwKLR}.

Following Brundan and Kleshchev, we will sometimes use $y_i$ to
represent multiplication by a dot on the $i$th black strand, and $\psi_i$ to denote the crossing of the $i$th and $i+1$st black strands and $e(\Bi)$ to denote the sum of
all pictures where there are no crossings or dots, and the black
strands are labeled with $\Bi=(i_1,\dots,i_n)$ in that order.
\begin{figure}[tb]
\begin{tikzpicture}
\draw[wei] (0,0) to [in=-90,out=90] (1,2);
\draw[very thick] (1,0) to [in=-90,out=90] node (a) [pos=.9]{} (2,2) ;
\draw[very thick] (2,0) to [in=-90,out=90] node (b) [pos=.1]{} (0,2) ;
\node (c) at  (5,1.8){ a non-violating strand};
\node (d) at  (5,0.2){ a violating strand};
\draw[->] (c) -- (a);
\draw[->] (d) -- (b);
\end{tikzpicture}
\caption{An example of a violating and non-violating strand}
\label{fig:violating}
\end{figure}
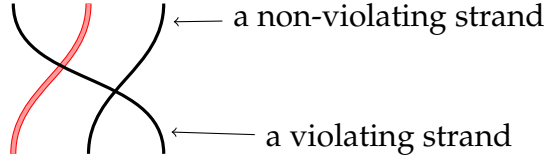

\subsection*{Grading} This algebra is graded with degrees given by 
\begin{itemize}
\item a black/black crossing: $-\langle\al_i,\al_j\rangle$,
\item a black dot:
$\langle\al_i,\al_i\rangle=2d_i$
\item a red/black crossing:
$\langle\al_i,\la\rangle=d_i\la^i$.
\end{itemize}

This algebra is endowed with a natural anti-automorphism $a\mapsto
\dot a$ given by reflecting diagrams in the horizontal axis.  If $M$ is
a right module over this algebra, we let $\dot M$ be the left module
given by twisting the action by this anti-automorphism.
\begin{defn}
For a  finite-dimensional right module $M$, we define the {\bf dual module} by $M^\star=\dot M^*$, where $(\cdot)^*$ denotes usual vector space duality interchanging left and right modules.  
\end{defn} 
This is a right module since both vector space dual and the anti-automorphism interchange left and right modules.

\begin{defn}
  For a sequence of weights $\bla=(\la_1,\dots,\la_\ell)$, we let
  $\alg^\bla$ be the subalgebra of $\alg$ where
  the red lines are labeled, in order, with the elements of $\bla$.
  We let $\cata^\bla=\alg^\bla\modu$ be the category of finite
  dimensional representations of $\alg^\bla$ graded by $\Z$.

  We let $\alg^\bla_\al$ for $\al\in \wela(\fg)$ be the subalgebra of
  $\alg^\bla$ where the sum of the roots associated to the black strands
  is $\sum_i\la_i-\al$.
\end{defn}
We also let $\tilde{\alg}^\bla$ denote the corresponding subalgebra of
$\tilde{\alg}$, and $K^\bla$ denote the kernel of the natural map
$\tilde{\alg}^\bla\to \alg^\bla$.  By definition, $K^\bla$ is the span of
the diagrams in $\tilde{\alg}^\bla$ with a violating strand, since these
elements are generators of the kernel and their span is closed under
left and right multiplication.

Consider a sequence of simple roots $\Bi=(i_1,\dots, i_n)$, and a weakly
increasing map \linebreak $\kappa\colon [1,\ell]\to [0,n]$. 

We can define an idempotent  $e(\Bi,\kappa)$ as the crossingless diagram where the strands are labeled by the roots in the order given by
$\Bi$, with the $j$th red line immediately right of the $\kappa(j)$th
black line, except that if $\kappa(j)$'s agree, the original order of
red lines is preserved. By
convention, if $\kappa(i)=0$, then the $i$th red strand is left of all
black strands.  Note that if $e({\Bi},\kappa)$ is not trivial, we must
have $\kappa(1)=0$.
\begin{defn}
We consider the projective
modules $P_{\Bi}^\kappa=e(\Bi,\kappa)\alg^\bla$
and  $\tilde P^\kappa_\Bi=e(\Bi,\kappa)\tilde \alg^\bla$ and let $K^\kappa_\Bi$ be the kernel of the natural map $\tilde P^\kappa_\Bi\to P^\kappa_\Bi$.
\end{defn} 
Note that the kernel ${K}^\bla$ can also be described as the span of the elements that factor through $\tilde{P}_{\Bi}^{\kappa}$ where $\kappa(1)\neq 0$, that is, the trace of these projectives.  In categorical terms, the projective modules over $\alg^\bla$ are the quotient of the category of projective modules over $\tilde{\alg}^\bla$ by this collection of projectives.

We can
generalize this notion a bit by allowing multiplicities $\vartheta_j$;
we associate a projective to the sequence
$(i_1^{(\vartheta_1)},\dots,i_n^{(\vartheta_n)})$ which is a submodule
of the projective for the sequence where $i_j^{(\vartheta_j)}$ has
been expanded to $\vartheta_j$ instances of $i_j$.  This is the
projective given by multiplying each block of strands in the expanded
projective on the bottom by the idempotent denoted $e_{\vartheta_j}$
in \cite[\S 2]{KLII}, which we illustrate in Figure \ref{fig:e4}.

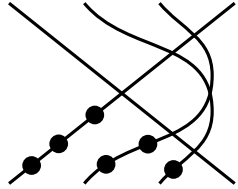
\begin{figure}[ht]
  \centering
  \begin{tikzpicture}[very thick,yscale=1.2,xscale=-1]
    \draw (0,0) -- (3,2);
\draw (1,0) to[out=135, in=-90] node[pos=.2,fill,circle,inner sep=2.5pt]{}  (.3,.8) to[out=90, in=-135](2,2);
\draw (2,0) to[out=135,  in=-90]  node[pos=.15,fill,circle,inner sep=2.5pt]{} node[pos=.4,fill,circle,inner sep=2.5pt]{}(.3,1.2) to[out=90,  in=-135] (1,2);
\draw (3,0) -- node[pos=.1,fill,circle,inner sep=2.5pt]{} node[pos=.22,fill,circle,inner sep=2.5pt]{} node[pos=.38,fill,circle,inner sep=2.5pt]{} (0,2);
  \end{tikzpicture}
  \caption{The idempotent $e_4$.}
  \label{fig:e4}
\end{figure}
Usually, we will not require these multiplicities, and will thus exclude them from the notation.  Unless they are indicated explicitly, the reader should assume that they are 1.

Under decategorification, the projective $P^\kappa_\Bi$ is sent to the vector $$F_{i_n}^{(\theta_{i_n})}\cdots F_{i_{\kappa(\ell)}}^{(\theta_{i_{\kappa(\ell)}})}(\cdots (F_{i_{\kappa(3)}}^{(\theta_{i_{\kappa(3)}})}\cdots F_{i_{\kappa(2)+1}}^{(\theta_{i_{\kappa(2)+1}})}(F_{i_{\kappa(2)}}^{(\theta_{i_{\kappa(2)}})}\cdots F_{i_{1}}^{(\theta_{i_{1}})}v_1)\otimes v_2)\otimes \cdots \otimes v_\ell),$$
where $v_i\in V_{\la_i}$ is a fixed highest weight vector, as we prove in Section \ref{sec:decat}.

\subsection{Examples}

To give a simple illustration of the behavior of our algebra, let us consider $\fg=\mathfrak{sl}_2$, and $\bla=(1,1)$.  Thus, our diagrams have 2 red lines, both labeled with 1's.

In this case, the algebras $\alg^\bla_\al$ are  easily described as follows:
\begin{itemize}
\item $\alg^{(1,1)}_2\cong \K$: it is just multiples of the diagram which is just a pair of red lines.
\item $\alg^{(1,1)}_0$ is spanned by 
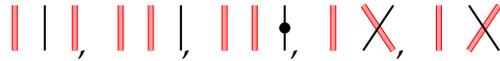
\begin{center}
\tikz[xscale=.8, yscale=.6]{\draw[wei] (0,0)--(0,1); \draw[thick] (.5,0) --(.5,1);\draw[wei] (1,0)--(1,1); },\quad \tikz[xscale=.8,yscale=.6]{\draw[wei] (0,0)--(0,1); \draw[wei] (.5,0) --(.5,1);\draw[thick] (1,0)--(1,1);},\quad \tikz[xscale=.8,yscale=.6]{\draw[wei] (0,0)--(0,1); \draw[wei] (.5,0) --(.5,1);\draw[thick] (1,0)--(1,1) node[midway,fill, circle,inner sep=1.5pt]{};},\quad \tikz[xscale=.8,yscale=.6]{\draw[wei] (0,0)--(0,1); \draw[wei] (1,0) --(.5,1);\draw[thick] (.5,0)--(1,1);},\quad \tikz[xscale=.8,yscale=.6]{\draw[wei] (0,0)--(0,1); \draw[wei] (.5,0) --(1,1);\draw[thick] (1,0)--(.5,1);}
\end{center}
One can easily check that this is the standard presentation of a regular block of category $\cO$ for $\mathfrak{sl}_2$ as a quotient of the path algebra of a quiver (see, for example, \cite{Str03}).
\item $\alg^{(1,1)}_{-2}\cong \operatorname{End}(\K^3)$: quotienting out by the left ideal generated by all diagrams with crossings gives the unique irreducible representation.  The algebra is spanned by the diagrams, which one can easily check multiply (up to sign) as the elementary generators of $\operatorname{End}(\K^3)$.
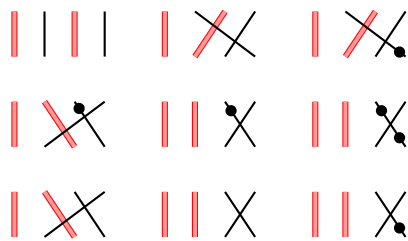
\begin{center}
\begin{tikzpicture}[yscale=1.2,xscale=2]
\node at (0,0){ \tikz[xscale=.8, yscale=.6]{\draw[wei] (0,0)--(0,1); \draw[thick] (.5,0) --(.5,1);\draw[wei] (1,0)--(1,1);\draw[thick] (1.5,0) --(1.5,1); }};

\node at (0,-1) {\tikz[xscale=.8,yscale=.6]{\draw[wei] (0,0)--(0,1); \draw[wei] (1,0) --(.5,1);\draw[thick] (1.5,0)--(1,1) node[pos=.85,fill, circle,inner sep=1.5pt]{}; \draw[thick] (.5,0) --(1.5,1) ;}};

\node at (1,-1) {\tikz[xscale=.8,yscale=.6]{\draw[wei] (0,0)--(0,1); \draw[wei] (.5,0) --(.5,1);\draw[thick] (1.5,0)--(1,1) node[pos=.8,fill, circle,inner sep=1.5pt]{}; \draw[thick] (1,0) --(1.5,1) ;}};
\node at (1,0) {\tikz[xscale=.8,yscale=.6]{\draw[wei] (0,0)--(0,1); \draw[wei] (.5,0) --(1,1);\draw[thick] (1.5,0)--(.5,1); \draw[thick] (1,0) --(1.5,1) ;}};

\node at (0,-2) {\tikz[xscale=.8,yscale=.6]{\draw[wei] (0,0)--(0,1); \draw[wei] (1,0) --(.5,1);\draw[thick] (1.5,0)--(1,1); \draw[thick] (.5,0) --(1.5,1) ;}};

\node at (2,-2) {\tikz[xscale=.8,yscale=.6]{\draw[wei] (0,0)--(0,1); \draw[wei] (.5,0) --(.5,1);\draw[thick] (1.5,0)--(1,1) node[pos=.2,fill, circle,inner sep=1.5pt]{}; \draw[thick] (1,0) --(1.5,1);}};

\node at (2,0) {\tikz[xscale=.8,yscale=.6]{\draw[wei] (0,0)--(0,1); \draw[wei] (.5,0) --(1,1);\draw[thick] (1.5,0)--(.5,1) node[pos=.1,fill, circle,inner sep=1.5pt]{}; \draw[thick] (1,0) --(1.5,1);}};

\node at (2,-1) {\tikz[xscale=.8,yscale=.6]{\draw[wei] (0,0)--(0,1); \draw[wei] (.5,0) --(.5,1);\draw[thick] (1.5,0)--(1,1) node[pos=.8,fill, circle,inner sep=1.5pt]{} node[pos=.2,fill, circle,inner sep=1.5pt]{}; \draw[thick] (1,0) --(1.5,1) ;}};

\node at (1,-2) {\tikz[xscale=.8,yscale=.6]{\draw[wei] (0,0)--(0,1); \draw[wei] (.5,0) --(.5,1); \draw[thick] (1.5,0)--(1,1) ; \draw[thick] (1,0) --(1.5,1);}};

\excise{
\node at (1,2){ \tikz[xscale=.8,yscale=.6]{\draw[wei] (0,0)--(0,1); \draw[wei] (1,0) --(.5,1);\draw[thick] (.5,0)--(1,1);}};

 \tikz[xscale=.8,yscale=.6]{\draw[wei] (0,0)--(0,1); \draw[wei] (.5,0) --(1,1);\draw[thick] (1,0)--(.5,1);}
}
\end{tikzpicture}

\end{center}
\end{itemize}
\excise{

Perhaps a more interesting example is the case of $\fg=\mathfrak{sp}_4\cong \mathfrak{so}_5$.  We let $\al$ be the long root, and $\be$ be the short root.  We let $\bla=(\nicefrac{1}{2}\al+\be,\nicefrac{1}{2}\al+\be)$.  In this case:
\begin{itemize}
\item $\alg^\bla_{\al+2\be}\cong \K$.
\item $\alg^\bla_{\al+\be}\cong \K$.
\item 

\end{itemize}
}

\subsection{A basis and spanning set}
Recall that a {\bf reduced word} in the symmetric group is a list of $k$ adjacent transpositions $(i,i+1)$ whose product cannot be written as a shorter product of adjacent transpositions.  For each choice of a reduced word $\Bw$ for a permutation of $n+\ell$
letters which doesn't invert any pair of red strands, we have an element $\psi_{\Bw}$ of $P_{\Bi}^\kappa$ given by
replacing the simple reflection $(i,i+1)$ with the crossing of the $i$
and $i+1$st strands (red or black) and multiplying out the result.

\begin{prop}\label{basis}
 For any fixed choice of reduced word for each permutation, the algebra $\tilde{\alg}^\bla$ has a basis given $e(\Bi,\kappa)\psi_{\Bw}y_1^{a_1}\cdots y_n^{a_n}$ for all permutations which preserve the relative order of the red strands and any $n$-tuple $\{a_i\in \Z_{\geq 0}\}.$
\end{prop}

This proposition is crucial in that it not only gives us a basis, but an ordered basis; permutations have a natural partial order, the strong Bruhat order.

We will always refer to the process of rewriting an element in terms of this basis as ``straightening'' since visually, it is akin to pulling all the strands taut until they are straight, though this image is slightly misleading, as we will explain momentarily.

\begin{proof}
The proof is directly analogous to that of \cite[Theorem 2.5]{KLI}. 

  First we show is that this set spans, for which is suffices to show that $\psi_\Bw$ for any word can be rewritten in terms of
  $y_i$'s times $\psi_{\Bw'}$ for our fixed choice of reduced words
  and shorter diagrams.

  If $\Bw$ is not a reduced word in the symmetric group, then by
  applying braid relations (which hold modulo shorter words), we can
  assume that there are two consecutive crossings of the same strands,
  which can be simplified using the relations and written in terms of $\psi_{\Bw'}$ for shorter words $\Bw'$.

  If $\Bw$ is a reduced word, then the fixed reduced word
  corresponding to the same permutation $\Bw'$ differs from $\Bw$ by
  Tits moves, so the difference between $\psi_\Bw-\psi_{\Bw'}$ can
  thus be written in terms of shorter diagrams.

The difficult part is to show that the elements are linearly independent.  First, we note that $\tilde{\alg}^\bla$ has a version of Khovanov and Lauda's polynomial representation, where $\tilde{\alg}^\bla$ acts on a direct sum of polynomial rings $\K[y_1,\dots,y_n]$ over all choices of $\Bi$ and $\kappa$ by the rule (where in each case, there are $k-1$ black strands to the left of the portion of the diagram shown) shown in Figure \ref{poly-rep}.
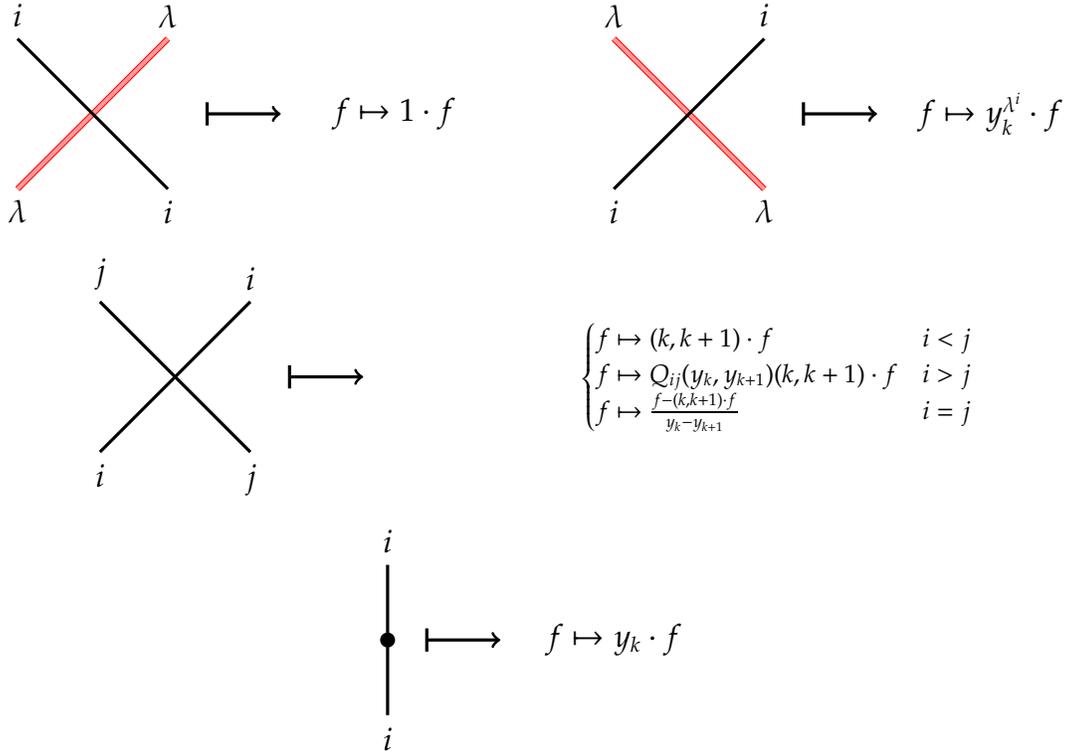
\begin{figure}
\begin{center}
\begin{tikzpicture}
\node at (-4,3.5){
\begin{tikzpicture}
\draw[wei] (-1,-1) -- (1,1) node[at start,below]{
$\la$} node[at end,above]{
$\la$};
\draw[very thick] (1,-1) -- (-1,1) node[at start,below]{
$i$} node[at end,above]{
$i$};
\draw[very thick,|->] (1.5,0) -- (2.5,0);
\node at (4,0) {$f\mapsto 1\cdot f$ } ;
\end{tikzpicture}};

\node at (4,3.5){
\begin{tikzpicture}
\draw[wei] (1,-1) -- (-1,1) node[at start,below]{
$\la$} node[at end,above]{
$\la$};
\draw[very thick] (-1,-1) -- (1,1) node[at start,below]{
$i$} node[at end,above]{
$i$};
\draw[very thick,|->] (1.5,0) -- (2.5,0);
\node at (4,0) {$f\mapsto y_k^{\la^i}\cdot f$ } ;
\end{tikzpicture}
};

\node at (0,0){
\begin{tikzpicture}
\draw[very thick] (1,-1) -- (-1,1) node[at start,below]{
$j$} node[at end,above]{
$j$};
\draw[very thick] (-1,-1) -- (1,1) node[at start,below]{
$i$} node[at end,above]{
$i$};
\draw[very thick,|->] (1.5,0) -- (2.5,0);
\node[scale=.8] at (8,0) {$\begin{cases}
f\mapsto (k,k+1)\cdot f & i<j\\   
f\mapsto Q_{ij}(y_k,y_{k+1})(k,k+1)\cdot f & i>j\\
f\mapsto \frac{f-(k,k+1)\cdot f}{y_k-y_{k+1}} & i=j \\  
\end{cases}$ } ;
\end{tikzpicture}
};

\node at (0,-3.5){
\begin{tikzpicture}
\draw[very thick] (1,-1) -- (1,1) node[at start,below]{
$i$} node[at end,above]{
$i$} node[circle,midway,fill,inner sep=2pt]{};
\draw[very thick,|->] (1.5,0) -- (2.5,0);
\node at (4,0) {$f\mapsto y_k\cdot f$ } ;
\end{tikzpicture}
};
\end{tikzpicture}

\end{center}
\caption{The polynomial representation of $\tilde{\alg}^\bla$}
\label{poly-rep}
\end{figure}

The action of black diagrams is that of Khovanov-Lauda (in original signs, this is \cite[Theorem 2.3]{KLI}, and is discussed with sign modifications in the final section of \cite{KLII}; the most general version for arbitrary $Q_{*,*}$ is covered in \cite[Proposition 3.12]{Rou2KM}), so the only relations we need check are our additional relations (\ref{dumb}) and (\ref{cost}).  The only one of these which is interesting is the first line of (\ref{dumb}).  The LHS is \[f\mapsto \frac{y_k^{\la^i} f-y_{k+1}^{\la^i}(k,k+1)\cdot f}{y_k-y_{k+1}}\] and the RHS is $$f\mapsto y_{k+1}^{\la^i} \frac{f-(k,k+1)\cdot f}{y_k-y_{k+1}}+\frac{y_k^{\la^i}-y_{k+1}^{\la^i}}{y_k-y_{k+1}}f$$ and the relation is verified.

The most important consequence of this is that Khovanov and Lauda's algebra $R$ injects into $\tilde{\alg}^\la$, since any element of the kernel acts trivially on the polynomial representation, and thus is trivial.

Now, we show that we have a basis in general by reducing to this case.  Assume that there is a non-trivial linear relation between vectors of the form in the statement.  Then we can compose on the bottom with the element $\theta_\kappa$, which pulls all black strands to the right and red to the left, and on the top with $\dot\theta_\kappa$.  Pulling all black strands to the right (as described above when showing our desired elements span), we obtain a relation in $R$.  On the other hand, there must be a $\psi_{\Bw}\By^{\mathbf{a}}$ with nontrivial coefficient maximal in Bruhat order compared to all other diagrams with non-trivial coefficients.  Since pulling right only adds correction terms strictly smaller in Bruhat order, we have a relation in $R$ where the corresponding diagram to $\psi_{\Bw}\By^{\mathbf{a}}$  has non-trivial coefficient.  Since these elements are a basis, this coefficient must be trivial, giving a contradiction.  Thus, this relation is trivial and we have a basis of $\tilde{\alg}^\la$.
\end{proof}

\begin{prop}\label{straighten}
  For any fixed choice of reduced word for each permutation, the
  elements $\psi_{\Bw}$ generate $P_{\Bi}^\kappa$ as a module over the
  subalgebra generated by the $y_i$'s.
\end{prop}
\begin{proof}
Clear from the fact that $\tilde{\alg}^\bla$ surjects onto $\alg^\bla$.
\end{proof}

In order to organize our computations, we must keep track of leading terms in this basis under multiplication; the term ``straightening'' suggests that these will roughly correspond to the multiplication of permutations.  The reality is a bit more subtle.  In order to do this, we consider the category $\mho_n$ whose objects are ordered $n$ elements sets labeled with simple roots of our algebra, and whose morphisms are label preserving maps.  Obviously, every diagram in $\tilde \alg^\bla$ gives such a map by simply tracing out the black strands (we ignore red strands for the time being).  We now wish to put a slightly strange composition on these maps which will give us a different category from the naive one with these morphisms.  

In order to compose morphisms $a$ and $b$, we factor each in a minimal length way into the naive product of a number of simple involutions, i.e.\ those that switch adjacent elements in the order.  Now, we consider the concatenation of these words, which we endeavor to simplify.  We impose the usual braid relations on involutions, but we change how they square.  If $s_i=(i,i+1)$ in cycle notation, we impose that $s_i^2=1$ if the $i$th and $i+1$th have different labels and $s_i^2=s_i$ is the labels are the same. 

Note that if the concatenation is not a reduced word, we can apply braid relations until there are two adjacent $s_i$'s in the word, which we can simplify to obtain a shorter word.  This process terminates at a reduced word for a unique permutation.  We note that morphisms in this category can be given the usual Bruhat order.

\begin{prop}
Given any diagram  $x\in \tilde{\alg}^\bla$ with associated morphism $\omega_x$ in $\mho_n$,  when $x$ is written in terms of basis elements, all diagrams which appear have associated morphisms shorter than or equal to $\omega_x$ in $\mho_n$.
\end{prop}
\begin{proof}
This is clear from the quiver Hecke relations of Figure \ref{quiver-hecke} and the algorithm for writing a morphism in terms of the basis, since all relations for reducing the ``length'' of a diagram, or to adjust it to fit a particular reduced word of a permutation only introduce extra terms shorter in Bruhat order.  We must use $\mho_n$ because these relations will sometimes remove a $s_i$ which permutes two like colored strands from a word where $s_i^2$ appears.  This could increase the length in the usual multiplication of the symmetric group, but will not in $\mho_n$.  
\end{proof}

This proposition has another important consequence.  Let $\kappa_1,\kappa_2$ be two weakly increasing functions $[1,\ell]\to [0,n]$ and assume that for some $j$ we have $\kappa_i(j)=\kappa_i(j+1)$ for $i=1,2$.  Then, we let $\bla'$ denote $\bla$ with the block $\la_k,\la_{k+1}$ replaced by $\la_k+\la_{k+1}$ and let $$\kappa_i'(k)=\begin{cases} \kappa_i(k) &k\leq j\\
\kappa_i(k+1) & k>j.
\end{cases}$$

There is an obvious map $$\tilde{c}:e({\Bi,\kappa_1'})\tilde \alg^{\bla'}e({\Bi,\kappa_2'})\to e({\Bi,\kappa_1})\tilde \alg^\bla e({\Bi,\kappa_2})$$ given by separating the $k$th red strand into 2 strands, labeled with $\la_k$ and $\la_{k+1}$, and also an induced map on quotients $${c}:e({\Bi,\kappa_1'}) \alg^{\bla'}e({\Bi,\kappa_2'})\to e({\Bi,\kappa_1}) \alg^\bla e({\Bi,\kappa_2}).$$ 
\begin{cor}\label{split-strands}
The maps $\tilde{c}$ and $c$ are isomorphisms. 
\end{cor}
\begin{proof}
The fact for $\tilde{c}$ simply follows from the fact that the bases of Proposition \ref{basis} correspond under this map.  

Note further that under $\tilde{c}$ that any element of $e({\Bi,\kappa_1})\tilde \alg^\bla e({\Bi,\kappa_2})$ which has a violating strand can be rewritten by sliding all crossings and dots out of the space between the $k$ and $k+1$st strands to be the image of an element with a violating strand under $\tilde{c}$.  Since the kernels to the projections to the domain and target of $c$ correspond under $\tilde{c}$, we must have that $c$ is an isomorphism.
\end{proof}

\subsection{Relationship to quiver Hecke algebras}
\label{sec:QHA}

 If $\bla=(\la)$, then we will simplify notation by writing $\alg^\la$ for $\alg^\bla$, and $P_\Bi$ for $P_\Bi^0$.
\begin{thm}
\label{cyclotomic}
$R^\la\cong \alg^\la$.
\end{thm}
\begin{proof}
  By composing the inclusion $R\hookrightarrow \tilde{\alg}^\la$ given by adding a
  red line at the left and the projection
  $\tilde{\alg}^\la\to \alg^\la$, we obtain a map.  This map is a surjection
  since any  element of the basis of Proposition \ref{basis} not in the image
  contains a strand to the left of the single red strand and thus is
  sent to 0.

  The image of $R$ in $\tilde{\alg}^\la$ is readily identifiable: it is
  the span of all diagrams where both at $y=0$ and $y=1$, the single
  red strand is left of all blacks.  The image is clearly contained in
  this space, since the image of a diagram in $R$ satisfies this
  condition for all values of $y$, and any diagram with this condition
  can be rewritten using the Theorem \ref{basis} as a sum of elements
  where no two strands cross twice.  Since the red strand is at the
  far left both at $y=0$ and $y=1$, it cannot cross a black strand
  exactly once, and thus must not cross with any of them; that is, we
  have written our element in terms of basis vectors in the image of
  $R$.  Let $e_0$ be the idempotent given by the image of the identity
  in $R$.  We note that left multiplication by $e_0$ kills exactly the
  diagrams which do not have the red strand at the far left at the
  bottom and similarly for right multiplication and the top, so
  $R=e_0\tilde{\alg}^\la e_0$.

  The kernel of the map $R\to \alg^\la$ is thus the intersection $K^\la\cap R$; we
  must show that this coincides with the cyclotomic ideal.  First note
  that $K^\la\cap R=e_0K^\la e_0$.  By definition, $K^\la$ is spanned
  by elements with a violating strand, so $K^\la\cap R$ is spanned by
  all elements with a violating strand where the red strand is at the
  left at the top and bottom.

  In such a diagram, we can slide all violating black strands back
  over the red. We thus obtain $\la^i$ dots on all $\al_i$-colored
  strands that were violating in the earlier diagram.  In particular,
  any one of these strands which has no other strand to its left at
  the point where it was violating carries $\la^i$ dots, and thus lies
  in the cyclotomic ideal.  On the other hand, for any element in the
  cyclotomic ideal, when can simply slide the leftmost strand left at
  the point where it carries $\la^i$ dots to obtain a violating
  strand. This gives the equality of ideals and thus the desired
  isomorphism.
\end{proof}

This cyclotomic quotient plays several important roles in
``controlling'' the representation theory of $\alg^\bla$.
Consider the projectives where $\kappa(i)=0$ for all $i$, in which case
we will simply denote the projective for $\kappa$ by $P^0_\Bi$.  We
note that $P^0_\Bi$ carries an obvious action of $R$ by composition on
the bottom.  We let $P^0=\oplus_{\Bi}P^0_\Bi$ be the sum of all such
projectives with $\kappa(i)=0$, and $P=\oplus_{\Bi}P_\Bi$ be the
corresponding module over $\alg^\la$.

\begin{prop}\label{self-dual-end}
$\displaystyle \End_{\alg^\bla}(P^0)\cong \alg^\la\cong R^\la .$
\end{prop}
\begin{proof}
  The first isomorphism follows from repeated application of Corollary
  \ref{split-strands}.  The second is just a restatement of
  Proposition \ref{cyclotomic}
\end{proof}

\subsection{The module category structure}\label{sec:module-full}

Based on the graphical calculus developed in the Section \ref{sec:categorification}, we can define an action of $\tU$ on the categories $\cata^\bla$.  First, we define a candidate functors by a simple extension of our graphical calculus.  Each of these is defined sending a module $M$ to a module spanned by diagrams containing a coupon that carries elements of $M$.

The {\bf induction} $\tilde \fF_i M$ of an $\tilde{\alg}^\bla$-module $M$ is the vector space generated by diagrams as in Figure \ref{funcs} for $m\in M$, modulo the relation that the sum of diagrams which are identical outside the coupon is given by adding the labels on the coupon.  

The algebra $\tilde{\alg}^\bla$ acts by multiplication on the top, simplifying using Proposition \ref{straighten} so that all crossings of strands connecting the coupon occur below the new strand, and absorbing these into the coupon.

More algebraically, this is an extension of scalars; We have a map $\nu_i\colon \tilde{\alg}^\bla\to \tilde{\alg}^\bla$ given by adding a $i$-colored strand at the far right, and $\tilde{\fF}_iM\cong \tilde \alg^\bla\otimes_{\tilde \alg^\bla}M$ where the tensor product is taken over the ring map $\nu_i$.
\begin{defn}
{\bf Induction} for $\alg^\bla$-modules is defined by $\fF_iM=\tilde{\fF}_iM\otimes_{\tilde{\alg}^\bla}\alg^\bla$ for $M\in \cata^\bla_\mu$.

Analogous {\bf restriction} functors $\tilde{\fE}_i,\fE_i$ right adjoint to these are defined by the second set of pictures in Figure \ref{funcs}.  
\end{defn}

These functors give an action of $\tU$, as we will show momentarily; we should note that in order for this action to make sense, we must assign a category to each weight, refining the category that corresponds to the entire representation.  To calculate the weight in which $P^{\kappa}_{\Bi}$ belongs, one should add the weights on the red lines minus the roots on the black strands.  

More generally, we can imagine labeling the regions of the diagram starting with 0 at the left, and using the rule given in \cite[\S 3.1.1]{KLIII}, which the additional rule that the label on the region right of a red strand minus that to its left is the label of the strand itself. The weight we identify above would be the label at the far right of the diagram.

\begin{figure}
\centering
  \begin{tikzpicture}[very thick,label distance=7pt]
 \node at (4,0){ \begin{tikzpicture}[very thick]
\draw[dir] (.25,-2.5) to[out=90,in=180]  (1.2, -1.9) to[out=0, in=90] (2.5,-2.7) node[below,at end]{$i$};
\draw[wei] (-1.5,-2.5)  to[out=90, in=-90] (-1.5,-1.25);
\draw[dir] (-.25,-2.5) to[out=90, in=-90] (-.25,-1.25);
\draw[dir] (.75,-2.5) to[out=90, in=-90] (.25,-1.25);
\draw[dir] (1.5,-2.5) to[out=90, in=-90] (1.3,-1.25);
\node at (-.85,-1.6) {$\cdots$};
\node at (.85,-1.6) {$\cdots$};
\node (a) at (0,-2.8)[draw,fill=white,inner xsep=42pt,inner ysep=10pt,label=below:{$\fE_i$}]{m};

  \end{tikzpicture}};
\node at (-4,0){  \begin{tikzpicture}[very thick]
\draw[postaction={decorate,decoration={markings,
    mark=at position .9 with {\arrow[scale=1.3]{<}}}}] (2.5,-2.7) to[out=90, in=-90] (.25,-1.25)  node[below,at start]{$i$};
\draw[wei] (-1.5,-2.5)  to[out=90, in=-90] (-1.5,-1.25);
\draw[dir] (-.25,-2.5)  to[out=90, in=-90] (-.25,-1.25);
\draw[dir] (.25,-2.5) to[out=90, in=-90] (.75,-1.25);
\draw[dir] (1.5,-2.5) to[out=90, in=-90] (2,-1.25);
\node at (-.85,-1.6) {$\cdots$};
\node at (1.325,-1.6) {$\cdots$};
\node (a) at (0,-2.8)[draw,inner xsep=42pt,inner ysep=10pt,fill=white, label=below:{$\fF_i$}]{m};
  \end{tikzpicture}};
  \end{tikzpicture}
\caption{The functors $\fE_i$ and $\fF_i$}
\label{funcs}
\end{figure}

\begin{prop}\label{full-action}
There is a representation of $\tU$ which sends \[\mu\mapsto
\cata^\bla_\mu \quad \eE_i\mapsto \fE_i\quad  \eF_i\mapsto \fF_i.\] The action of 2-morphisms is simply by composition on the bottom of the diagram, perhaps followed by simplification.

In particular, the functors $\fF_i$ and $\fE_i$ are exact.
\end{prop} 
We have added the orientations in Figure \ref{funcs} in order to make the action of 2-morphisms easier to visualize.
\begin{proof}
 First note that it is enough to show
that the correct relations hold if the functors are applied to
$M=P^{\kappa}_{\Bi}$ for any $(\Bi,\kappa)$. 

This can be proven by constructing an auxiliary category which
clearly has a $\tU$ action and which has $\alg^\bla$ as a quotient.
This category is quite close in spirit to $\tilde{\alg}^\bla$, but we
must use an enlargement of it.  Thus, we define a 2-category $\widetilde{\tU}$ whose \begin{itemize}
\item objects are weights,
\item 1-morphisms are sequences of
$\eE_i$'s, $\eF_i$'s and $\eI_\la$'s such that sum of the corresponding weights is the difference between target and image.  We translate these into
sequences of colored dots as usual by sending $\eI_\la$ to red dots
marked with $\la$.
\item 2-morphisms between two of these objects are
$\K$-linear combinations of immersed oriented diagrams where no red strands cross or self-intersect that
match, subject to the relations of Figures \ref{inv-rels},
\ref{pop-rels}, \ref{opp-cancel} and \ref{quiver-hecke}, and the
relations for $\alg^\bla$ (remember, all these relations are local and
imposed up to isotopy, but they do take into account orientations of
red and black strands.).  Furthermore, we must impose similar relations
between red strands and oppositely oriented red strands
  
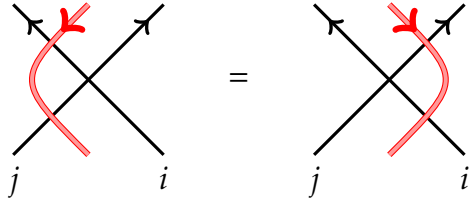
\begin{equation*}
    \begin{tikzpicture}[very thick]
      \draw[postaction={decorate,decoration={markings,
    mark=at position .9 with {\arrow[scale=1.3]{>}}}}] (-3,0)  +(1,-1) -- +(-1,1) node[at start,below]{$i$};
      \draw[postaction={decorate,decoration={markings,
    mark=at position .9 with {\arrow[scale=1.3]{>}}}}] (-3,0) +(-1,-1) -- +(1,1)node [at start,below]{$j$};
      \draw[wei,postaction={decorate,decoration={markings,
    mark=at position .9 with {\arrow[scale=1.3]{<}}}}] (-3,0)  +(0,-1) .. controls +(-1,0) ..  +(0,1);
      \node at (-1,0) {=};
      \draw[postaction={decorate,decoration={markings,
    mark=at position .9 with {\arrow[scale=1.3]{>}}}}] (1,0)  +(1,-1) -- +(-1,1) node[at start,below]{$i$};
      \draw[postaction={decorate,decoration={markings,
    mark=at position .9 with {\arrow[scale=1.3]{>}}}}] (1,0) +(-1,-1) -- +(1,1) node [at start,below]{$j$};
      \draw[wei,postaction={decorate,decoration={markings,
    mark=at position .9 with {\arrow[scale=1.3]{<}}}}] (1,0) +(0,-1) .. controls +(1,0) ..  +(0,1);   
 \end{tikzpicture}
  \end{equation*}
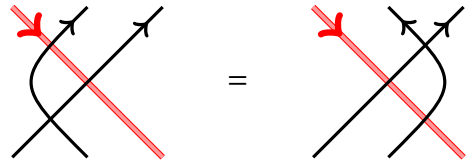
\begin{equation*}
    \begin{tikzpicture}[very thick,baseline=2.85cm]
      \draw[wei,postaction={decorate,decoration={markings,
    mark=at position .9 with {\arrow[scale=1.3]{<}}}}] (-3,3)  +(1,-1) -- +(-1,1);
      \draw[postaction={decorate,decoration={markings,
    mark=at position .9 with {\arrow[scale=1.3]{>}}}}] (-3,3)  +(0,-1) .. controls +(-1,0) ..  +(0,1);
      \draw[postaction={decorate,decoration={markings,
    mark=at position .9 with {\arrow[scale=1.3]{>}}}}] (-3,3) +(-1,-1) -- +(1,1);
      \node at (-1,3) {=};
      \draw[wei,postaction={decorate,decoration={markings,
    mark=at position .9 with {\arrow[scale=1.3]{<}}}}] (1,3)  +(1,-1) -- +(-1,1);
  \draw[postaction={decorate,decoration={markings,
    mark=at position .9 with {\arrow[scale=1.3]{>}}}}] (1,3)  +(0,-1) .. controls +(1,0) ..  +(0,1);
      \draw[postaction={decorate,decoration={markings,
    mark=at position .9 with {\arrow[scale=1.3]{>}}}}] (1,3) +(-1,-1) -- +(1,1);    \end{tikzpicture}
  \end{equation*}
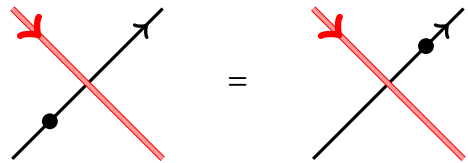
\begin{equation*}
    \begin{tikzpicture}[very thick]
  \draw[postaction={decorate,decoration={markings,
    mark=at position .9 with {\arrow[scale=1.3]{>}}}}] (-3,6) +(-1,-1) -- +(1,1);
  \draw[wei,postaction={decorate,decoration={markings,
    mark=at position .9 with {\arrow[scale=1.3]{<}}}}](-3,6) +(1,-1) -- +(-1,1);
\fill (-3.5,5.5) circle (3pt);
\node at (-1,6) {=};
 \draw[postaction={decorate,decoration={markings,
    mark=at position .9 with {\arrow[scale=1.3]{>}}}}] (1,6) +(-1,-1) -- +(1,1);
  \draw[wei,postaction={decorate,decoration={markings,
    mark=at position .9 with {\arrow[scale=1.3]{<}}}}](1,6) +(1,-1) -- +(-1,1);
\fill (1.5,6.5) circle (3pt);
    \end{tikzpicture}
  \end{equation*}
  
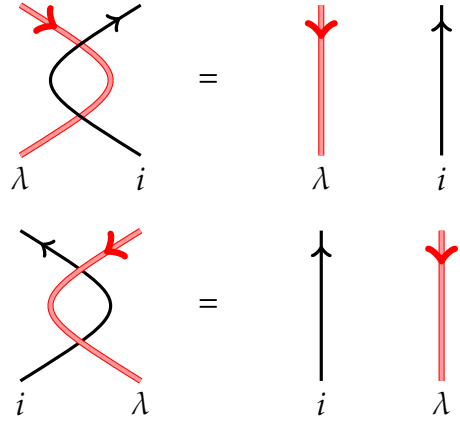
\begin{equation*}
  \begin{tikzpicture}[very thick,baseline=1.6cm]
    \draw[postaction={decorate,decoration={markings,
    mark=at position .9 with {\arrow[scale=1.3]{>}}}}] (-2.8,0)  +(0,-1) .. controls +(1.6,0) ..  +(0,1) node[below,at start]{$i$};
       \draw[wei,postaction={decorate,decoration={markings,
    mark=at position .9 with {\arrow[scale=1.3]{<}}}}] (-1.2,0)  +(0,-1) .. controls +(-1.6,0) ..  +(0,1) node[below,at start]{$\la$};
           \node at (-.3,0) {=};
    \draw[wei,postaction={decorate,decoration={markings,
    mark=at position .9 with {\arrow[scale=1.3]{<}}}}] (2.8,0)  +(0,-1) -- +(0,1) node[below,at start]{$\la$};
       \draw[postaction={decorate,decoration={markings,
    mark=at position .9 with {\arrow[scale=1.3]{>}}}}] (1.2,0)  +(0,-1) -- +(0,1) node[below,at start]{$i$};
          \draw[wei,postaction={decorate,decoration={markings,
    mark=at position .9 with {\arrow[scale=1.3]{<}}}}] (-2.8,3)  +(0,-1) .. controls +(1.6,0) ..  +(0,1) node[below,at start]{$\la$};
  \draw[postaction={decorate,decoration={markings,
    mark=at position .9 with {\arrow[scale=1.3]{>}}}}] (-1.2,3)  +(0,-1) .. controls +(-1.6,0) ..  +(0,1) node[below,at start]{$i$};
           \node at (-.3,3) {=};
    \draw [postaction={decorate,decoration={markings,
    mark=at position .9 with {\arrow[scale=1.3]{>}}}}](2.8,3)  +(0,-1) -- +(0,1) node[below,at start]{$i$};
       \draw[wei,postaction={decorate,decoration={markings,
    mark=at position .9 with {\arrow[scale=1.3]{<}}}}] (1.2,3)  +(0,-1) -- +(0,1) node[below,at start]{$\la$};
     
  \end{tikzpicture}
\end{equation*}
\end{itemize}  

This category acts on $\oplus_{\mu,\nu}\cata^\mu_\nu$ by the usual action of $\eE_i$ and $\eF_i$, and letting $\eI_\la$ act by sending $M$ to the same module considered as a module over $\alg^{\mu+\la}$.  On the level of 2-morphisms, this action sends the crossing 
$\begin{tikzpicture} [thick,scale=1.3,baseline] \draw[wei] (0,0) -- (.3,.3); \draw (.3,0) -- (0,.3); \end{tikzpicture}$  to the obvious projection map $\eF_i\eI_\la\to \eI_\la\eF_i$ and 
$\begin{tikzpicture} [thick,scale=1.3,baseline]
          \draw[wei] (.3,0) -- (0,.3); \draw
          (0,0) -- (.3,.3); 
        \end{tikzpicture}$
to the map multiplying at the bottom by $\la^i$ dots on the new strand
formed by applying $\fF_i$.

In particular, the inclusion of $\tU$ by horizontally composing with any set of red lines to the left is injective.
It follows by the same arguments as Theorem \ref{basis} that $\widetilde{\tU}$ has a basis analogous to that of Khovanov and Lauda for $\tU$.  

Now consider the $\tU$-module subcategory of $\widetilde{\tU}$ where the red lines are fixed to have labels $\bla$ in order, and consider its quotient by all 1-morphisms of the form $A\eE_i$ and all 2-morphisms given by positive degree bubbles at the far left of the diagram.  The argument that the idempotent completion of this category is the category of projective $\alg^\bla$ modules is precisely the same as the proof of Proposition \ref{quotient-is-cyc}.
\end{proof}

This shows, in particular, that $K_0(\alg^\bla)$ is a module over $U_q^\Z(\fg)$, which we will show in the next section is isomorphic to the tensor product $V_\bla^\Z$.

\section{Standard modules}
\label{sec:standard}

\subsection{Standard modules defined}

When analyzing the structure of re\-pres\-en\-ta\-tion-theoretic categories,
such as the categories $\cO$ appearing in Stroppel's construction of
Khovanov homology \cite{Str06b}, a crucial role is played by the Verma modules and their analogues. The property of ``having objects like Verma
modules'' was formalized by Cline-Parshall-Scott as the property of
being {\bf quasi-hereditary} \cite{CPS88}.  Unfortunately, this is too
strong of an assumption for us; as we noted earlier, the cyclotomic
QHA is Frobenius, and thus very far from being
quasi-hereditary (any ring which is both Frobenius and quasi-hereditary is semi-simple).

Luckily, our categories satisfy a weaker condition: they are {\bf
  standardly stratified}, as defined by the same authors \cite{CPS96}.
To show this, we must construct a collection of modules which are
called {\bf standard}, and show that projectives have a filtration by
these modules compatible with a preorder.

We define a preorder on $(\Bi,\kappa)$'s by calling $(\Bi,\kappa)\leq (\Bi',\kappa')$ if
\begin{equation*}
\sum_{k\leq\kappa(j)}\al_{i_k}\leq\sum_{k\leq\kappa'(j)}\al_{i_k'}\quad\text{ for all $j\in[1,\ell]$}.
\end{equation*}
This preorder can be packaged as the dominance order for a function $\bal_{\Bi,\kappa}\colon[1,\ell]\to \rola(\fg)$ which we call a {\bf root function} given by
\begin{equation*}
  \bal_{\Bi,\kappa}(k)=\sum_{\kappa(k-1)<j\leq \kappa(k)}\al_{i_j}.
\end{equation*} Note that this preorder is
entirely insensitive to permutations of the black strands which do not
cross any red strands.

\begin{defn}
  By convention, we call a red/black crossing where black strands go
  from NW to SE {\bf left} and the mirror image of such a crossing
  {\bf right}.

Note that this terminology does not apply to black/black crossings; if
we call a crossing left or right we are implicitly assuming it is black/red.
\end{defn}
  \begin{center}
    \begin{tikzpicture}
      \node at (-3,0) [label=below:{a ``left'' crossing}]
      {\begin{tikzpicture} [very thick,scale=2.3] \draw[wei] (0,0) --
          (.3,.3); \draw (.3,0) -- (0,.3);
        \end{tikzpicture}}; \node at (3,0)[label=below:{a ``right''
        crossing}] {\begin{tikzpicture} [very thick,scale=2.3]
          \draw[wei] (.3,0) -- (0,.3); \draw
          (0,0) -- (.3,.3); 
        \end{tikzpicture}};
    \end{tikzpicture}
  \end{center}

The significance of these definitions is that a map induced between
projectives by adding a left crossing on the bottom always sends
a projective to one smaller in the preorder $\leq$, and \emph{vice versa} for
right crossings.   We will call
a black strand that makes a left crossing below all right crossings {\bf
  standardly violating}.

Let $L^\kappa_\Bi\subset P_{\Bi}^{\kappa}$ be the submodule generated by diagrams
with no right crossings, and at least one left crossing; that is, the module spanned by all diagrams with standardly violating strands. 
\begin{prop}\label{st-trace}
  The image of any map from a projective higher than $(\Bi,\kappa)$ in
  the preorder $\leq $ is contained in $L^\kappa_\Bi\subset P_\Bi^\kappa$, and
  these images generate $L^\kappa_\Bi$.  That is, the submodule $L^\kappa_\Bi$ is the
  ``trace'' of these projectives.
\end{prop}
\begin{proof}
  Generation is clear: any diagram with only left crossings defines a
  map from a higher projective to  $P_\Bi^\kappa$ with the image of
  the idempotent being the original diagram.

  To show that any such image lands in $L^\kappa_\Bi$, consider an arbitrary map
  from a higher projective.  This is given by a sum of diagrams in
  $P_\Bi^\kappa$ whose upper end points are given by the idempotent
  for that projective. Now, apply Proposition \ref{straighten}
  with a set of representatives that do all crossings within blocks
  consisting of a red strand and the black strands to its immediate
  left before doing any others.  By the definition of the preorder, all
  these diagrams must have at least one left crossing which occurs
  before we make any crossings between these blocks, and all
  right crossings involve strands from different blocks; thus the image lies in $L^\kappa_\Bi$.
\end{proof}
 \begin{defn}
   We define $S^\kappa_{\Bi}=P_{\Bi}^{\kappa}/L^\kappa_\Bi$ to be the {\bf standard
     module} for $\kappa$ and $\Bi$.
\end{defn}
Proposition \ref{st-trace} shows that this matches the definition of a
standard module for an algebra with preorder on its projectives given
in (for instance) \cite{MS}, so our terminology matches theirs.
Below, when we speak of a {\bf group} of black strands, we will always
mean the set of black strands which originate between two consecutive
red strands at the bottom of the diagram.

 Let $e_\bal$ be the idempotent which is 1 on
  projectives $P^\kappa_\Bi$ with $\bal_{\Bi,\kappa}=\bal$.  We let
  $S_\bal$ be the standard quotient of the projective $e_\bal \alg^\bla$.
Let $\cC^{\bal}$ be the subcategory of modules with a presentation in $\mathrm{add}(S_\bal)$ for fixed
  $\bal$.  

\begin{figure}[ht]
    \begin{tikzpicture}[xscale=2,yscale=1.2]
      \node[draw,rectangle,inner xsep=100pt,inner ysep=20pt] (a) at
      (0,0){$e_\bal \alg^\bla$}; \node [draw,rectangle,inner
      xsep=20pt,inner ysep=10pt] (b) at (-1.2,-1.7){$\alg_{\bal(1)}^{\la_1}$};
      \node [draw,rectangle,inner xsep=20pt,inner ysep=10pt] (c) at
      (1.2,-1.7){$\alg_{\bal(\ell)}^{\la_\ell}$};
      \draw [very thick,wei] (a.-45) -- +(0,-1.8); \draw [very
      thick,wei] (a.-135) -- +(0,-1.8); \draw[very thick,wei] (a.south
      west) -- +(0,-1.8); \draw[very thick] (b.36) to[out=90,in=-90]
      (a.-143); \draw[very thick] (b.144) to[out=90,in=-90]
      (a.-165.5); \draw[very thick] (c.144) to[out=90,in=-90] (a.-37);
      \draw[very thick] (c.36) to[out=90,in=-90] (a.-15); \draw[ultra
      thick,loosely dotted] (-.9,-1) -- (-1.5,-1); \draw[ultra
      thick,loosely dotted] (.9,-1) -- (1.5,-1); \draw[ultra
      thick,loosely dotted] (-.9,-2.4) -- (-1.5,-2.4); \draw[ultra
      thick,loosely dotted] (.9,-2.4) -- (1.5,-2.4); \draw[ultra
      thick,loosely dotted] (.34,-1.6) -- (-.34,-1.6); \draw[very
      thick] (b.-36) -- +(0, -.4); \draw[very thick] (b.-144) -- +(0,
      -.4); \draw[very thick] (c.-144)-- +(0, -.4); \draw[very thick]
      (c.-36)-- +(0, -.4);
    \end{tikzpicture}
    \caption{The action of $\alg_{\bal(1)}^{\la_1}\otimes\cdots \otimes
      \alg_{\bal(\ell)}^{\la_\ell}$ on $e_\bal \alg^\bla$.}
    \label{stand}
  \end{figure}

Acting by black-black crossings on just each group of strands as in
Figure~\ref{stand} gives a map $\alg_{\al(1)}^{\la_1}\otimes\cdots
\otimes \alg_{\al(\ell)}^{\la_\ell}\to \End_{\alg^\bla}(S_\bal)$, so we can think of $S_{\bal}$ as a $\alg_{\al(1)}^{\la_1}\otimes\cdots
\otimes \alg_{\al(\ell)}^{\la_\ell}-\alg^{\bla}_{\al}$-bimodule, and $S=\oplus_{\bal} S_\bal$ as a $\alg^{\la_1}\otimes\cdots
\otimes \alg^{\la_\ell}-\alg^{\bla}$-bimodule.
\begin{defn}
  The {\bf standardization functor} is the tensor
  product with this bimodule:
  \begin{equation*}
    \mathbb{S}^{\bla}(-)=-\otimes_{\alg^{\la_1}\otimes\cdots
\otimes \alg^{\la_\ell}}S\colon\cata^{\la_1;\dots;\la_\ell}\to \cata^{\bla}
  \end{equation*}
\end{defn}

More generally, we can construct partial standard modules, where we
only kill the left crossings for some of the red strands.  This will
give us a standardization functor   \begin{equation*}
    \mathbb{S}^{\bla_1;\dots;\bla_m}:\cata^{\bla_1;\dots;\bla_\ell}\to \cata^{\bla}
  \end{equation*}
for any sequence of sequences
$\bla_1,\dots,\bla_m$ such that the concatenation $\bla_1\cdots\bla_m$
is equal to $\bla$.

Of particular interest is the standardization functor which corresponds to adding a new red strand labeled $\mu$ and no black ones, since this categorifies the inclusion of $V_{\bla}\otimes\{v_{high}\}\hookrightarrow V_{\bla}\otimes V_{\mu}$.  We denote this functor $\mathbb{S}^{\bla;\mu}(-\boxtimes P_\emptyset)=\fI_{\mu}$.

We can think of the standardization functor as a (very far from full)
inclusion of the naive tensor product category into $\cata^\bla$.  This
functor is full when only considered on objects corresponding to a single root function, but there are, of course, many
``new'' maps between the different values.

\subsection{Decategorification}
\label{sec:decat}

In order to understand the Grothendieck group $K_0(\alg^\bla)$, we need to better understand its Euler form.  In particular, we need a candidate bilinear form on $V_\bla$.  There is a system of
non-degenerate $U_q(\fg)$-invariant sesquilinear forms
$\langle,\rangle$ on all tensor products $V_\bla^\Z$ defined by 
$\langle v,w\rangle=\langle \Theta^{(\ell)}v,w\rangle_p,$ where $\Theta^{(\ell)}$ is the $\ell$-fold {\bf quasi-R-matrix} and $\langle -,-\rangle_p$ is the term-wise $q$-Shapovalov form.  The usual quasi-R-matrix on two tensor factors is defined in \cite[\S 4]{Lusbook}; the $\ell$-fold one is defined inductively by $\Theta^{(\ell)}=(\Theta^{(2)}\otimes 1^{\otimes \ell-2})\cdot\Delta\otimes 1^{\otimes \ell-2}(\Theta^{(\ell-1)})$. 
\begin{prop}
This form is non-degenerate, and $\tau$-Hermitian in the sense that we have $\langle u\cdot v,v'\rangle=\langle v,\tau(u)\cdot v'\rangle$ for any $v,v'\in V_\bla$ and $u\in U_q(\fg)$, where $\tau$ is the antiautomorphism defined in Section \ref{sec:cyc}.

Furthermore, for any $j<\ell$, the natural map
$V_{\la_1}\otimes\cdots\otimes
V_{\la_j}\otimes\{v_h^{j+1}\}\otimes\cdots\otimes \{v_h^\ell\} \hookrightarrow V_\bla$ is an isometric embedding. 
\end{prop}
\begin{proof}
The first statement follows from 
\begin{multline*}\langle u\cdot v,v'\rangle=\langle
  \Theta^{(\ell)}\Delta(u)v,v'\rangle_p=\langle \bar
  \Delta(u)\Theta^{(\ell)} v,v'\rangle_p=\langle
  \Theta^{(\ell)}v,(\tau\otimes\cdots \otimes \tau)\bar \Delta(u)v'\rangle_p\\ =\langle \Theta^{(\ell)} v,\Delta(\tau(u))v'\rangle_p=\langle v,\tau(u)\cdot v'\rangle.\end{multline*}

The second reduces to case of two factors, since $\langle -,-\rangle$ is a multiple of the $q$-Shapovalov form on any simple submodule of a tensor product. In this case it follows form the fact that $\Theta^{(2)}\in U^-\otimes U^+$ and $\Theta^{(2)}_0=1\times 1$, so $\Theta^{(2)}$ fixes $v\otimes v_{high}$. 
\end{proof}

Let $v_{\Bi}^\kappa\in V_\bla$ be defined inductively by
\begin{itemize}
\item if $\kappa(\ell)=n$, then
  $v_{\Bi}^\kappa=v_{\Bi}^{\kappa^-}\otimes v_\ell$ where $v_\ell$ is
  the highest weight vector of $V_{\la_\ell}$, and $\kappa^-$ is the
  restriction to $[1,\ell-1]$.
\item If $\kappa(\ell)\neq n$, so $v_{\Bi}^\kappa=F_{i_n}v^{\kappa}_{\Bi^-}$, where
  $\Bi^-=(i_1,\dots,i_{n-1})$. 
\end{itemize}

\begin{thm}\label{Uq-action}
There is a canonical isomorphism $\eta:K_0(\alg^\bla)\to V_\bla^\Z$ given by
$[P^\kappa_\Bi]\mapsto  v_{\Bi}^\kappa$ intertwining the inner product defined above with the Euler form.
\end{thm}
\begin{proof}
First, note that 
$$\dim_q \Hom(P^{\kappa}_{\Bi},P^{\kappa'}_{\Bi'})=\langle
v_{\Bi'}^{\kappa'},v_{\Bi'}^{\kappa'}\rangle.$$  
We prove this by induction on $n$ and $\ell$.  Unless $n=\kappa(\ell)=\kappa'(\ell)$, we can move a $\fF_i$ from one side to become a $\fE_i$ on the other (up to shift).  The decompositions of $\fE_iP^{\kappa}_{\Bi}$ into $P^{\kappa''}_{\Bi''}$'s matches that of the vector since both are done using the commutation relations between $\fE_i$ and $\fF_i$ or $E_i$ and $F_i$, which we already know match.

If  $n=\kappa(\ell)=\kappa'(\ell)$, then the dimension of the $\Hom$-space and the inner product are both unchanged by simply removing the red line (obviously, this holds if we use $\tilde{P}^\kappa_\Bi$ and $\tilde{P}^{\kappa'}_{\Bi'}$ by the basis of Theorem \ref{basis}, and the isomorphism only sends elements with violating strands to elements with violating strands).  This shows the equality.

Thus, if we are given any linear relation satisfied by
$[P^\kappa_\Bi]$'s, the corresponding linear combination of
$v_{\Bi}^\kappa$'s is in the kernel of this form, and thus 0 in
$V_\bla$.  Thus, the map $\eta$ is well-defined and surjective.

A surjective map of finitely generated free $\Z[q,q^{-1}]$ modules is
an isomorphism if and only if they have the same rank (the kernel must
be a summand, which is zero if and only if its complement has the rank
of the whole module).  Thus, we need only prove that $\cata^\bla_\nu$ has at most $\dim (V_\bla)_\nu$ simple modules.

Consider a simple module $L$, and let $\bal$ be maximal among root functions for which $Le_\bal\neq 0$.  Let $K$ be any simple constituent of $Le_\bal$ as a $\alg^{\la_1}\otimes\cdots
\otimes \alg^{\la_\ell}$-module.  Then, by adjunction, we have a surjective map $\mathbb{S}^\bla(K)\to L$.  As modules over $\alg^{\la_1}\otimes\cdots
\otimes \alg^{\la_\ell}$, we have a surjective map $K\to \mathbb{S}^\bla(K)e_\bal$ which is an isomorphism, so by assumption composing with the map to $L$ gives an inclusion of $K$.  This implies that every proper submodule of $\mathbb{S}^\bla(K)$ is killed by $e_\bal$; thus, the sum of all proper submodules is itself killed by $e_\bal$ and thus proper itself.  This implies that the cosocle of $\mathbb{S}^\bla(K)$ is simple and thus $L$.  Thus $L$ is uniquely determined by $K$, and there no more simple objects in $\cata^\bla$ than there are in $\cata^{\la_1;\dots,;\la_\ell}$, which has exactly $\dim V_\bla$ simple modules.  Thus, the ranks coincide, and we are done.
\end{proof}

Now, we let $s^\kappa_\Bi=F_{i_{\kappa(2)}}\cdots F_{i_1}v_1\otimes
  \cdots \otimes F_{i_n}\cdots F_{i_{\kappa(\ell)+1}}v_{\ell}$.

\begin{prop}
  $\displaystyle \eta([S^\kappa_\Bi])=s^\kappa_\Bi$. 
\end{prop}
\begin{proof}
This proof depends on two inequalities, which we will use to ``squeeze'' the inner products of the two sides of the equality with projectives.  First, we prove by induction that
\begin{equation}\label{sta-ineq}
\dim_q \Hom(P^{\kappa_1}_{\Bi_1},S^{\kappa_2}_{\Bi_2})\leq
\langle
v_{\Bi_1}^{\kappa_1},s_{\Bi_2}^{\kappa_2}\rangle.
\end{equation}
 Consider the restriction of a standard module
  $\fE_iS^\kappa_\Bi$. This carries a filtration by submodules $q_i$
  where $q_i$ is the submodule generated by the collection of diagrams
  where the rightmost strand at the top lands to the right of the
  $i$th strand and left of the $i+1$st at the bottom.

 \begin{figure}[ht]
    \centering
    \begin{equation*}
      \begin{tikzpicture}[very thick,yscale=1]
        \draw[wei] (1,-1) -- +(0,2) node[at
        start,below]{$\la_1$} node[at end,above]{$\la_1$}; \draw
        (1.5,-1) -- +(0,2); \node at (2.3,0){$\cdots$}; \draw[wei]
        (3.5,-1) -- +(0,2) node[at start,below]{$\la_{m+1}$} node[at
        end,above]{$\la_{m+1}$}; \draw (4,-1) -- +(0,2); \draw (6,-1) --
        +(0,2); 
        \node at (5,0){$\cdots$}; \draw[postaction={decorate,decoration={markings,
    mark=at position .5 with {\arrow[scale=1.3]{>}}}}] (6.7,-1) to[in=40,out=140] node[at start,below=2pt]{$i$}
        (2.8,-1) node[at end,below=2pt]{$i$} ;
      \end{tikzpicture}
    \end{equation*}
    \caption{The filtration on $\fE_iS^\kappa_\Bi$.}
    \label{el-p}
  \end{figure}

  We let $\kappa_m$ and $\Bi_m$ be associated to the sequence pictured
  at the bottom of Figure \ref{el-p}. 
  Then we have a natural map 
\begin{multline}
S_i=\mathbb{S}^{\bla}(\cdots \boxtimes P_{\Bi^{m-1}}\boxtimes \fE_iP_{\Bi^m}\boxtimes P_{\Bi^{m+1}} \boxtimes\cdots )\left(\sum_{j=1}^{m-1}\langle\al_i,\la_j-\bal(j)\rangle\right) \to q_i/q_{i+1}.
\end{multline}
sending a $\boxtimes$ of diagrams to the horizontal composition of those diagrams with the strand attaching to $\fE_i$ pulled through the bottom of all the diagrams to its right (see Figure \ref{horiz-map}). 
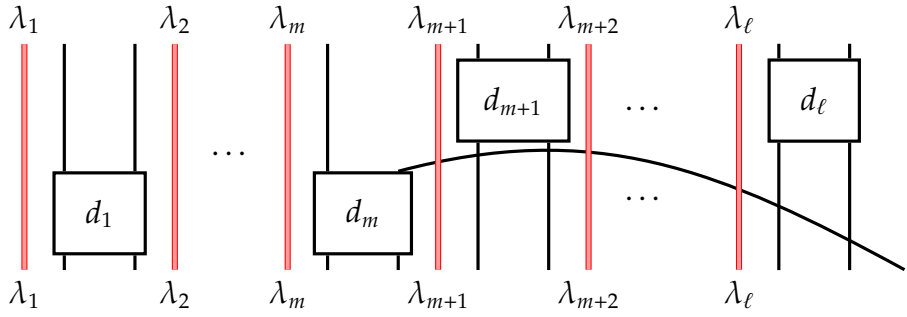
\begin{figure}
\centering
    \begin{equation*}
      \begin{tikzpicture}[very thick,yscale=1.5]
        \draw[wei] (1,-1) -- +(0,2) node[at
        start,below]{$\la_1$} node[at end,above]{$\la_1$}; 
  \node (a) [inner xsep=12pt, inner ysep=10pt,draw] at (2,-.5){$d_1$};
\draw (a.-130) -- +(0,-.13);
\draw (a.-50) -- +(0,-.13);
\draw (a.130) -- +(0,1.13);
\draw (a.50) -- +(0,1.13);
        \draw[wei] (3,-1) -- +(0,2) node[at
        start,below]{$\la_2$} node[at end,above]{$\la_2$}; 

\draw[wei] (4.5,-1) -- +(0,2) node[at start,below]{$\la_m$} 
  node[at   end,above]{$\la_m$}; 
\node (a) [inner xsep=12pt, inner ysep=10pt,draw] at (5.5,-.5){$d_m$};
\draw (a.-130) -- +(0,-.13);
\draw (a.-50) -- +(0,-.13);
\draw (a.130) -- +(0,1.13);
\draw (12.7,-1) to[in=10,out=160]   (a.50);
  \draw[wei] (6.5,-1) -- +(0,2) node[at start,below]{$\la_{m+1}$} 
  node[at   end,above]{$\la_{m+1}$}; 
  \node (a) [inner xsep=10pt, inner ysep=10pt,draw] at (7.5,.5){$d_{m+1}$};
\draw (a.-130) -- +(0,-1.13);
\draw (a.-50) -- +(0,-1.13);
\draw (a.130) -- +(0,.13);
\draw (a.50) -- +(0,.13);
\draw[wei] (8.5,-1) -- +(0,2) node[at start,below]{$\la_{m+2}$} 
  node[at   end,above]{$\la_{m+2}$}; 
   \node at (9.25,.4){$\cdots$};
   \node at (9.25,-.4){$\cdots$}; 
  \draw[wei] (10.5,-1) -- +(0,2) node[at start,below]{$\la_{\ell}$} 
  node[at   end,above]{$\la_{\ell}$}; 
  \node (a) [inner xsep=12pt, inner ysep=10pt,draw] at (11.5,.5){$d_\ell$};
\draw (a.-130) -- +(0,-1.13);
\draw (a.-50) -- +(0,-1.13);
\draw (a.130) -- +(0,.13);
\draw (a.50) -- +(0,.13);
        \node at (3.75,0){$\cdots$}; 

      \end{tikzpicture}
    \end{equation*}
\caption{The map to $q_m/q_{m+1}$}
\label{horiz-map}
\end{figure}
This map is clearly surjective, so applying the induction hypothesis, we see that
\begin{multline*}\dim \Hom(\fF_iP^{\kappa}_{\Bi},S^{\kappa'}_{\Bi'})=\dim \Hom(P^{\kappa}_{\Bi},\fE_iS^{\kappa'}_{\Bi'})\leq \sum_{j=1}^{\ell}\dim \Hom(P^{\kappa}_{\Bi},S_j)\\ \leq \sum_{j=1}^{\ell}\langle
v_{\Bi_1}^{\kappa},E_i^{(j)}s_{\Bi'}^{\kappa'}\rangle_1=\langle
v_{\Bi}^{\kappa},E_is_{\Bi'}^{\kappa'}\rangle_1=
\langle
F_iv_{\Bi}^{\kappa},s_{\Bi'}^{\kappa'}\rangle_1,
\end{multline*}
where $E_i^{(j)}$ is $E_i$ just acting in the $j$th tensor factor.

If $\kappa(\ell)\neq n$, then we can write $P^\kappa_\Bi$ as the image of a $\fF_i$, and this shows the induction step.
If $\kappa(\ell)=n$, then either Hom to a standard is 0, or the red strand can be removed from both.  This shows the inequality (\ref{sta-ineq}).

Now, we move on to showing the equality $$\dim_q \Hom(P^{\kappa_1}_{\Bi_1},S^{\kappa_2}_{\Bi_2})=
\langle
v_{\Bi_1}^{\kappa_1},s_{\Bi_2}^{\kappa_2}\rangle.$$
This will immediately imply the desired result by non-degeneracy of the Euler form.

Consider the module
  \begin{math}
    \fF_iS^\kappa_\Bi.
  \end{math}
equipped with the filtration consisting of
  submodules $p_m$ generated by diagrams where the black strand
  starting at the far right never passes left of the $m$th red strand.

Acting on the element $x_m$ depicted in the Figure \ref{sta-filt} induces a map $$S^{\kappa_m}_{\Bi_m}\left(-\sum_{j=m+1}^{\ell}\langle\al_i,\la_j-\bal(j)\rangle\right)\to p_m/p_{m-1}$$ which is clearly surjective.   \begin{figure}[ht]
    \centering
    \begin{equation*}
      \begin{tikzpicture}[very thick,yscale=1]
        \draw[wei] (1,-1) -- +(0,2) node[at
        start,below]{$\la_1$} node[at end,above]{$\la_1$}; \draw
        (1.5,-1) -- +(0,2); \node at (2.3,0){$\cdots$}; \draw[wei]
        (3.5,-1) -- +(0,2) node[at start,below]{$\la_{m+1}$} node[at
        end,above]{$\la_{m+1}$}; \draw (4,-1) -- +(0,2); \draw (6,-1) --
        +(0,2); 
        \node at (5,0){$\cdots$}; \draw (6.7,-1) to[in=-60,out=140] node[at start,below=2pt]{$i$}
        (2.6,1) node[at end,above=2pt]{$i$} ;
      \end{tikzpicture}
    \end{equation*}
    \caption{The element $x_m$ inducing the filtration on $\fF_iS^\kappa_\Bi$.}
    \label{sta-filt}
\end{figure}

Thus, we obtain a second inequality  \begin{multline*}\dim \Hom(\fE_iP^{\kappa}_{\Bi},S^{\kappa'}_{\Bi'})=\dim \Hom(P^{\kappa}_{\Bi},\fF_iS^{\kappa'}_{\Bi'})\leq \sum_{j=1}^{\ell}\dim \Hom(P^{\kappa}_{\Bi},S^{\kappa'_j}_{\Bi'_j})\\ \leq \sum_{j=1}^{\ell}\langle
v_{\Bi}^{\kappa},s_{\Bi_j'}^{\kappa_j'}\rangle_1=\langle
v_{\Bi}^{\kappa},F_is_{\Bi'}^{\kappa'}\rangle_1=
\langle
E_iv_{\Bi}^{\kappa},s_{\Bi'}^{\kappa'}\rangle_1.\end{multline*}
Since the initial and final quantities are equal by induction, the above can only hold if the inequality (\ref{sta-ineq}) is always an equality.
\end{proof}

We note that we have now shown that the morphisms between standard modules and successive quotients of $\fF_iS^\kappa_\Bi$ and $\fE_iS^\kappa_\Bi$ must be isomorphisms for dimension reasons.
\begin{cor}  For any standard module $S^\kappa_\Bi$, the modules
  $\fF_iS^\kappa_\Bi$ and $\fE_iS^\kappa_\Bi$ possess 
standard filtrations that categorify the
  identities \begin{multline*} \Delta^{(\ell)}(E_i)=E_i\otimes
    1\otimes \cdots \otimes 1+\tilde K_i\otimes E_i\otimes 1\otimes
    \cdots \otimes 1+ \cdots + \\ \tilde K_i\otimes\cdots \otimes
    \tilde K_i\otimes E_i \otimes 1+\tilde K_i\otimes \cdots\otimes
    \tilde K_i\otimes E_i.\end{multline*}
  \begin{multline*}
    \Delta^{(\ell)}(F_i)=F_i\otimes \tilde K_{-i} \otimes \cdots
    \otimes \tilde K_{-i}+1\otimes F_i\otimes \tilde K_{-i} \otimes
    \cdots \otimes \tilde K_{-i}+\cdots +\\ 1\otimes \cdots \otimes
    1\otimes F_i\otimes \tilde K_{-i}+ 1\otimes \cdots \otimes
    1\otimes F_i.
  \end{multline*}
\end{cor}

This result also shows that the exactness of the standardization
functor:
\begin{prop}\label{standard-exact}
  The standardization functor
  $\mathbb{S}^{\bla_1;\dots;\bla_m}:\cata^{\bla_1;\dots;\bla_\ell}\to
  \cata^{\bla}$ is exact.
\end{prop}
\begin{proof}
Note that we need only consider the case where $m=\ell$ and
$\bla_i=(\la_i)$.  We induct as in the proof of Theorem \ref{Uq-action} on $n$ and
$\ell$.  
  It suffices to prove that
  $\Hom(P_\Bi^\kappa,\mathbb{S}^{\bla}(-))$ is always exact since
  every indecomposable projective is a summand of $P_\Bi^\kappa$.  

Unless $n=\kappa(\ell)$, the projective $P_\Bi^\kappa$ is a sum of summands of modules of the form $\fF_i(P')$.
  Thus, we can use the
  adjunction \[\Hom(\fF_i(P'),\mathbb{S}^{\bla}(-))\cong
  \Hom(P',\fE_i\mathbb{S}^{\bla}(-)).\]  Since
  $\fE_i\mathbb{S}^{\bla}(M)$ is filtered by the
  modules $\mathbb{S}^{\bla}({}_j\fE_iM)$ where
  ${}_j\fE_i$ is the categorification functor applied in the $j$th
  tensor factor.  
By induction, we have that
$\mathbb{S}^{\bla}({}_j\fE_i(-))$ is exact, so this
establishes this induction step.

If $n=\kappa(\ell)$, then $
\Hom(P_\Bi^\kappa,\mathbb{S}^{\bla}(M))$ is the same as
$\Hom(P_\Bi^{\kappa^-},\mathbb{S}^{(\la_1,\dots, \la_{\ell-1})}(M^+))$
where $M^+$ is the $\alg^{\la_1}\otimes \cdots\otimes
\alg^{\la_{\ell-1}}$ submodule in $M$ where the weight for
$\alg^{\la_\ell}$ is $\la_\ell$.  Since $M\mapsto M^+$ is exact (it is
the projection of a sum of idempotents), by induction $M\mapsto
\mathbb{S}^{(\la_1,\dots, \la_{\ell-1})}(M^+)$ is exact as well.  This
completes the induction step, and thus the proof.
\end{proof}

\subsection{Simple modules and crystals}
\label{sec:crystal}

Lauda and Vazirani show that there is a natural crystal structure on simple representations of $R^\la=\alg^\la$, which is isomorphic to the usual highest weight crystal $\mathcal{B}(\la)$.  A similar crystal structure exists for simples of $\alg^\bla$; we denote the set of isomorphism classes of simple modules by $\mathcal{B}^\bla$.

\nc{\te}{\tilde{e}}
\nc{\tf}{\tilde{f}}
\nc{\soc}{\operatorname{soc}}

Note that we have a candidate for a map $$h\colon \mathcal{B}^{\la_1}\times\cdots \times \mathcal{B}^{\la_\ell}\to \mathcal{B}^\bla,$$
given by $\cosoc \mathbb{S}^{\bla}(L_1\boxtimes\cdots\boxtimes L_\ell)$; it's not immediately obvious that this module is simple, but in fact, we have already shown in the proof of Theorem \ref{Uq-action} that this map is well-defined, and surjective.  Since $\cata^{\bla}$ and $\cata^{\la_1;\dots,;\la_\ell}$ has the same number of simples, we have that

\begin{thm}\label{h-bijection}
The map $h$ defines a bijection.
\end{thm}

For a simple module $L$, the modules $\cosoc(\fF_iL)$, and $\soc(\fE_iL)$ are both several copies of a single simple module.  We define $\tf_i(L)$ and $\te_i(L)$ to be these simples.

\begin{thm}
  These operators make the classes of the simple modules a perfect basis of $K_0(\alg^\bla)$ in the sense of Berenstein and Kazhdan \cite[Definition 5.30]{BeKa}.  In particular, they define a crystal structure on simple modules.
\end{thm}
\begin{proof}
This proof uses entirely standard techniques.  If $a$ is the largest integer such that $\te_i^a(L)\neq 0$, then $\fE_i^a(L)$ is semi-simple; in fact, it is a sum of copies of $\te_i^a(L)$ (since $\fF_i^{(a)}(\te_i^a(L))$ surjects onto $L$).  In particular, any other simple constituent of $\fE_i(L)$ is killed by $\te_i^{a-1}$.  This is the definition of a perfect basis.
\end{proof}

\begin{prop}
Any simple module $L\in \mathcal{B}^\bla$ is isomorphic to its dual: $L\cong L^\star$.
\end{prop}
\begin{proof}
\excise{First, we note that if $L$ is self-dual, then so is $\fF_iL$, since $\fF_i$ commutes with $\star$.  
Thus, the socle of $\fF_iL$ is isomorphic to the dual of the cosocle.
On the other hand, since $\fE_i$ and $\fF_i$ are biadjoint, we have an
induced map $\fF_iL\to \fF_iL$ sending the cosocle to the socle, which
induces an isomorphism from the obvious quotient copy of $L$ to the
obvious sub copy of $L$ inside $\fE_i\fF_iL$ (after applying $\fE_i$
to the map).  This map thus must induce an isomorphism between a
summand of the socle to a summand of the cosocle.  Thus,
$\tilde{f}_i(L)$ is self-dual.  Of course, an analogous argument shows
that $\tilde{e}_i$ also preserves self-dual modules.  Thus, the result
is clear for $\mathcal{B}^\la$.}
This was shown for $\mathcal{B}^\la$ by Khovanov and Lauda in \cite[\S
3.2]{KLI}.

Now, consider the general case.  On the subcategory of modules killed by $e_{\bal'}$ for $\bal'\nleq\bal$, the functor $M\mapsto Me_\bal$ is a functor to $\cata^{\la_1;\dots;\la_\ell}$, which has a left adjoint $\mathbb{S}^\bla$ and right adjoint $\star\circ \mathbb{S}^\bla\circ \star$.  Obviously, the socle of $\mathbb{S}^\bla(L_1\boxtimes\cdots\boxtimes L_\ell)^\star$ and the cosocle of $\mathbb{S}^\bla(L_1\boxtimes\cdots\boxtimes L_\ell)$ are dual simple modules.  

On the other hand, we have a map 
\[\mathbb{S}^\bla(L_1\boxtimes\cdots\boxtimes L_\ell)\to \mathbb{S}^\bla(L_1^\star\boxtimes\cdots\boxtimes L_\ell^\star)^\star\]
This map is non-zero, since the induced map \[\mathbb{S}^\bla(L_1\boxtimes\cdots\boxtimes L_\ell)e_\bal\to\mathbb{S}^\bla(L_1^\star\boxtimes\cdots\boxtimes L_\ell^\star)^\star e_\bal\] is the identity map from $L_1\boxtimes\cdots\boxtimes L_\ell$ to $L_1\boxtimes\cdots\boxtimes L_\ell$. On the other hand, its image lies in the socle, and thus is an isomorphism from $h(L_1,\dots,L_\ell)$ to its dual, since we already know $L_i$ is self-dual.
\end{proof}

Since $K_0(\alg^\bla)\cong V_\bla$, this implies that an isomorphism of crystals exists between $\mathcal{B}^\bla$ and  $\mathcal{B}^{\la_1}\times \cdots\times \mathcal{B}^{\la_\ell}$ without actually determining what it is.
\begin{conj}
  The crystal structure induced on $\mathcal{B}^\bla$ by $h$ has
  Kashiwara operators given by $\tf_i$ and $\te_i$, where
  $\mathcal{B}^{\la_1}\times \cdots\times \mathcal{B}^{\la_\ell}$ is
  endowed with the tensor product crystal structure.
\end{conj}
This conjecture is shown in recent work of the author and Losev \cite[7.2]{LoWe}.

Choose any infinite sequence $\{i_1,\dots,\}$ of simple roots such
that each root appears infinitely often.  For any element $v$ of a
highest weight crystal, there are unique integers $\{a_1,\dots\}$ such
that $\cdots \te_{i_2}^{a_2}\te_{i_1}^{a_1}v=v_{high}$ and
$\te_k^{a_k+1}\cdots \te_{i_1}^{a_1}v=0$; the parametrization of the
elements of the crystal by this tuple is called the ``string
parametrization.''

Our system of projectives $P^\kappa_\Bi$ is quite redundant; there are many more of them than there are simple modules, as Proposition \ref{h-bijection} shows.  We can produce a smaller projective generators by using string parametrizations.

\begin{defn}
We call a sequence $(\Bi,\kappa)$ {\bf stringy} if the sequence of $i$'s between the $j$th and $j+1$st red lines is the string parametrization of a crystal basis vector in $V_{\la_j}$.

We will implicitly use the canonical identification between stringy sequences and $\mathcal{B}^{\bla}$ via $h$.  
\end{defn}

As in Khovanov and Lauda \cite[\S 3.2]{KLI}, we order the elements of
the crystal by first decreasing weight (so that the smallest element
is the highest weight vector) and then lexicographically by the string
parametrization.

For the tensor product crystal, we use the dominance order on
$\bal$'s, with the order on nodes in the factors used lexicographically to break ties.

\begin{prop}\label{prop:stringy}
The projective cover of any simple appears as a summand of $P^{\kappa}_{\Bi}$ where   $(\Bi,\kappa)$ is the corresponding stringy sequence.  This cover is, in fact, the unique indecomposable summand which doesn't appear in $P^{\kappa'}_{\Bi'}$ for $(\Bi',\kappa')>(\Bi,\kappa)$.  
\end{prop}
As a matter of convention, we call the root function of the stringy sequence where an indecomposable projective first appears the root function of that projective.
\begin{proof}
Obviously, $P^{\kappa}_{\Bi}\twoheadrightarrow S^{\kappa}_{\Bi}=\mathbb{S}^{\bla}(\fF_{i_{\kappa(2)}}^{(a_{\kappa(2)})}\cdots \fF_{i_1}^{(a_1)}P_\emptyset\boxtimes\cdots\boxtimes
\fF_{i_{n}}^{(a_n)}\cdots \fF_{i_{\kappa(\ell)+1}}^{(a_{\kappa(\ell)+1})}P_\emptyset)$ which in turn surjects to the corresponding simple, by the definition of Kashiwara operators on simple modules, and the map $h$.  Thus, the indecomposable projective cover is a summand of $P^{\kappa}_{\Bi}$.

For a simple $L$, there is only a map of $P^\kappa_\Bi$ to $L$ if $Le_{\Bi,\kappa}\neq 0$, which is impossible for $(\Bi,\kappa)$ stringy unless $L$ is the image under $h$ of simples that appear in $\fF_{i_{\kappa(j)}}^{(a_{\kappa(j)})}\cdots \fF_{i_{\kappa(j-1)+1}}^{(a_{\kappa(j-1)})}P_\emptyset$, or $L$ is higher in the dominance order. Since only larger simples in Khovanov and Lauda's order appear in  $\fF_{i_{\kappa(j)}}^{(a_{\kappa(j)})}\cdots \fF_{i_{\kappa(j-1)+1}}^{(a_{\kappa(j-1)})}P_\emptyset$ by \cite[Lemma 3.7]{KLI}, the projective cover of any simple which appears other than that for our chosen stringy sequence is a summand in a projective for a higher stringy sequence.
\end{proof}

For an indecomposable projective $P$, its {\bf standard quotient} is its quotient under the sum of all images of maps from projectives with higher root sequences.  This coincides with its image in $S^{\kappa}_\Bi$, the standard quotient for its associated stringy sequence.  This standard quotient is indecomposable, since it is a quotient of an indecomposable projective.

\begin{prop}
Consider $({\Bi},{\kappa})$ with the associated root function $\bal$.  Then the sum of indecomposable summands of $P_{\Bi}^{\kappa}$ that have the same root function surject to $S^\kappa_\Bi$, which is a direct sum of the standard quotients of those projectives.
\end{prop}
\begin{proof}
If an indecomposable summand of $P_{\Bi}^{\kappa}$ has a different root function, it must be higher, so this summand is in the image of a higher stringy projective and thus in $L^\kappa_\Bi$.  Thus, the other summands must surject.

Similarly, it is clear that the intersection of any indecomposable with the same root function with $L^\kappa_\Bi$ is exactly the trace of the projectives with higher root functions.
\end{proof}

Finally, we prove a result which, while somewhat
technical in nature, is very important for understanding how to
decategorify our construction.  As in \cite[\S 2.12]{BGS96}, we let
$C^{\uparrow}(\alg^\bla)$ denote the category of complexes of graded
modules satisfying $C^i_j=0$ for $i\gg 0$ or $i+j\ll 0$.

\begin{thm}
  Every simple module over $\alg^\bla$ has a projective resolution in
  $C^{\uparrow}(\alg^\bla)$.
  In particular, each simple module $L$ has a well-defined class in
  $K_0(\alg^\bla)\otimes_{\Z[q,q^{-1}]}\Z((q))\cong V_\bla$.
\end{thm}
This observation would be clear if $T^\bla$ were Morita equivalent to
a positively graded algebra.  This case is called {\bf mixed} by Achar
and Stroppel \cite{AchS}, and is carefully worked out in their paper.
As shown in \cite[4.6]{WebwKLR}, this is true when $\K$ is
characteristic 0, the Cartan matrix of $\fg$ is symmetric, and
polynomials $Q_{ij}$ are carefully chosen, but as the example
\cite[5.6]{WebCB} shows, it can fail outside these cases.  Thus, we
instead prove this weaker result.
\begin{proof}
  The proof is by induction on our order above.  First, we
  do the base case of $\bla=(\la)$ and $\la-\al=k\al_i$.  This case,
  $\alg^\bla_\al$ is Morita equivalent to its center, which is the
  cohomology ring on a Grassmannian of $k$-planes in
  $\la^i$-dimensional space.  In particular, it is positively graded,
  so such a resolution exists.

  Now, we bootstrap to the case where $\bla=(\la)$ but $\al$ is
  arbitrary.  In this case, we may assume that $L'=\te_{i_1}^{a_1}L$
  has this type of resolution.  Now, we
  consider $$M=\mathrm{Ind}_{\al+a_1\al_{i_1},a_i\al_{i_1}}L'\boxtimes
  L(i_1^{a_1}),$$ where here we use the notation of \cite[\S
  3.2]{KLI}.  The module $M$ has a projective resolution of the
  prescribed type, by inducing the outer tensors of the resolutions on
  the two factors.  Furthermore, there is a surjection
  $M\twoheadrightarrow L$ whose kernel has composition factors smaller
  in the order given above on simples, by \cite[Theorem
  3.7]{KLI}. Since each of these has an appropriate resolution by
  induction, we may lift the inclusion of each composition factor to a
  map of projective resolutions, and take the cone to obtain a
  resolution of $L$ in $C^{\uparrow}(\alg^\bla)$.

  Finally, we deal with the general case using standardization; let
  $L=h(\{L_i\})$.  By standardizing the resolutions of $L_i$, we
  obtain a standard resolution of
  $\mathbb{S}^{\bla}(L_1\boxtimes\cdots\boxtimes L_\ell)$.  Replacing
  each standard with its finite projective resolution, we obtain a
  projective resolution of the same module.  As before, the kernel of
  the surjection of this module to $L$ has composition factors all
  smaller in the partial order, so we may attach projective
  resolutions of each composition factor to obtain a projective
  resolution of $L$ in $C^{\uparrow}(\alg^\bla)$.
\end{proof}

\subsection{Standard stratification}

Now, we proceed to showing that the algebra $\alg^\bla$ is standardly stratified.  
Consider the set $\Phi$ of permutations of the bottom ends of the
strands which only move black strands into blocks to their left and
are minimal coset representatives for the permutations of the strands
at the top of the diagram.  We first give these a partial order which
only depends only on the resulting idempotent at the top of the
diagram.

So, we first preorder $\Phi$ according to this preorder on the
idempotent $(\Bi_\phi,\kappa_\phi)$ which appears at the top of the
diagram.  Then within the permutations giving a single idempotent, we
use the Bruhat order.  Unlike the preorder above, this is a partial order.

\begin{figure}[ht]
  \centering
  \begin{tikzpicture}[very thick]
      \draw[ultra thick, loosely dotted, -] (0,-.5) to (-1,-.5);
      \draw[ultra thick, loosely dotted, -] (5.5,-.5) to (6.5,-.5);
      \draw (1.9,-2) to[out=90,in=-90] (.5,-.6) to[out=90,in=-90] (.5,1);
      \draw (4,-2) to[out=90,in=-90] (.8,-.3) to[out=90,in=-90] (.8,1);
      \draw (5,-2) to[out=90,in=-90] (1.1,.2) to[out=90,in=-90] (1.1,1);
      \draw[wei] (1.5,-2) -- (1.5,1);
      \draw (2.2,-2) -- (2.2,1);
      \draw (4.3,-2) to[out=90,in=-90] (2.5,.4) to[out=90,in=-90] (2.5,1);
      \draw[wei] (2.9,-2) -- (2.9,1);
      \draw (3.3,-2) -- (3.3,1);
      \draw[wei] (4.7,-2) -- (4.7,1);
      \draw (5.3,-2) to[out=90,in=-90] (4,.4) to[out=90,in=-90] (4,1);
      \draw (5.6,-2) to[out=90,in=-90] (4.3,.5) to[out=90,in=-90] (4.3,1);
    \end{tikzpicture}
\caption{The element $x_\phi$}
\end{figure}
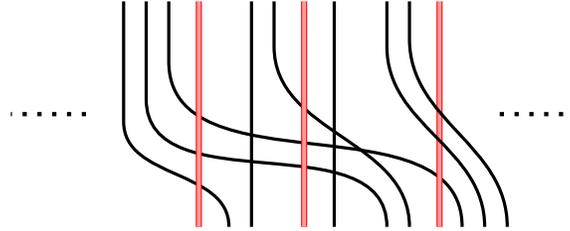

Let $x_\phi$ be an element where we permute the strands exactly
according to a chosen reduced word of $\phi\in \Phi$.  Let $$P_{\leq
  \phi}=\langle x_{\phi'} \vert \phi'\leq \phi\rangle \subset
P_{\Bi}^\kappa \qquad P_{< \phi}=\langle x_{\phi'} \vert \phi'<
\phi\rangle \subset P_{\Bi}^\kappa$$

The element $x_\phi$ is not unique, since it depends on a choice
of reduced word; however, any two choices differ by an element of $L_{<\phi}$,
so the filtration described above is unique.

\begin{prop}\label{standard-filtration}
$P_{\leq \phi}/P_{<\phi}\cong S^{\kappa_{\phi}}_{\Bi_\phi}$.
\end{prop}
We note that some of these subquotients are trivial, but in this case
the corresponding standard module is trivial as well.
\begin{proof}
Since this map is surjective, we have $\dim P_{\leq\phi}/P_{<\phi}\leq \dim S^{\kappa_{\phi}}_{\Bi_\phi}$.  On the other hand, we have $v^\kappa_\Bi=\sum_{\phi\in\Phi} q^{-\deg x_\phi} s^{\kappa_\phi}_{\Bi_\phi}$, so taking inner product with $[\alg^\bla]$, we obtain $\dim P^\kappa_\Bi=\sum_{\phi\in\Phi}\dim S^{\kappa_{\phi}}_{\Bi_\phi}$.  

Thus we must have equality above, and the map is an isomorphism for degree reasons.
\end{proof}
\begin{cor}
The algebra $\alg^\bla$ is standardly stratified with standard modules given by the standard quotients of indecomposable projectives, and the preorder on simples/standards/projectives given by the dominance order on root functions $\bal$.
\end{cor}
\begin{cor}
  Every standard module has a finite length projective resolution.
\end{cor}
This is a standard fact about finite dimensional standardly stratified algebras; in particular, any module with a standard filtration has a well-defined class in $K_0(\alg^\bla)$. 
\begin{proof}
  First note that if a module $M$ is filtered by modules which have
  finite length projective resolutions, these resolutions can be glued
  to give a finite length resolution of the entire module.

  Now, we induct on the partial order $\leq$.  If a standard is
  maximal in this order, it is projective. For an arbitrary standard,
  there is a map $P^{\Bi}_{\kappa}\to   S^{\Bi}_{\kappa}$ with kernel
  filtered by standards higher in the partial order.  Since each of
  these has a finite length projective resolution, $S^\kappa_\Bi$ does
  as well.
\end{proof}

We note that $e(\Bi,\kappa)\alg^\bla e(\Bi,0)$ has a unique element
consisting of a diagram with no dots and no crossings between black
strands which simply pulls red strands to the left and black to the
right. As before, we call this element $\theta_\kappa$ (leaving
$\Bi$ implicit).

\begin{lemma}\label{self-dual-embedding}
The map from $P^\kappa_{\Bi}\to P^{0}_{\Bi}$ given by the action of  $\theta_\kappa$ is injective.
\end{lemma}
\begin{proof}
Obviously, this map is filtered, where we include $\Phi_{\Bi,\kappa}\subset \Phi_{\Bi,\kappa}$ by precomposing with the permutation that pushes all black strands to the right.  Furthermore, it induces an isomorphism on each successive quotient in this image.  Thus, it is injective.
\end{proof}

We let $\cata^{\la_1;\dots;\la_n}_{\bal}=  \alg_{\al(1)}^{\la_1}\otimes\cdots \otimes \alg_{\al(\ell)}^{\la_\ell}\modu$, and let $\cC^{\bal}$ be the subcategory of modules which have a presentation by standard modules with root function $\bal$.

\begin{prop}\label{semi-orthogonal}
We have a natural isomorphism $$\End_{\alg^\bla}(S_\bal)\cong \alg_{\al(1)}^{\la_1}\otimes\cdots \otimes \alg_{\al(\ell)}^{\la_\ell}.$$
In particular, $\cC^\bal$ is equivalent to
    $\cata^{\la_1;\dots;\la_n}_{\bal}$.  
    The triangulated subcategories generated by $\cC^{\bal}$ form a semi-orthogonal
    decomposition of the derived category $D^+(\cata^\bla_{\al})$ with respect to
    dominance order.
  \end{prop}

\begin{proof}
  Since every standard with root function $\bal$ is a summand of
  $S_\bal$ and  $S_\bal$ has trivial higher Exts, it follows
  immediately that
  \begin{equation*}
    \cC^\bal\cong\End^{op}(S_\bal)\modu.
  \end{equation*}

  Let us calculate this endomorphism algebra.  By the projective
  property, every endomorphism of $S_\bal$ is induced by an
  endomorphism of $e_\bal \alg^\bla$.  Thus $\End^{op}(S_\bal)$ is a
  subquotient of $e_\bal \alg^\bla e_\bal$.  We take the
  quotient of the subalgebra which preserves
  the kernel of  $S_\bal$ by the ideal of those that send
  everything to that kernel.

  Apply Proposition \ref{straighten} in the case where each reduced word puts each group of black strands and red immediately to its left in the correct order first, followed by a shortest coset representative for this Young subgroups.  This implies that the diagram from any permutation which has a left crossing has at least one before any right crossings.  By the definition of the standard quotient such a
  diagram acts trivially.  On the other hand,
  an element of $e_\bal \alg^\bla e_\bal$ must have equal numbers of the
  two types of crossings, so our element can be ``straightened'' so
  that no red and black strands ever cross.  Thus, we have a
  surjective map from $\tilde{\alg}_\al^{\la_1}\otimes\cdots \otimes
  \tilde{\alg}_\al^{\la_\ell}$ to $\End(S_\bal)$.

By definition of a standard quotient, the cyclotomic ideal of this tensor product is killed by
  the map to $\End^{op}(S_\bal)$, so we have a surjective map
  $\alg^{\la_1}_{\bal(1)}\otimes \cdots \otimes \alg^{\la_\ell}_{\bal(\ell)}\to \End^{op}(S_\bal)$,
  which we need only show is also injective.  Since $\Ext^{>0}(S_\bal,S_\bal)=0$, this is equivalent to showing that \[\dim \End(S_\bal,S_\bal)=\langle[S_\bal],[S_\bal]\rangle_1=\dim\alg^{\la_1}_{\bal(1)}\otimes \cdots \otimes \alg^{\la_\ell}_{\bal(\ell)}.\] The second equality follows from the equality $\langle a\otimes b,a'\otimes b'\rangle=\langle a,b\rangle \langle a',b'\rangle$ if $a,a'$ and $b,b'$ are weight vectors with each pair having the same weight, which follows, in turn, from the upper-triangularity of $\Theta^{(2)}$.

Finally, we establish the semi-orthogonal decomposition: by Proposition~\ref{standard-filtration}, the subcategory
  generated by $\cC^{\bal'}$ for $\bal'> \bal$ in the dominance order
  is the same as that generated by $P^{\kappa}_{\Bi}$ such that
  $\bal_{\Bi,\kappa}> \bal$. Since all the simple modules in
  $S_\Bi^\kappa$ are given by idempotents $e_{\Bi,\kappa}$ such that
  $\bal_{\Bi,\kappa}\leq \bal$, we have
  \begin{equation*}
    \Ext^\bullet(S_{\Bi'}^{\kappa'},S_{\Bi}^\kappa)=0
  \end{equation*}
  whenever $\bal_{\Bi,\kappa}< \bal_{\Bi',\kappa'}$, and higher
  $\Ext$'s vanish when equality holds.
\end{proof}

\subsection{Self-dual projectives}
\label{sec:self-dual}

One interesting consequence of the module structure over $\tU$ and standard stratification is the
understanding it gives us of the self-dual projectives of our
category.  Self-dual projectives have played a very important role in
understanding the structure of representation theoretic categories
like $\cata^\bla$. For example, the unique self-dual projective in BGG
category $\cO$ for $\fg$ was key in Soergel's
analysis of that category \cite{Soe90,Soe92}, and the self-dual
projectives in category $\cO$ for a rational Cherednik algebra provide
an important perspective on the Knizhnik-Zamolodchikov functor defined
by Ginzburg, Guay, Opdam and Rouquier \cite{GGOR}.  In particular, as
Mazorchuk and Stroppel show \cite{MS}, these modules also play an
important role in the identification of the Serre functor; we will
apply their results to describe the Serre functor of the perfect derived category of $\alg^\bla$-modules in the sequel  \cite{KI-HRT} to this paper.

\begin{thm}\label{self-dual}
If $P$ is an indecomposable projective $\alg^\bla$-module, then the following are equivalent:
\begin{enumerate}
\item $P$ is injective.
\item $P$ is a summand of the injective hull of an indecomposable standard module.
\item $P$ is isomorphic (up to grading shift) to a summand of $P^0$.
\end{enumerate}
\end{thm}
\begin{proof}
$(3)\rightarrow(1)$:
To establish this, we show that $P^0$ is self-dual; that is, there is a non-degenerate pairing
$P^0_\Bi\otimes P^0_\Bi\to \K$.  This is given by $(a,b)=\tau_\la(a\dot
b)$, where $\tau_\la$ is the Frobenius trace on $\End(P^0)\cong
\alg^{\la}$ given in Section \ref{sec:cyc}.  Thus $P^0$ is both
projective and injective, so any summand of it is as well.

$(1)\rightarrow(2)$: Since $P$ is indecomposable and injective, it is the injective hull of any submodule of $P$.  Since $P$ has a standard stratification, it has a submodule which is standard.

$(2)\rightarrow(3)$: We have already established that $P^0$ is injective, so we need only 
establish that any simple in the socle of $S^\kappa_\Bi$ is a summand
of the cosocle of $P^0$ (since the injective hull of $S^\kappa_\Bi$
coincides with that of its socle).  It suffices to show that there is no non-trivial submodule of $S^\kappa_\Bi$ killed by $e_{0,*}$.  If such a submodule $M$ existed, then we would have $M\dot{\theta_\kappa}=0$. Thus, its preimage $M'$ in $P^\kappa_\Bi$ satisfies $M' \dot{\theta_\kappa}\subset L^0_\Bi$.  But the injectivity of Lemma \ref{self-dual-embedding} and the fact that $L^\kappa_\Bi\dot{\theta_\kappa}=L^0_\Bi\cap P^\kappa_\Bi \dot{\theta_\kappa}$, this implies that $M=0$.
\end{proof}

For two rings $A$ and $B$, we say an $A$-$B$ bimodule $M$ has the {\bf
  double centralizer property} if $\End_B(M)=A$ and $\End_A(M)=B$.  In
particular, this implies that if $M$ is projective as a $B$-module,
the functor $$\Hom( M,-):B\modu\to A\modu$$ is fully faithful on
projectives (it could be quite far from being a Morita equivalence, as
the theorem below shows).

\begin{cor}\label{doub-cen}
    The projective-injective $P^0$ has the double centralizer property
  for the actions of $\alg^{\la}$ and $\alg^\bla$ on the left and right.
\end{cor}
\begin{proof}
By \cite[Corollary 2.6]{MS}, this follows immediately from the fact that the injective hull of an indecomposable standard is also a summand of $P^0$.
\end{proof}
Thus, in this case, our algebra can be realized as the endomorphisms
of a collection of modules over $R^\la$, in a way analogous to the
realization of a regular block of category $\cO$ as the modules over
endomorphisms of a particular module over the coinvariant algebra, or
of the cyclotomic $q$-Schur algebra as the endomorphisms of a module
over the Hecke algebra.

In fact, these modules are easy to identify.  Given $(\Bi,\kappa)$, we
consider the element $y_{\Bi,\kappa}$ of $P^0_{\Bi}$ given
by $$y_{\Bi,\kappa}=e_{\Bi}\prod_{j=1}^{\ell}\prod_{k=\kappa(j)+1}^ny_k^{\la_j^{i_k}}.$$
Pictorially this is given by multiplying the element with no
black/black crossings going from $(\Bi,0)$ to $(\Bi,\kappa)$ (which we
denote $\vartheta_\kappa$) by its horizontal reflection
$\dot\vartheta_\kappa$, and then straightening the strands.

\begin{figure}[ht]
  \centering
  \begin{tikzpicture}[very thick, scale=1.3]
    
\draw[wei] (1.5,-1) to node[below, at start]{$\la_1$} (1.5,1);
\draw[wei] (2,-1) to[out=50, in=-50] node[below, at start]{$\la_2$} (2,1) ;
\draw[wei] (2.5,-1) to[out=60, in=-60] node[below, at start]{$\la_3$} (2.5,1)  ;
\draw[wei] (3,-1) to[out=10, in=-10] node[below, at start]{$\la_4$}  (3,1) ;
\draw (3.5,-1) to[out=150, in=-90]  node[below, at start]{1}(2,0) to[out=90,  in=-150] (3.5,1) ;
\draw (4,-1)  to[out=150,in=-90] node[below, at start]{5} (3,0) to[out=90,  in=-150] (4,1);
\draw (4.5,-1)  to[out=140,in=-90] node[below, at start]{4} (3.3,0) to[out=90,   in=-140] (4.5,1);
\draw (5,-1) tonode[below, at start]{2}  (5,1) ;
\draw[loosely dashed,thin] (1,0) -- (5.5,0);
\node at (6,.5) {$\dot\theta_{0,1,1,3}$};
\node at (6,-.5) {$\theta_{0,1,1,3}$};
  \end{tikzpicture}

  \caption{The element $y_{(1,5,4,2),(0,1,1,3)}$.}
  \label{fig:yik}
\end{figure}
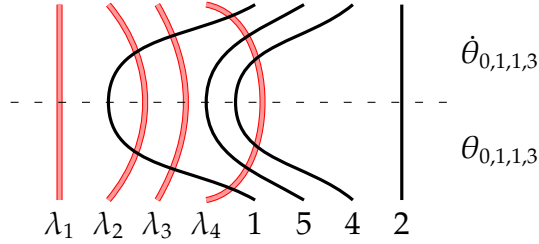

\begin{prop}\label{prop:add-embed}
The algebra $\alg^\bla$ is isomorphic to the algebra $\End_{\alg^\la}(\bigoplus_{\kappa}y_{\Bi,\kappa} \alg^\la  )$.
\end{prop}
\begin{proof}
  Based on Corollary \ref{doub-cen}, all we need to show is that
  $\Hom_{\alg^\bla}(P^0,P^{\kappa}_{\Bi})\cong y_{\Bi,\kappa} P^0_{\Bi} $
  as a $\alg^\la$ representation.  A map $m$ from $P^0_{\Bi'}$ to
  $P^\kappa_{\Bi}$ is simply a linear combination of diagrams starting
  at $\Bi$ with the correct placement of red strands and ending at
  $\Bi'$ with all red strands to the right.  By Proposition \ref{straighten}, we
  can assure that all red/black crossings occur above all black/black
  ones, so $m=\vartheta_\kappa m'$, where $m\in \alg^\la$.
 
  Thus, we have
  maps $$\Hom_{\alg^\bla}(P^0,P^{0}_{\Bi})\overset{\vartheta_\kappa}\longrightarrow
  \Hom_{\alg^\bla}(P^0,P^{\kappa}_{\Bi})\overset{\dot\vartheta_\kappa}\longrightarrow
  \Hom_{\alg^\bla}(P^0,P^{0}_{\Bi})$$ given by composition. The first of
  these is surjective, as we argued above. Furthermore, the latter is injective,
  by Proposition \ref{self-dual-embedding}.
  Thus, $\Hom_{\alg^\bla}(P^0,P^{\kappa}_{\Bi})$ is isomorphic to the
  image of the composition of these maps, which is $y_{\Bi,\kappa}
  \alg^\la$.
\end{proof}
For some choices of $\Bi$ and $\kappa$, the element $y_{\Bi,\kappa}$
has already appeared in work of Hu and Mathas \cite{HM}.  Assume that
$\fg=\mathfrak{sl}_n$ and specialize to the case where $\la_j=\omega_{\pi_j}$ for some
$\pi_j$.  As suggested by the notation, we will later want to think of
$\pi_j$ as a composition. We can define stringy sequences for this
algebra using the reduced decomposition of the longest element of the
Weyl group $w_0=s_{n-1}(s_{n-2}s_{n-1})\cdots(s_1\cdots s_{n-1})$.   

As illustrated in Figure \ref{fig:partitions}, the stringy sequences
for the fundamental representation $V_{\omega_i}$ are gotten by
\begin{itemize}
\item taking a partition diagram which fits in an $i\times(n-i)$ box,
\item  filling it with its content, shifted so that the box $(1,1)$ is
  filled with $i$,
\item taking the row-reading word.
\end{itemize}
\begin{figure}
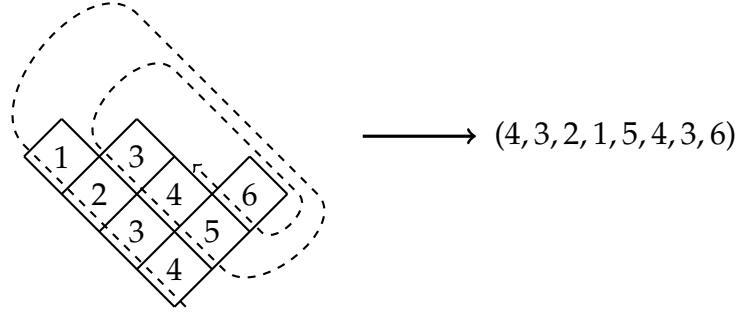

  \centering
  \tikz{
\node (a) at (-3,0) {
\tikz[thick,scale=.5]{
\draw (-4,4)--(0,0); 
\draw (-3,5) -- (1,1);
\draw (-1,5) -- (2,2);
\draw (2,4) -- (3,3);
\draw (-4,4)-- (-3,5);
\draw (-3,3) -- (-1,5);
\draw (-2,2) -- (0,4);
\draw (-1,1) -- (2,4);
\draw (0,0) -- (3,3);
\node at (0,1) {$4$};
\node at (1,2) {$5$};
\node at (2,3) {$6$};
\node at (-1,2) {$3$};
\node at (-2,3) {$2$};
\node at (-3,4) {$1$};
\node at (0,3) {$4$};
\node at (-1,4) {$3$};
\draw[dashed,->] (.3,0) -- (-4,4.3) to[out=135,in=135] (-1,7.8) --
(3.8,3) to[out=-45, in=-45] (1.3,1) -- (-2,4.25) to[out=135,in=135]
(0,6.3) -- (3.3,3) to[out=-45, in=-45] (2.3,2) -- (.5,3.8);
}
};
\node[outer sep=3pt] (b) at (3,0){$(4,3,2,1,5,4,3,6)$};
\draw[very thick, ->] (a)-- (b);
}
  \caption{The stringy sequence attached to a partition for $n=7$ and $i=4$.}
  \label{fig:partitions}
\end{figure}
For a multipartition $\xi=(\xi^{(1)},\dots, \xi^{(\ell)})$, with
$\xi^{(i)}$ fitting in a $\pi_i\times (n-\pi_i)$ box, we can thus
define $(\Bi_\xi,\kappa_\xi)$ where $\Bi_\xi$ is the concatenation of
these row-reading words, and $\kappa_\xi(k)$ is the number of
the boxes in the first $k-1$ partitions.  The element
$y_{\Bi_\xi,\kappa_\xi}$ is exactly that denoted
$\psi_{\mathfrak{t}^\xi \mathfrak{t}^\xi}$ in \cite{HM,HMQ}.

Mathas and Hu have defined another algebra,
which they call a {\bf quiver
    Schur algebra}\footnote{This is an unfortunate terminological clash
    with \cite{SWschur}, where a non-equivalent, but graded Morita
    equivalent algebra is given the same name; after forgetting the
    grading, this is the difference between defining Schur algebras
    using all permutation modules attached to partitions or to
    compositions.} $\mathcal{S}^\la_m$.
\begin{thm}\label{quiver-schur}
  For $\fg=\mathfrak{sl}_n$, the category $\cata^\bla$ is equivalent
  (as a graded category)
  to a sum of blocks of graded representations of $\mathcal{S}^\la_m$ for the charges $(\pi_1,\dots,\pi_\ell)$.
\end{thm}
If we considered the case where $\fg=\mathfrak{sl}_\infty$ (thought of
as the Kac-Moody algebra of the $A_\infty$-quiver), then we could say
that $\cata^\bla$ is simply equivalent to $\oplus_m\mathcal{S}^\la_m\modu$.
\begin{proof}
  By \cite[4.35]{HMQ},  the graded category of projectives in a block of Mathas and Hu's algebra
  is equivalent to an additive subcategory of $\alg^\la\modu$.  By
  Proposition \ref{prop:add-embed}, the graded category of projectives in each weight space of
  $\cata^\bla$ is also equivalent to such a subcategory.  Thus, we
  need only show that we can match these subcategories coincide.  

Each block of $\mathcal{S}^\la_m$ is the sum of images of the
idempotents $e(\Bi)$ where $\Bi$ ranges over all integer sequences
with a fixed number $m_i$ of occurrences of $i$.  As long as only
numbers between $1$ and $n-1$ appear, we can associate to this
multiplicity data a weight $\mu=\la-\sum_im_i\al_i$.  We wish to
show that this block is equivalent to $\cata^\bla_\mu$.  Let $m=\sum m_i$.

The image of projective modules over $\mathcal{S}^\la_m$ is the subcategory
additively generated by $\psi_{\mathfrak{t}^\xi
  \mathfrak{t}^\xi}\alg^\la=y_{\Bi_\xi,\kappa_\xi}\alg^\la$ as we
range over all multipartitions with $m$ boxes fitting inside the correct
$\pi_i\times (n-\pi_i)$ boxes.  These are the same as the images of
the projectives $P_{\Bi_\xi}^{\kappa_\xi}$ under the functor
$\Hom(P^0,-)$.  By Proposition \ref{prop:stringy}, every
indecomposable projective over $\alg^\bla_\mu$ is a summand of a
unique one of these modules, so those which have weight $\mu$ already
additively generate the image of the $\alg^\bla_\mu\mpmod$ in
$\alg^\la_\mu\modu$.  Thus, that image coincides with the corresponding
image for the quiver Schur algebra.
\end{proof}

\section{Comparison to category \texorpdfstring{$\cO$}{O}}
\label{sec:type-A}

\subsection{Cyclotomic degenerate Hecke algebras}

Now, we specialize to the case where $\fg\cong \mathfrak{sl}_n$.  In
this case, we can reinterpret our results in terms of the work of
Brundan and Kleshchev \cite{BKSch,BKKL} who have shown that in this
case, the cyclotomic Khovanov-Lauda algebra is a cyclotomic degenerate
affine Hecke algebra (cdAHA).

Recall that the degenerate affine Hecke algebra (dAHA) is the algebra
with generators $x_1,\dots,x_d$ and $w\in S_d$ such that
$$s_ix_j=x_{s_i\cdot j}s_i-\delta_{j,i}+\delta_{j,i+1}\qquad\qquad x_ix_j=x_jx_i$$ 
for the simple reflections in $s_i\in S_d$ and the usual relations
between permutations.

We have a natural action of $H_d$ on the $\mathfrak{gl}_N$ module
$P\otimes V^{\otimes d}$ for any $\fgl_n$ representation $P$, where
$V=\C^N$ is the defining representation of $\mathfrak{gl}_N$: 
\begin{itemize}
\item  $S_d$ acts on the $d$ copies of $V$, and
\item $x_1$ acts by $C\otimes 1^{\otimes d-1}$ where $C$ is the Casimir element of $\mathfrak{gl}_N$.  
\end{itemize}
 We will
be most interested in the case where $P$ is a certain parabolic Verma
module for a parabolic $\fp$; in this case, by the definition of
induction, $$P\otimes V^{\otimes d}\cong
U(\fgl_n)\otimes_{U(\fp)}(W\otimes V^{\otimes d})$$ for a finite
dimensional representation $W$ of $\fp$.  
\begin{defn}
Parabolic category $\cO$, which we denote $\cO^\fp$, is  the full
subcategory of $\mathfrak{gl}_N$-modules with a weight decomposition
where $\fp$ acts locally finitely.
\end{defn}
Since induction sends finite-dimensional modules to $\fp$-locally finite modules, $P\otimes V^{\otimes d}$ lies in this category.

Attached to each parabolic $\fp\subset\fgl_N$, we have a unique composition
$\pi=(\pi_1,\dots, \pi_\ell)$ such that $\fp$ is conjugate to block-diagonal matrices for this composition (the composition $\pi$ can be recovered as the gaps in
the finest flag $\fp$ preserves).  These can be used to define a weight 
$\la=\sum_i\omega_{\pi_i}\in\wela(\fg)$; that is, $\la^j=\#\{i|\pi_i=j\}$.

\begin{defn}
The {\bf cyclotomic
degenerate affine Hecke algebra} is the quotient of the dAHA given by $$H^\la=\bigoplus_{d\geq 0}
H_d/\Big\langle\prod_{i=1}^n(x_1-i)^{\la^i}\Big\rangle.$$
\end{defn}
 This has a
natural system of orthogonal idempotents $e_d$ for all $d\geq 0$ which
project to the image of $H_d$.  Brundan and Kleshchev show that when
$P$ is the parabolic Verma module associated to the ``ground state''
tableau on $\pi$, then the action of dAHA on $P\otimes V^{\otimes d}$ factors through its cyclotomic quotient.

Thus, we have a functor $\Hom_{\mathfrak{gl}_N}(P\otimes
V^{\otimes d},-):\cO^\fp\to H^\la\modu$.  This functor is very far from being an
equivalence, but on each block of $\cO^\fp$ it is either 0, or fully
faithful on projectives.  Thus, certain blocks of $\cO^\fp$  can be described in
terms of endomorphism rings of modules over $H^\la$.

In \cite{BKKL}, Brundan and Kleshchev show that each category $\cata^\la_\mu$ is
equivalent to a block of the representations of $H^\la$. Thus, using this isomorphism, we can
also express $\cata^\bla_\mu$ in terms of endomorphisms of modules over
$H^\la$.

There is an idempotent of $H_d$ associated to any length $d$ sequence
of integers.  We let $e_\fg$ be the sum of these idempotents
corresponding to sequences of integers in $[1,n]$.  In this section, we use the polynomials $Q_{ij}$ as defined in the previous section for a fixed orientation of the type A (or later, affine type A) quiver.  The most obvious choice is \[Q_{ij}(u,v)=\begin{cases} 1 & i\neq j\pm 1\\
u-v & i=j+1\\
v-u & i=j-1 
\end{cases}\]
\begin{prop}[\cite{BKKL}]\label{BK}
  There is an isomorphism $\Upsilon\colon \alg^\la\to e_\fg H^\la
  e_\fg\overset{\text{def}}=H^{\la,n}$.
\end{prop}
Under this map, we have that
 $\Upsilon(y_je(\Bi))=  e(\Bi)(x_j-i_j)$, and $\Upsilon^{-1}(T_i)$ is
 in a linear combination of $y_{i}^ay_{i+1}^b\psi_ie(\Bi)$ and
 $y_{i}^ay_{i+1}^be(\Bi)$ by \cite[(3.41-42)]{BKKL}.

\subsection{Comparison of categories}
\label{sec:comp-categ}

First, let us endeavor to understand how we can translate the $\alg^\la$-modules
$y_{\Bi,\kappa}\alg^\la$ defined in Section \ref{sec:QHA} into the language of the cdAHA using $\Upsilon$.
It's immediate from Proposition \ref{BK} that
\begin{equation*}
\Upsilon(y_{\Bi,\kappa})=  e(\Bi)\prod_{j=1}^{\ell}\prod_{k=\kappa(j)+1}^n(x_k-i_k)^{\la_j^{i_k}}.
\end{equation*}

However, the strong dependence of this element on $e(\Bi)$ makes it
problematic for use in the Hecke algebra.


We first specialize to the case where $\la_j=\omega_{\pi_j}$ for some
$\pi_j$.  As suggested by the notation, we will later want to think of
$\pi_j$ as a composition.  This bit of notation allows us to associate
to each $\kappa$ an element of $H^{\la,n}$ (note that there is no
dependence on $\Bi$):
\begin{equation}
z_\kappa=\prod_{j=1}^\ell \prod_{k=1}^{\kappa(j)}(x_k-\pi_j)
\end{equation}

We let $M^\kappa_\Bi=e(\Bi) z_{\kappa}H^{\la,n}$ and $M^\kappa=z_{\kappa}H^{\la,n}$.

\begin{prop}
  For all $\Bi$, we have $y_{\Bi,\kappa} H^{\la,n}= M^\kappa_\Bi$. In
  particular, we have an isomorphism $\alg^\bla\cong \End(\oplus_{\kappa}
  M^\kappa)$.
\end{prop}
\begin{proof}
  If $a\neq i_j$, then we can rewrite $e(\Bi)$ as 
  $$e(\Bi)=(x_j-a)e(\Bi)\bigg(\frac{-1}{a-i_j}-\frac{x_j-i_j}{(a-i_j)^{2}}-\frac{(x_j-i_j)^2}{(a-i_j)^{3}}-\cdots\bigg)$$
  since $(x_j-i_j)e(\Bi)$ is nilpotent.  It follows that 
\begin{equation}\label{Hec-nil}
    e(\Bi)(x_k-\pi_j)H^{\la,n}=e(\Bi)(x_k-i_k)^{\la_j^{i_k}}H^{\la,n}
\end{equation}
since $\la_j^{i_k}=\delta_{\pi_j,i_k}$.  Thus, applying
(\ref{Hec-nil}) to each term in $z_\kappa$, the result follows.
\end{proof}

We note that the modules $M^\kappa$ are closely related to the
permutation modules discussed by Brundan and Kleshchev in \cite[\S
6]{BKSch}.  Each way of filling $\pi$ as a tableau such that the
column sums are $\kappa(i)-\kappa(i-1)$ results in a permutation
module which is a summand of $M^\kappa$.

Now we wish to understand how the modules $M^\kappa$ are related to
parabolic category $\cO$.  Let $N=\sum_{j}\pi_j$ be the number of
boxes in $\pi$. As before, the $\pi_i$ give a composition of $N$,
and thus a parabolic subgroup $\fp\subset \fgl_N$, which is precisely the operators preserving a flag
$V_1\subset V_2\subset \cdots \subset V$.  If, as usual, $\kappa$ is a weakly increasing function on
$[1,\ell]$ with non-negative integer values and further $\kappa(\ell)\leq d$,
then we let $$V_\kappa^d=V_1^{\otimes \kappa(1)}\otimes V_2^{\otimes
  \kappa(2)-\kappa(1)}\otimes \cdots \otimes V^{d-\kappa(\ell)}$$ as a
$\fp$-representation.  We can induce this representation to an object
in $\cO^\fp$ which we denote $$P^\kappa_d\cong
U(\fgl_n)\otimes_{U(\fp)}(\C_{-\rho}\otimes V^d_\kappa ),$$ where $\C_{-\rho}$ is the 1-dimensional $\fp$-module defined in \cite[pg. 4]{BKSch}.

All the objects $P^\kappa_d$ live in the subcategory we denote
$\cO^\fp_{> 0}$ which is generated by all parabolic Verma modules whose
corresponding tableau has positive integer entries.  We also consider
a much smaller subcategory which has only finitely many simple
objects: let $\cO^\fp_{n}$ be the subcategory of $\cO^\fp$ generated
by all parabolic Vermas whose corresponding tableau only uses the
integers $[1,n]$.  Let $\pr_n:\cO^\fp\to\cO^\fp_n$ be the projection
to this subcategory ($\cO^\fp_n$ is a sum of blocks, so there is a unique
projection).

\begin{prop}
If one ranges over all $\kappa$ and all integers $d$, then $\displaystyle\oplus_{\kappa,d}V^d_\kappa$ is a projective generator for $\cO^\fp_{>0}$.
\end{prop}
\begin{proof}
This follows from a simple modification of the proof of \cite[Theorem 4.14]{BKSch}.  In the notation of that proof, we have that $P^\kappa_d\cong R(P^{\kappa^-}_{\kappa(\ell)}\otimes\C_{-\rho})\otimes V^{\otimes d-\kappa(\ell)}$, where $\kappa^-$ is the restriction of $\kappa$ to $[1,\ell-1]$. As noted in that proof,  by induction, this is two functors which preserve projective modules applied to a projective module; thus $P^\kappa_d$ is projective.

Each of Brundan and Kleshchev's divided power modules is a summand in one of the $P^\kappa_d$, as we noted earlier.  Since any indecomposable projective of $\cO^\fp$ is a summand of a divided power module, the same is true of the $P^\kappa_d$'s.
\end{proof}

\begin{prop}\label{Hecke-equivalence}
For all $d,\kappa$, we have
\begin{align*}
 z_\kappa H^\la e_d &\cong  \Hom(P\otimes V^{\otimes d}, P_\kappa^d)\\
  M^\kappa e_d &\cong  \Hom(P\otimes V^{\otimes d}, \pr_n(P_\kappa^d)).
  \end{align*}
\end{prop}
\begin{proof}
  This rests on a single computation, which is that the image in
  $P\otimes V$ of the action of $\prod_{i=j+1}^\ell (x_1-\pi_i)$ is
  $$U(\fgl_n)\otimes_{U(\fp)}(\C_{-\rho}\otimes V_j)\subset U(\fgl_n)\otimes_{U(\fp)}(\C_{-\rho}\otimes V)\cong P\otimes V;$$ this follows
  from \cite[Lemma 3.3]{BKSch}.  This shows that the image of
  $z_\kappa$ acting on $P\otimes V^{\otimes d}$ is $P_\kappa^d$, so by
  the projectivity of $P\otimes V^{\otimes d}$, every homomorphism to
  $P_\kappa^d$ factors through this one.

  We can identify those homomorphisms whose image is in
  $\pr_n(P_\kappa^d)\subset P_\kappa^d$ as those killed by some power
  of $\chi^n_j=\prod_{i=1}^n(x_j-i)$ for each $j$ (if a number $m$
  appears in a tableau, then $x_j-m$ is nilpotent for some $j$, and so
  if $m\notin [1,n]$, then $\chi^n_j$ is invertible for that $j$).
  Thus, this homomorphism space is the subspace of $z_\kappa H^\la e_d $ on
  which all $\chi^n_j$ act nilpotently, which is precisely $M^\kappa e_d$.
\end{proof}
\begin{cor}\label{equiv}
We have an equivalence $\Xi:\cata^\bla\overset{\cong}\longrightarrow\cO^\fp_n$.
\end{cor}

We can generalize this statement a bit further: let us now consider
the case where the weights $\la_i$ are not fundamental.  In this case,
to each weight $\la_i$ we have a unique Young diagram given by writing
it as a sum of fundamental weights, and we obtain a pyramid $\pi$ by
concatenating these horizontally (this is the pyramid associated
earlier to the refinement of $\bla$ into fundamental weights).  We
associate a parabolic $\fp$ with the pyramid as before.

For each collection of semi-standard\footnote{In \cite{BKSch}, these
  are called ``standard.''} tableaux $T_i$ on each of these diagrams
which only use the integers $[1,n]$, this gives a tableau on $\pi$
(now just column-strict).  Such a tableau can be converted into a module
in $\cO^\fp$ for $\fgl_N$ (where $N=\sum \pi_i$) by taking the
projective cover of the $\fp$-parabolic Verma module corresponding to
this tableau.  Let $\cO^\fp_{\bla}$ be the subcategory of modules presented by these
projectives.

\begin{prop}
The functor $\Xi$ induces an equivalence of $\cO^\fp_\bla$ and $\cata^\bla$.
\end{prop}
\begin{proof}
  What is clear from Corollary \ref{equiv} is that $\cata^\bla$ is
  equivalent to the subcategory of $\cO^\fp_\bla$ consisting of objects presented by
  projectives $\pr_n(P_\kappa^d)$ for the sequence of weights obtained by
  breaking $\bla$ into fundamental weights, where we require $\kappa$
  to be constant on the blocks of fundamental weights obtained by
  breaking up $\la_i$.  In terms of category $\cO$, we only induce
  finite-dimensional $\fp$ vector spaces obtained by tensoring the
  vector spaces which appear in a particular flag preserved by $\fp$,
  the gaps of which encode the sequence $\bla$.

  That is, the indecomposable projectives of $\cata^\bla$ are sent to
  the indecomposable projectives which appear as summands of these
  $\pr_n(P_d^\kappa)$.  Thus these are in bijection, and there can
  only be $\dim V_\bla$ of the latter.  Since there is exactly that
  number of tableaux which are semi-standard in blocks as described
  above, we need only show that these occur as summands.

  This follows from the relationship between the crystal structure on
  tableaux and projectives in category $\cO$.  Specifically, since any
  tableau which is semi-standard in blocks can be obtained from the
  empty tableau by the operations of attaching a fresh Young diagram
  filled with the ground state tableau and of applying crystal
  operators, the argument from \cite[Corollary 4.6]{BKSch} shows that
  the projective corresponding to such a tableau is a summand of an
  appropriate $P_d^\kappa$.
\end{proof}

We note that this shows that our categorification corresponds to that
for twice fun\-da\-men\-tal weights of $\mathfrak{sl}_n$ recently given by
Hill and Sussan \cite{HS}.

The category $\cO^\fp$ has a natural endofunctor given by tensoring
with $V$.  Restricting to $\cO^\fp_n$, we can take the functor
$f_\bullet=\pr_n(-\otimes V)$. This functor has a natural
decomposition $f_\bullet=\oplus_{i=1}^n f_i$ in terms of the
generalized eigenspaces of $x_1$ acting on $-\otimes V$;  we need only
take $i\in [0,n]$ since these are the only eigenvalues of $x_1$ on the
projection to $\cO^\fp_n$.

\begin{prop}\label{trans-act}
We have a commutative diagram
\begin{equation*}
    \begin{tikzpicture}[yscale=1.1,xscale=1.5,very thick]
        \node (a) at (1,1) {$\cO^\fp_n$};
        \node (b) at (-1,1) {$\cO^\fp_n$};
        \node (c) at (1,-1) {$\cata^\bla$};
        \node (d) at (-1,-1) {$\cata^\bla$};
        \draw[->] (b) -- (a) node[above,midway]{$f_i$};
        \draw[->] (d) -- (c) node[below,midway]{$\fF_i$};
        \draw[->] (c) --(a) node[right,midway]{$\Xi$};
         \draw[->] (d) --(b) node[left,midway]{$\Xi$};
    \end{tikzpicture}
\end{equation*}
\end{prop}
\begin{proof}
  The functor $f_\bullet$ corresponds to tensoring a
  $H^{\la,n}_d$-module with $H^{\la,n}_{d+1}$. By Proposition
  \ref{BK}, this corresponds to tensoring over $T^\bla_\mu$ with
  $\oplus_iT^\bla_{\mu-\al_i}$ via the map $\oplus \nu_i$.  This is,
  of course, the
  functor $\oplus_{i=1}^n\fF_i$.  Via Brundan and Kleshchev's
  isomorphism, $x_n$ acts on $\fF_iM$  for any $M$ by $y_n+i$; that is, $x_n-i$ acts
  invertibly on $\fF_jM$ for $j\neq i$ and nilpotently on $\fF_iM$.
  This shows the desired isomorphism.
\end{proof}

For any parabolic subalgebra $\fq\supset \fp$ with Levi $\fl=\fq/\!\rad\fq$, we have an induction functor \[\ind_\fl^{\fgl_N}\overset{def}= U(\fgl_N)\otimes_{U(\fq)}-:\cO^\fp(\fl)\to \cO^\fp\] where $\cO^\fp(\fl)$ denotes the parabolic category $\cO$ for $\fl$ and the parabolic $\fp/\!\rad\fq$ (here $\fl$-representations are considered as $\fq$ representations by pullback). 

Choices of $\fq$ are in bijection with partitions of $\bla$ into
consecutive blocks $\bla_1,\dots,\bla_k$. Let $\Xi_\fl:\cata^{\bla_1;\dots ;\bla_k}\to \cO^\fp(\fl)$ be the comparison functor analogous to $\Xi$ for $\fl$.

\begin{prop}\label{ind-sta}
We have a commutative diagram
\begin{equation*}
    \begin{tikzpicture}[yscale=1.1,xscale=1.9,very thick]
        \node (a) at (1,1) {$\cO^\fp_n$};
        \node (b) at (-1,1) {$\cO^\fp_n(\fl)$};
        \node (c) at (1,-1)  {$\cata^\bla$};
        \node (d) at (-1,-1){$\cata^{\bla_1;\dots ;\bla_k}$};
        \draw[->] (b) -- (a) node[above,midway]{$\ind_{\fl}^{\fgl_N}$};
        \draw[->] (d) -- (c) node[below,midway]{$\mathbb{S}^{\bla_1,\dots,\bla_k}$};
        \draw[->] (c) --(a) node[right,midway]{$\Xi$};
         \draw[->] (d) --(b) node[left,midway]{$\Xi_\fl$};
    \end{tikzpicture}
\end{equation*}
\end{prop}
\begin{proof}
We know that both functors are exact, by Proposition
\ref{standard-exact}; thus need only check this on projectives. Consider a representation of $\fl$ given by an exterior product of projectives in category $\cO$ for each of its $\mathfrak{gl}_j$-factors \[P=P_1\boxtimes\cdots\boxtimes P_k.\]  Then the induction $\ind^{\fgl_N}_\fl P$ is a quotient of the projective $P'$ corresponding to the concatenation $T$ of the tableaux $T_i$ for the $P_i$.  The kernel is the image of all maps from projectives higher than $T$ in Bruhat order through a series of transpositions which change the content of at least one of the $T_i$.

Similarly, the standardization $\mathbb{S}^{\bla_1;\dots;\bla_k}(\Xi^{-1}_\fl(P))$ is a quotient of $\Xi^{-1}(P')$; the kernel is the image of all maps from projectives that correspond to idempotents for sequences where at least one black strand has been moved left from one block to the other.  Thus, these functors agree on the level of projective objects.

Now, we must show that they agree on morphisms; that is, we must show
that the action of $\alg^{\bla_1}\otimes \cdots \otimes \alg^{\bla_k}$
induced on $\ind_{\fl}^{\fgl_N}(\Xi(\alg^{\bla_1}\otimes \cdots
\otimes \alg^{\bla_k}))$ agrees with that on
$\Xi(\mathbb{S}^{\bla_1,\dots,\bla_k}(\alg^{\bla_1}\otimes \cdots
\otimes \alg^{\bla_k}))$ under an isomorphism between these objects.
Since $\alg^{\bla_1}_{\al_1}\otimes \cdots \otimes
\alg^{\bla_k}_{\al_k}$ is the full-endomorphism algebra of $S_\bal$,
it is also the full endomorphism algebra of $\Xi(S_\bal)$.  Thus, in
fact, any isomorphism $\Xi(S_\bal)\cong \ind_{\fl}^{\fgl_N}(\Xi(\alg^{\bla_1}_{\al_1}\otimes \cdots
\otimes \alg^{\bla_k}_{\al_k}))$ induces an isomorphism of functors.  
\end{proof}

Some care is required here on the subject of gradings. Brundan and
Kleshchev's results relating category $\cO$ to Khovanov-Lauda algebras
are ungraded; they imply no connection between the usual graded lift
of $\tO^\fp$ of category $\cO$ and the graded category of modules over
$\alg^\bla$.  Luckily, the uniqueness of Koszul gradings proven in \cite[2.5.2]{BGS96}
implies that any Morita equivalence between two Koszul graded algebras
can be lifted to a graded equivalence.

There are now two proofs in the literature that these algebras are
Koszul for in the type A case.  Hu and Mathas have shown that their
quiver Schur algebra is Koszul \cite[Th. C]{HMQ}; thus, we may use the Morita
equivalence of Theorem \ref{quiver-schur} to transport this result to
$\alg^\bla$.  The author has also given a direct geometric proof in
\cite[Th. 1]{Webksln}, by directly constructing a graded isomorphism of
$\alg^\bla$ with an Ext-algebra in the Koszul dual of $\cO^\fp_n$.

\begin{prop}\label{sln-Koszul}
  When $\fg=\mathfrak{sl}_n$ and $\bla$ is a list of fundamental
  weights, the algebra $\alg^\bla_\mu$ is Koszul.
\end{prop}
\begin{cor}
The equivalence $\Xi$ has a graded lift.
\end{cor}

We note that both the action of projective functors and of the
induction functors have graded lifts which are unique up to grading
shift, and thus are determined by their action on the Grothendieck
group.  Thus the graded lifts given by the action of $\tU$ and
$\mathbb{S}$ agree, up to an easily understood shift, with those used
in other papers on graded category $\cO$ (most importantly for us,
this is used in the work of Mazorchuk-Stroppel \cite{MS} and Sussan \cite{Sussan2007} on link homologies, which we address in this paper's sequel \cite{KI-HRT}).

\subsection{The affine case}
We note that the constructions of the previous subsection generalize in an
absolutely straightforward way to the affine case by simply replacing
the results of Section 3 of \cite{BKKL} with Section 4.

We let $\hat H_d$ denote the affine Hecke algebra (not the degenerate
one we considered earlier). Fix an element  $\zeta\in \overline{\K}$,
the separable algebraic closure of $\K$ such
that \[1+\zeta+\zeta^2+\cdots+\zeta^{n-1}=0,\] and $n$ is smallest
integer for which this holds (for example, if $\K$ is characteristic
0, these means that $\zeta$ is a primitive $n$th root of unity). The
{\bf cyclotomic affine Hecke algebra} or {\bf Ariki-Koike algebra}
(introduced in \cite{AK}) for $\la$ is the quotient $$\hat
H^\la=\bigoplus_{d}\hat H_d/\langle (x_1-\zeta^{i})^{\al_i^\vee(\la)}
\rangle.$$
where we adopt the slightly strange convention that if $\zeta\in \Z$, then $\zeta^i=\zeta+i$, and otherwise it is the usual power operation.

\begin{thm}[\mbox{\cite[Main Theorem]{BKKL}}]
When $\fg\cong \widehat{\mathfrak{sl}}_n$, there is an isomorphism $\alg^\la\cong \hat H^\la$.
\end{thm}

This symmetric Frobenius algebra has a natural quasi-hereditary cover, called the {\bf cyclotomic $q$-Schur algebra}, defined by Dipper, James and Mathas \cite{DJM}.  Indecomposable projectives over this algebra are indexed by ordered $k=\sum_{i=0}^{n}\al_i^\vee(\la)$-tuples of partitions.

\begin{prop} When
$\fg=\widehat{\mathfrak{sl}}_n$, then ${\cata}^\bla$ is equivalent to the
subcategory of representations of the cyclotomic $q$-Schur algebra consisting
of objects presented by certain projective modules.

If all $\la_i$ are fundamental, then these are exactly the projectives for the multipartitions where each constituent partitions are $n$-regular.

In general, we break the multipartition into smaller ones consisting of the first $k_1=  \sum_{i=0}^{n}\al_i^\vee(\la_1)$ partitions, the next $k_2$, etc, and take the projectives for multipartitions where each of these smaller multi-partitions is $n$-Kleshchev. 
\end{prop}
\begin{proof}
By Corollary \ref{doub-cen}, $\alg^{\bla}$ is the endomorphism algebra of certain modules over $\alg^\la$, which one can see by the same arguments as Proposition \ref{Hecke-equivalence} are of the form $\hat z_\la \hat \alg^\la$ where \begin{equation*}
\hat z_\kappa=\prod_{j=1}^\ell \prod_{k=1}^{\kappa(j)}(x_k-\zeta^{\pi_j}).
\end{equation*}
These are permutation modules for the Ariki-Koike algebra, exactly those corresponding to multi-partitions where all constituent partitions have all parts size 1. Thus, in the case where all $\la$'s are fundamental, the category of modules over $\alg^\bla$ is the subcategory of representations of the cyclotomic $q$-Schur algebra generated by summands of these, and in the case where not all representations are fundamental, we must restrict these projectives further.  

The descriptions above follow from the fact that for the permutation module of the multipartition where all parts are 0 except for the last, which has all parts 1, the indecomposable projectives which appear are exactly those for $n$-Kleshchev multipartitions.
\end{proof}

Thus, our categorification can be seen a generalization of the Ariki
categorification theorem \cite{Acat}.  As mentioned in the
introduction, the author and Stroppel address the question of how to
describe the entirety of the cyclotomic $q$-Schur algebra
diagrammatically in \cite{SWschur}.

 \bibliography{../gen}

\def\cprime{$'$}
\providecommand{\bysame}{\leavevmode\hbox to3em{\hrulefill}\thinspace}
\providecommand{\MR}{\relax\ifhmode\unskip\space\fi MR }
\providecommand{\MRhref}[2]{%
  \href{http://www.ams.org/mathscinet-getitem?mr=#1}{#2}
}
\providecommand{\href}[2]{#2}
\begin{thebibliography}{GGOR03}

\bibitem[AK94]{AK}
Susumu Ariki and Kazuhiko Koike, \emph{A {H}ecke algebra of {$({\mathbb
  Z}/r{\mathbb Z})\wr{\mathfrak S}\sb n$} and construction of its irreducible
  representations}, Adv. Math. \textbf{106} (1994), no.~2, 216--243.
  \MR{MR1279219 (95h:20006)}

\bibitem[Ari96]{Acat}
Susumu Ariki, \emph{On the decomposition numbers of the {H}ecke algebra of
  {$G(m,1,n)$}}, J. Math. Kyoto Univ. \textbf{36} (1996), no.~4, 789--808.
  \MR{MR1443748 (98h:20012)}

\bibitem[AS]{AchS}
Pramod Achar and Catherina Stroppel, \emph{Completions of {G}rothendieck
  groups}, \arxiv{1105.2715}.

\bibitem[BGS96]{BGS96}
Alexander Beilinson, Victor Ginzburg, and Wolfgang Soergel, \emph{Koszul
  duality patterns in representation theory}, J. Amer. Math. Soc. \textbf{9}
  (1996), no.~2, 473--527.

\bibitem[BK07]{BeKa}
Arkady Berenstein and David Kazhdan, \emph{Geometric and unipotent crystals.
  {II}. {F}rom unipotent bicrystals to crystal bases}, Quantum groups, Contemp.
  Math., vol. 433, Amer. Math. Soc., Providence, RI, 2007, pp.~13--88.

\bibitem[BK08]{BKSch}
Jonathan Brundan and Alexander Kleshchev, \emph{Schur-{W}eyl duality for higher
  levels}, Selecta Math. (N.S.) \textbf{14} (2008), no.~1, 1--57.

\bibitem[BK09]{BKKL}
\bysame, \emph{{Blocks of cyclotomic Hecke algebras and Khovanov-Lauda
  algebras}}, Invent. Math. \textbf{178} (2009), 451--484.

\bibitem[CL]{CaLa}
Sabin Cautis and Aaron Lauda, \emph{Implicit structure in 2-representations of
  quantum groups}, \arxiv{1111.1431}.

\bibitem[CPS88]{CPS88}
E.~Cline, B.~Parshall, and L.~Scott, \emph{Finite-dimensional algebras and
  highest weight categories}, J. Reine Angew. Math. \textbf{391} (1988),
  85--99.

\bibitem[CPS96]{CPS96}
Edward Cline, Brian Parshall, and Leonard Scott, \emph{Stratifying endomorphism
  algebras}, Mem. Amer. Math. Soc. \textbf{124} (1996), no.~591, viii+119.

\bibitem[CR08]{CR04}
Joseph Chuang and Rapha{\"e}l Rouquier, \emph{Derived equivalences for
  symmetric groups and {$\mathfrak {sl}_2$}-categorification}, Ann. of Math.
  (2) \textbf{167} (2008), no.~1, 245--298.

\bibitem[DJM98]{DJM}
Richard Dipper, Gordon James, and Andrew Mathas, \emph{Cyclotomic {$q$}-{S}chur
  algebras}, Math. Z. \textbf{229} (1998), no.~3, 385--416. \MR{MR1658581
  (2000a:20033)}

\bibitem[GGOR03]{GGOR}
Victor Ginzburg, Nicolas Guay, Eric Opdam, and Rapha{\"e}l Rouquier, \emph{On
  the category {$\mathcal O$} for rational {C}herednik algebras}, Invent. Math.
  \textbf{154} (2003), no.~3, 617--651.

\bibitem[HL10]{HL}
Alexander~E. Hoffnung and Aaron~D. Lauda, \emph{Nilpotency in type {$A$}
  cyclotomic quotients}, J. Algebraic Combin. \textbf{32} (2010), no.~4,
  533--555. \MR{2728758 (2011m:20009)}

\bibitem[HM]{HMQ}
J.~Hu and A.~Mathas, \emph{Quiver {S}chur algebras {I}: linear quivers},
  \arxiv{1110.1699}.

\bibitem[HM10]{HM}
\bysame, \emph{Graded cellular bases for the cyclotomic
  {K}hovanov-{L}auda-{R}ouquier algebras of type {$A$}}, Adv. Math.
  \textbf{225} (2010), no.~2, 598--642.

\bibitem[HMM]{HMM}
David Hill, George Melvin, and Damien Mondragon, \emph{Representations of
  quiver {H}ecke algebras via {L}yndon bases}, \arxiv{0912.2067}.

\bibitem[HS]{HS}
David Hill and Joshua Sussan, \emph{The {K}hovanov-{L}auda 2-category and
  categorifications of a level two quantum $\mathfrak{sl}(n)$ representation},
  \arxiv{0910.2496}.

\bibitem[KK]{KK}
Seok-Jin Kang and Masaki Kashiwara, \emph{Categorification of highest weight
  modules via {K}hovanov-{L}auda-{R}ouquier algebras}, \arxiv{1102.4677}.

\bibitem[KL09]{KLI}
Mikhail Khovanov and Aaron~D. Lauda, \emph{A diagrammatic approach to
  categorification of quantum groups. {I}}, Represent. Theory \textbf{13}
  (2009), 309--347. \MR{2525917 (2010i:17023)}

\bibitem[KL10]{KLIII}
\bysame, \emph{A categorification of quantum {${\rm sl}(n)$}}, Quantum Topol.
  \textbf{1} (2010), no.~1, 1--92. \MR{2628852 (2011g:17028)}

\bibitem[KL11]{KLII}
\bysame, \emph{A diagrammatic approach to categorification of quantum groups
  {II}}, Trans. Amer. Math. Soc. \textbf{363} (2011), no.~5, 2685--2700.
  \MR{2763732 (2012a:17021)}

\bibitem[KR11]{KlRa}
Alexander Kleshchev and Arun Ram, \emph{Representations of
  {K}hovanov-{L}auda-{R}ouquier algebras and combinatorics of {L}yndon words},
  Math. Ann. \textbf{349} (2011), no.~4, 943--975. \MR{2777040 (2012b:16078)}

\bibitem[Lau]{LauSL2}
Aaron~D. Lauda, \emph{Categorified quantum $\mathfrak{sl}(2)$ and equivariant
  cohomology of iterated flag varieties}, \arxiv{0803.3848}.

\bibitem[Lau10]{LauSL21}
Aaron~D. Lauda, \emph{A categorification of quantum {${\rm sl}(2)$}}, Adv.
  Math. \textbf{225} (2010), no.~6, 3327--3424. \MR{2729010 (2012b:17036)}

\bibitem[Lus93]{Lusbook}
George Lusztig, \emph{Introduction to quantum groups}, Progress in Mathematics,
  vol. 110, Birkh\"auser Boston Inc., Boston, MA, 1993.

\bibitem[LV11]{LV}
Aaron~D. Lauda and Monica Vazirani, \emph{Crystals from categorified quantum
  groups}, Adv. Math. \textbf{228} (2011), no.~2, 803--861. \MR{2822211}

\bibitem[LW]{LoWe}
Ivan Losev and Ben Webster, \emph{{On uniqueness of tensor products of
  irreducible categorifications}}, \arxiv{1303.1336}.

\bibitem[MS08]{MS}
Volodymyr Mazorchuk and Catharina Stroppel, \emph{Projective-injective modules,
  {S}erre functors and symmetric algebras}, J. Reine Angew. Math. \textbf{616}
  (2008), 131--165.

\bibitem[Rou]{Rou2KM}
Raphael Rouquier, \emph{2-{K}ac-{M}oody algebras}, \arxiv{0812.5023}.

\bibitem[Soe90]{Soe90}
Wolfgang Soergel, \emph{Kategorie {$\mathcal{ O}$}, perverse {G}arben und
  {M}oduln \"uber den {K}oinvarianten zur {W}eylgruppe}, J. Amer. Math. Soc.
  \textbf{3} (1990), no.~2, 421--445. \MR{MR1029692 (91e:17007)}

\bibitem[Soe92]{Soe92}
\bysame, \emph{The combinatorics of {H}arish-{C}handra bimodules}, J. Reine
  Angew. Math. \textbf{429} (1992), 49--74. \MR{MR1173115 (94b:17011)}

\bibitem[Str]{Str06b}
Catharina Stroppel, \emph{{Perverse sheaves on Grassmannians, Springer fibres
  and Khovanov homology}}, \arxiv{math.RT/0608234}.

\bibitem[Str03]{Str03}
\bysame, \emph{Category {$\mathcal O$}: quivers and endomorphism rings of
  projectives}, Represent. Theory \textbf{7} (2003), 322--345 (electronic).

\bibitem[Sus07]{Sussan2007}
Joshua Sussan, \emph{Category $\mathcal {O}$ and $\mathfrak{sl}(k)$ link
  invariants}, 2007, \arxiv{0701045}.

\bibitem[SW]{SWschur}
Catharina Stroppel and Ben Webster, \emph{Quiver {S}chur algebras and
  $q$-{F}ock space}, \arxiv{1110.1115}.

\bibitem[TW]{TW}
Peter Tingley and Ben Webster, \emph{{M}irkovi\'c-{V}ilonen polytopes and
  {K}hovanov-{L}auda-{R}ouquier algebras}, \arxiv{1210.6921}.

\bibitem[Weba]{WebCB}
Ben Webster, \emph{Canonical bases and higher representation theory},
  \arxiv{1209.0051}.

\bibitem[Webb]{Webcatq}
\bysame, \emph{A categorical action on quantized quiver varieties},
  \arxiv{1208.5957}.

\bibitem[Webc]{KI-HRT}
\bysame, \emph{Knot invariants and higher representation theory {II}:
  categorification of quantum knot invariants}, \arxiv{1005.4559}.

\bibitem[Webd]{Webksln}
\bysame, \emph{Tensor product algebras in type {A} are {K}oszul},
  \arxiv{1209.2777}.

\bibitem[Webe]{WebwKLR}
\bysame, \emph{Weighted {K}hovanov-{L}auda-{R}ouquier algebras},
  {\arxiv{1209.2463}}.

\bibitem[Zhe]{Zheng2008}
Hao Zheng, \emph{Categorification of integrable representations of quantum
  groups}, \arxiv{0803.3668}.

\end{thebibliography}
\bibliographystyle{amsalpha}
\end{document}